\documentclass{amsart}
\usepackage{amsmath,amssymb,amsthm,geometry,graphics,color,mdwlist,mathrsfs}
\usepackage[all]{xy}

\DeclareMathOperator{\Hom}{Hom}

\DeclareMathOperator{\RHom}{\mathrm{R}\!\Hom}
\DeclareMathOperator{\cHom}{\mathscr{H}\hspace{-2pt}{\it om}}
\DeclareMathOperator{\End}{End}

\DeclareMathOperator{\REnd}{\mathrm{R}\!\End}
\DeclareMathOperator{\cEnd}{\mathscr{E}\hspace{-2pt}{\it nd}}
\def\lotimes{\otimes^\mathrm{L}}

\DeclareMathOperator{\rT}{\mathrm{T}}

\DeclareMathOperator{\cocone}{cocone}

\DeclareMathOperator{\HH}{HH}

\DeclareMathOperator{\md}{mod}
\renewcommand{\mod}{\md}
\DeclareMathOperator{\Mod}{Mod}

\DeclareMathOperator{\fl}{fl}

\DeclareMathOperator{\coh}{coh}

\DeclareMathOperator{\thick}{thick}
\DeclareMathOperator{\per}{per}

\DeclareMathOperator{\Loc}{Loc}

\DeclareMathOperator{\add}{add}

\def\seg{\#}

\def\rD{\mathrm{D}}

\def\rC{\mathrm{C}}

\def\Db{\mathrm{D^b}}

\newcommand\recollement[3]{\xymatrix{{#1}\ar[r]&{#2}\ar[r]\ar@/_9pt/[l]\ar@/^9pt/[l]&{#3}\ar@/_9pt/[l]\ar@/^9pt/[l] }}
\newcommand\recollementwithmaps[9]{\xymatrix{{#1}\ar[r]|-{#8}&{#2}\ar[r]|-{#5}\ar@/_10pt/[l]_-{#7}\ar@/^10pt/[l]^-{#9}&{#3}\ar@/_10pt/[l]_-{#4}\ar@/^10pt/[l]^-{#6} }}
\def\A{\mathscr{A}}
\def\B{\mathscr{B}}
\def\C{\mathscr{C}}

\renewcommand{\O}{\mathscr{O}}
\def\P{\mathscr{P}}
\def\cQ{\mathscr{Q}}
\renewcommand\S{\mathscr{S}}
\def\T{\mathscr{T}}
\def\U{\mathscr{U}}

\def\X{\mathscr{X}}
\def\Y{\mathscr{Y}}

\def\rL{\mathrm{L}}

\def\G{\Gamma}
\def\Ga{\Gamma}

\def\e{\varepsilon}

\def\La{\Lambda}

\def\Si{\Sigma}

\def\al{\alpha}

\def\om{\omega}

\def\Z{\mathbb{Z}}

\def\AA{\mathbb{A}}
\def\PP{\mathbb{P}}

\def\op{\mathrm{op}}
\def\dg{\mathrm{dg}}

\def\rsimeq{\rotatebox{-90}{$\simeq$}}
\def\xsimeq{\xrightarrow{\simeq}}
\def\ysimeq{\xleftarrow{\simeq}}

\def\vsubset{\rotatebox{90}{$\subset$}}

\def\vin{\rotatebox{90}{$\in$}}

\newtheorem{Thm}{Theorem}[section]
\newtheorem{Lem}[Thm]{Lemma}
\newtheorem{Prop}[Thm]{Proposition}
\newtheorem{Cor}[Thm]{Corollary}

\newtheorem{Prop-Def}[Thm]{Proposition-Definition}
\newtheorem{Thm-Def}[Thm]{Theorem-Definition}

\theoremstyle{definition}
\newtheorem{Def}[Thm]{Definition}
\newtheorem{Ex}[Thm]{Example}

\newtheorem{Qs}[Thm]{Question}
\newtheorem{Setup}[Thm]{Setting}

\theoremstyle{remark}
\newtheorem{Rem}[Thm]{Remark}

\newcounter{step}

\def\disoplus{\displaystyle\bigoplus}

\makeatletter

\@addtoreset{equation}{section}
\makeatother

\title{Calabi--Yau completions for roots of dualizing dg bimodules}
\author{Norihiro Hanihara}
\thanks{This work is supported by JSPS KAKENHI Grant Number JP22KJ0737}
\address{Faculty of Mathematics, Kyushu University, 744 Motooka, Nishi-ku, Fukuoka, 819-0395, Japan}
\email{hanihara@math.kyushu-u.ac.jp}
\date{}
\def\seg{\#}
\def\e{\varepsilon}

\def\ev{\mathrm{ev}}

\begin{document}
\begin{abstract}
Roots of shifted Serre functors appear naturally in representation theory and algebraic geometry. We give an analogue of Keller's Calabi-Yau completion for roots of shifted inverse dualizing bimodules over dg categories.
Given a positive integer $a$, we introduce the notion of the $a$-th root pair on smooth dg categories and define its Calabi-Yau completion. We prove that the Calabi-Yau completion has the Calabi-Yau property when the $a$-th root pair has certain invariance under an action of the cyclic group of order $a$, and observe that it is only twisted Calabi-Yau in general.
Next, we establish a bijection between Adams graded Calabi-Yau dg categories of Gorenstein parameter $a$ and $a$-th root pairs on a dg category with the cyclic invariance.
Applying this bijection, we prove that a certain operation on dg categories, called the $a$-Segre product, allows us to reproduce Calabi-Yau dg categories.
Furthermore, we discuss the cluster category of these Calabi-Yau completions, and prove that it is a $\Z/a\Z$-quotient of the usual cluster category, which thereby establishes the $a$-th root versions of cluster categories.
In the appendix, we give a generalization of Beilinson's theorem on tilting bundles on projective spaces to the setting of Adams graded dg categories.
\end{abstract}

\maketitle
\setcounter{tocdepth}{1}
\tableofcontents
\section{Introduction}
Serre duality on triangulated categories is a fundamental structure appearing in various areas of mathematics including algebraic geometry and representation theory, especially in the study of derived categories and their variants. 
If there exists a functor which realizes the Serre duality, called {\it Serre functor}, it is unique up to a natural isomorphism, thus it is an invariant of a triangulated category. 

Often, a description of the Serre functor $S$ of a triangulated category shows that there is a natural root of a {\it shifted Serre functor} $S_n:=S\circ [-n]$, which is the subject of this paper.
For example, the derived category $\Db(\coh\PP^n)$ of the projective space has the Serre functor $S=(-n-1)[n]$, so the shifted Serre functor $S_n=(-n-1)$ has an $(n+1)$-st root $(-1)$.
Moreover, by the uniqueness of Serre functors, the existence of a root of a shifted Serre functor translates under triangle equivalences. 
In our example, Beilinson's equivalence $\Db(\coh\PP^n)\simeq\Db(\mod A)$ for a finite dimensional algebra $A$ shows that the category $\Db(\mod A)$ also has an $(n+1)$-st root of $S_n$, which gives some non-trivial implications on the finite dimensional algebra $A$.
Such roots of shifted Serre functors arise not only from projective varieties, but also from representation theory of finite dimensional algebras or commutative rings, for example as certain reflection functors in quiver representations \cite{KMV,ha4} and tilting theory for singularity categories \cite{haI}.

The aim of this paper is to study the relationship between such roots of shifted Serre functors and Calabi-Yau algebras. We give an analogue of Keller's {Calabi-Yau completion} \cite{Ke11} for these roots, give their basic properties, including some results on generalized Segre products, and their application to cluster categories.

\subsection{Root pairs and their Calabi-Yau completions}
Algebraically, (the inverse of) the Serre functor is formalized, at the level of differential graded (dg) enhancements, as the {\it (inverse) dualizing bimodule}. Let $A$ be a dg algebra (rather than a category for simplicity) over a field $k$ enhancing a given triangulated category $\T$, that is, the perfect derived category $\per A$ of $A$ is triangle equivalent to $\T$. Then the {\it inverse dualizing bimodule} of $A$ is
\[ \RHom_{A^e}(A,A^e)\in\rD(A^e), \]
where $A^e=A^\op\otimes_k A$ is the enveloping algebra over $k$. This bimodule exhibits the Serre duality for certain pairs of objects in the derived category (see \cite{Ke08}).

Next, recall that a dg algebra $A$ is {\it smooth} if it is perfect as a bimodule, that is, $A\in\per A^e$. Then the {\it $n$-Calabi-Yau completion} \cite{Ke11} of a smooth dg algebra $A$ is defined for each integer $n$ as the derived tensor algebra
\[ \Pi_{n}(A):=\rT^\rL_A(\RHom_{A^e}(A,A^e)[n-1]), \]
more precisely, the tensor algebra $\rT_A\!\theta$ of a cofibrant resolution $\theta\to\RHom_{A^e}(A,A^e)[n-1]$ over $A^e$. The morphism $A\to\Pi_{n}(A)$ is an algebraic counterpart of the canonical bundle $\om_X\to X$ over a variety $X$.
It is a fundamental result due to Keller that the Calabi-Yau completion $\Pi=\Pi_{n}(A)$ is indeed $n$-Calabi-Yau, that is, $\Pi$ is smooth and there is an isomorphism
\[ \RHom_{\Pi^e}(\Pi,\Pi)[n]\simeq\Pi \text{ in } \rD(\Pi^e). \]
In fact, the above isomorphism comes from a class in the (negative) cyclic homology of $\Pi$, called the (left) Calabi-Yau structure \cite{Ke11+,BD19}.

Such Calabi-Yau (dg) algebras and their variants have been of great interest in representation theory, commutive or non-commutative algebraic geometry, and so on. They play an essential role in the categorification of cluster algebras \cite{BMRRT,Am09}. At the same time, they are higher dimensional Auslander algebras \cite{Iy07b} as well as non-commutative crepant resolutions \cite{VdB04,IW14}.
We refer, for example, to \cite{Am09,AS,Br,Gi,ha6,IQ,Iy07b,IW14,Ke11,VdB04,Wu23,Ye16} for such studies.

\bigskip

Our first goal is to generalize the above construction and result to the setting of roots of shifted Serre functors.
For this we start with an algebraic formulation of such roots.
Recall that a dg bimodule $X$ over a dg algebra $A$ is {\it invertible} if there is a dg bimodule $Y$ such that $X\lotimes_AY\simeq A$ and $Y\lotimes_AX\simeq A$ in $\rD(A^e)$. For $X\in\rD(A^e)$ we denote by $X^a$ the $a$-fold derived tensor power over $A$.
To construct a Calabi-Yau dg algebra from a root of a shifted Serre functor, we need not just a root of a shifted inverse dualizing bimodule, but also a projective module. This leads to the following definition.
\begin{Def}
Let $A$ be a smooth dg algebra, let $a$ be a positive integer, and let $d$ be an arbitrary integer.
\begin{enumerate}
\item An {\it $a$-th root} of the $d$-shifted inverse dualizing bimodule is an invertible bimodule $U$ over $A$ satisfying
\[ U^a\simeq\RHom_{A^e}(A,A^e)[d] \]
in $\rD(A^e)$.
\item An {\it $a$-th root pair} of the $d$-shifted inverse dualizing bimodule is a pair $(U,P)$ consisting of $U$ as above and $P\in\add A$ satifying the following.
\[ \per A=\thick (P\lotimes_AU^{a-1})\perp\cdots\perp\thick (P\lotimes_AU)\perp\thick P, \]
where $\perp$ means the stable $t$-structure (or the semi-orthogonal decomposition).
\item Let $(U,P)$ be an $a$-th root pair of $\RHom_{A^e}(A,A^e)[d]$ and let $e\in A$ be the idempotent corresponding to $A$. The {\it $(d+1)$-Calabi-Yau completion} of the root pair $(U,P)$ is the idempotent truncation of the derived tensor algebra
\[ \Pi_{d+1}^{(1/a)}(A; U,P)=\Pi_{d+1}^{(1/a)}(A):=e(\rT^\rL_A\!U)e. \]
\end{enumerate}
\end{Def}
For $a=1$, we must have $U=\RHom_{A^e}(A,A^e)[d]$ and $\add P=\add A$, thus our Calabi-Yau completion is nothing but the ordinary $(d+1)$-Calabi-Yau completion. We shall mostly use the notation $\Pi_{d+1}^{(1/a)}(A)$ for the Calabi-Yau completion of an $a$-th root pair on a dg algebra $A$. Note, however, that the Calabi-Yau completion depends on the root pair $(U,P)$, not just on $A$, so this notation is ambiguous.

A motivating example of a root pair comes from Beilinson's equivalence $\Db(\coh\PP^d)\simeq\Db(\mod A)$. As mentioned above, the $d$-shifted inverse Serre functor $\RHom_{A^e}(A,A^e)[d]$ has a $(d+1)$-st root corresponding the twist $(1)$ on $\Db(\coh\PP^d)$. Moreover, the tilting object $\bigoplus_{i=0}^d\O(i)$ yields a stable $t$-structure, hence a $(d+1)$-st root pair on $A$. In this case, the Calabi-Yau completion will be the polynomial ring in $d+1$ variables.
Geometrically, the diagram $e(\rT^\rL_A\!U)e\rightarrow \rT^\rL_A\!U\leftarrow A$ corresponds to $\AA^{d+1}\leftarrow\O_{\PP^d}(-1)\rightarrow\PP^d$, where the rightgoing map is the total space of the $(d+1)$-st root $\O(-1)$ of the canonical bundle over $\PP^d$, and  the leftgoing map is the blow-up of $\AA^{d+1}$ at the origin.

Note that the above example indicates that we need to take the idempotent subalgebra of the tensor algebra $\rT^\rL_A\!U$ in order for the algebra to be Calabi-Yau. In fact, we will see in \ref{fakeCY} how far the tensor algebra is from being Calabi-Yau.

We also remark that a semi-orthogonal decomposition as in (2) is also called a (rectangular) {Lefschetz decomposition} \cite{Ku07}; we refer to \cite{Ku19} for examples arising from algebraic geometry.

\medskip
A natural expectation is that our Calabi-Yau completion $\Pi_{d+1}^{(1/a)}(A)$ of a root pair $(U,P)$ is Calabi-Yau. We observe, however, that this is {\it not} necessarily the case; in fact, $\Pi_{d+1}^{(1/a)}(A)$ is only {twisted} Calabi-Yau in general, see \ref{example} for a specific example. This indicates that we need to take a careful choice of the $a$-th root $U$. We introduce the notion of {\it cyclic invariance} of a root pair (\ref{cyc}) which is a subtle condition for the Calabi-Yau completion to be indeed Calabi-Yau. The first main result of this paper is the following.
\begin{Thm}[=\ref{CYstr}]
Let $A$ be a smooth dg algebra and let $(U,P)$ be an $a$-th root pair of the $d$-shifted inverse dualizing bimodule which is cyclically invariant. Then the Calabi-Yau completion $\Pi_{d+1}^{(1/a)}(A)$ is smooth and $(d+1)$-Calabi-Yau.
\end{Thm}

Note that by \cite{VdB15}, any Calabi-Yau algebra (with certain compleness) is obtained as an (ordinary) Calabi-Yau completion, up to a deformation. 
Nevertheless, our description of Calabi-Yau algebras has advantages of, for example, being able to reproduce themselves, and having nice descriptions of the cluster categories, and we are totally ignorant of the effects of deformations on these structures.


\subsection{Correspondences}
Notice that the Calabi-Yau dg algebra $\Pi_{d+1}^{(1/a)}(A)=e(\rT^\rL_A\!U)e$ constructed above has a natural grading induced from the tensor grading on $\rT^\rL_A\!U$. Such an additional grading to the cohomological grading of dg algebras is sometimes called an {\it Adams grading}. It has played a prominent role in the structure theory of various kinds of (dg) algebras and categories, e.g.\!  \cite{MM,Ke11,IQ,haI,FKQ}.
One of the landmarks in this direction is a bijection between the following two classes of algebras (see \ref{mm}):
\begin{itemize}
\item positively graded $(d+1)$-Calabi-Yau (non-dg) algebras of Goresntein parameter $1$ (see \ref{GPDef}), and
\item finite dimensional $d$-representation infinite algebras in the sense of \cite{HIO} (see \ref{dRIDef}).
\end{itemize}
We extend the above correspondence to higher Gorenstein parameters in terms of root pairs. Furthermore, as a generalization of \cite{HaIO}, we state the result in the setting of dg algebras (rather than ordinary algebras). We refer to \ref{qe} for the notion of quasi-equivalence of root pairs.
\begin{Thm}[=\ref{MM}]\label{iMM}
Let $d$ be an integer and let $a$ be a positive integer. There exists a bijection between the following.
\begin{enumerate}
\renewcommand{\labelenumi}{(\roman{enumi})}
\renewcommand{\theenumi}{\roman{enumi}}
\item The set of graded quasi-equivalence classes of Adams positively graded $(d+1)$-Calabi-Yau algebras $\Pi$ of Gorenstein parameter $a$, that is, there is an isomorphism
\[ \RHom_{\Pi^e}(\Pi,\Pi^e)[d+1]\simeq\Pi(a) \text{ in } \rD^\Z(\Pi^e), \]
and such that $\Pi_0\in\per\Pi$ and $\Pi_0\in\per\Pi^\op$.
\item The set of quasi-equivalence classes of cyclically invariant root pairs $(U,P)$ on smooth dg algebras $A$.
\end{enumerate}
The correspondences are given as follows.
\begin{itemize}
	\item From {\rm(\ref{CYdgalg})} to {\rm(\ref{smdg})}: Take $\Pi$ to $A=\begin{pmatrix}\Pi_0&0&\cdots&0\\ \Pi_1&\Pi_0&\cdots&0\\ \vdots&\vdots&\ddots&\vdots\\ \Pi_{a-1}&\Pi_{a-2}&\cdots&\Pi_0\end{pmatrix}$, $U=\begin{pmatrix}\Pi_1 &\Pi_0&\cdots& 0\\\vdots&\vdots&\ddots&\vdots \\ \Pi_{a-1}&\Pi_{a-2}&\cdots&\Pi_0\\ \Pi_a&\Pi_{a-1}&\cdots&\Pi_1\end{pmatrix}$, and let $P=eA$ for the idempotent $e$ at the upper left entry of the matrix.
	\item From {\rm(\ref{smdg})} to {\rm(\ref{CYalg})}: Take the $(d+1)$-Calabi-Yau completion.
\end{itemize}
\end{Thm}
As a special case, when the Calabi-Yau dg algebras in (\ref{CYdgalg}) are concentrated in (cohomological) degree $0$, then we get $d$-representation infinite algebras with root pairs as the ones in (\ref{smdg}), see \ref{fdMM}. Also, we refer to \cite{HaIO} which is one of motivations of \ref{iMM}, for a special case $a=1$ into another direction.

\medskip
We also have the following characterization of (non-Calabi-Yau) dg algebras which appear as the tensor algebra $\Si:=\rT^\rL_A\!U$ of a bimodule $U$ giving an $a$-th root of $\RHom_{A^e}(A,A^e)[d]$. For such $\Si=\bigoplus_{i\geq0}U^i$, we denote by $\Si_{\geq l}:=\bigoplus_{i\geq l}U^i$ the two-sided ideal of $\Si$, where $U^i$ means the $i$-fold derived tensor power of $U$ over $A$.
\begin{Thm}[=\ref{fakeMM}]
Let $d$ be an integer and let $a$ be a positive integer. There exists a bijection between the following.
\begin{enumerate}
\renewcommand{\labelenumi}{(\roman{enumi})}
\renewcommand{\theenumi}{\roman{enumi}}
\item The set of graded quasi-equivalence classes of Adams positively graded smooth dg algebras $\Si$ satisfying $\RHom_{\Si^e}(\Si,\Si^e)[d+1]\simeq\Si_{\geq a-1}(a)$ in $\rD^\Z(\Si^e)$.
\item The set of quasi-equivalence classes of pairs $(A,U)$ where $A$ is a smooth dg algebra and $U$ is a cyclically invariant $a$-th root of $\RHom_{A^e}(A,A^e)[d]$.
\end{enumerate}
The correspondence from {\rm(i)} to {\rm(ii)} is given by $\Si\mapsto(\Si_0,\Si_1)$, and from {\rm(ii)} to {\rm(i)} by taking the derived tensor algebra.
\end{Thm}

The description of the inverse dualizing bimodule of $\Si$ as in (i) shows that although $\Si$ is not Calabi-Yau, it is close enough to be so, which is an important aspect in the description of the cluster category as in Section \ref{irC} below.
We also believe that certain non-Calabi-Yau dg algebras like $\Si$ will be useful for relating different cluster tilting objects.

\bigskip
We now turn to applications of our construction and correspondences, which we describe in the following subsections.
\subsection{$a$-Segre products}
While it is not immediately clear how to reproduce Calabi-Yau algebras from given ones, there is an easy way to do so for root pairs, namely, tensor products: Let $(U,P)$ be an $a$-th root pair of $\RHom_{A^e}(A,A^e)[d]$ on a smooth dg algebra $A$, and let $(V,Q)$ be an $a$-th root pair of $\RHom_{B^e}(B,B^e)[e]$ on another smooth dg algebra $B$. Note that we allow $d$ and $e$ to be different, but require $a$ to be the same. Then it is easy to see (\ref{UotimesV}) that $(U\otimes V,P\otimes B)$ is an $a$-th root pair of $\RHom_{C^e}(C,C^e)[d+e]$ on $C:=A\otimes B$. It is cyclically invariant whenever $(U,P)$ and $(V,Q)$ are.
Under the correspondence \ref{iMM}, the tensor product translates into the operation called {\it $a$-Segre product} (see Section \ref{a-Segre}, also \cite{ha6}) on graded dg algebras in the following sense. For Adams graded dg algebras $R=\bigoplus_{i\in\Z}R_i$ and $S=\bigoplus_{i\in\Z}S_i$, their $a$-Segre product is an Adams graded dg algebra
\[ \bigoplus_{i\in\Z}R_i\otimes\begin{pmatrix}	S_{i} &S_{i-1}&\cdots&S_{i-a+1}\\ S_{i+1}& S_i&\cdots&S_{i-a+2} \\ \vdots&\vdots&\ddots&\vdots \\ S_{i+a-1}&S_{i+a}&\cdots&S_{i}\end{pmatrix}. \]
When $a=1$ this is nothing but the classical Segre product $\bigoplus_{i\in\Z} R_i\otimes S_i$. Also, when $R=k[x]$ with (Adams) $\deg x=a$, then this is the {\it $a$-th quasi-Veronese algebra} of $S$.

Thanks to the correspondence in \ref{iMM}, the simple observation on the tensor product leads to the following reproduction theorem for Calabi-Yau dg algebras.
\begin{Thm}[{=\ref{CYSegre}}]\label{iCYSegre}
	Let $\Pi$ {\rm(}resp. $\Ga${\rm)} be an Adams positively graded $(d+1)$-Calabi-Yau dg algebra {\rm(}resp. $(e+1)$-Calabi-Yau dg category{\rm)} of Gorenstein parameter $a>0$ such that $\Pi_0$ {\rm(}resp. $\Ga_0${\rm)} is perfect over $\Pi$ {\rm(}resp. over $\Ga${\rm)} on each side. Then their $a$-Segre product is a positively graded $(d+e+1)$-Calabi-Yau dg algebra of Gorenstein parameter $a$.
\end{Thm}
For Calabi-Yau algebras (concentrated in cohomological degree $0$) which are module-finite, this is implicitly given in \cite[1.6]{ha6}, and the above result is its vast generalization. Also, as a very special case where $\Pi=k[x]$ is just a polynomial ring, we deduce that the $a$-th quasi-Veronese algebra of a Calabi-Yau dg algebra of Gorenstein parameter $a$ is still Calabi-Yau (\ref{qv}). 
Note that we need a root {\it pair} only on one of the factors, that is, the projective module $Q$ does not play a role in the construction of a root pair $(U\otimes V,P\otimes B)$ on $C$, which is enough to produce a Calabi-Yau dg algebra. We refer to \ref{Pi and Si} for the version of \ref{iCYSegre} in this setting.

\subsection{Folded cluster categories and strict root pairs}\label{irC}
Our next application is to the theory of cluster categories. Let $\Pi$ be a $(d+1)$-Calabi-Yau dg algebra which is connective (i.e.\! $H^i\Pi=0$ for $i>0$) and such that $H^0\Pi$ is finite dimensional. Then the {\it cluster category} \cite{Am09} of $\Pi$ is the Verdier quotient
\[ \rC(\Pi):=\per\Pi/\Db(\Pi) \]
of the perfect derived category by the full subcategory formed by dg modules of finite dimensional total cohomology. The fundamental theorem due to Amiot is that $\rC(\Pi)$ is a $d$-Calabi-Yau triangulated category, and contains a $d$-cluster tilting object $\Pi\in\rC(\Pi)$ \cite{Am09,Guo}, which are essential properties in the categorification of cluster algebras \cite{BMRRT}.
Note that the above definition of $\rC(\Pi)$ makes sense for arbitrary smooth dg algebras (not necessarily Calabi-Yau), and we will denote the above Verdier quotient by $\rC(\Ga)$ for any smooth dg algebra $\Ga$.

It is an important problem to understand the structure of the cluster category $\rC(\Pi)$. When $\Pi$ is the $(d+1)$-Calabi-Yau completion $\Pi_{d+1}(A)$ of a finite dimensional algebra $A$ (with some finiteness), then $\rC(\Pi_{d+1}(A))$ is the triangulated hull \cite{Ke05} of the orbit category $\Db(\mod A)/S_d$, and the image of $A\in\Db(\mod A)$ is the $d$-cluster tilting object $\Pi\in\rC(\Pi)$. This category is called the {\it $d$-cluster category} of $A$, denoted by $\rC_d(A)$. For general Calabi-Yau dg algebras, the structure of $\rC(\Pi)$ is quite mysterious.

We show that for our Calabi-Yau completion $\Pi=\Pi_{d+1}^{(1/a)}(A)=e(\rT^\rL_A\!U)e$ of a root pair $(U,eA)$, there is a nice description of its cluster category. In fact, it is equivalent to the cluster category of the tensor algebra $\Si=\rT^\rL_A\!U$, which shows that it is the triangulated hull of the orbit category $\Db(\mod A)/-\lotimes_AU$. Since $U$ is an $a$-th root of $S_d$, we see that this can be seen as a $\Z/a\Z$-quotient of $\rC_d(A)$, and we therefore call it the {\it $a$-folded $d$-cluster category} of $A$ and denote it by $\rC_d^{(1/a)}(A)$ \cite{ha3,haI}. Moreover, the cluster tilting object $\Pi\in\rC(\Pi)$ is exactly the image of $eA\in\Db(\mod A)$, as in the diagram below.
\[ \xymatrix@R=2mm{
	\Db(\mod A)\ar[r]&\Db(\mod A)/-\lotimes_AU\ar@{^(->}[r]&\rC(\Si)\ar[r]^-\simeq&\rC(\Pi)=:\rC_d^{(1/a)}(A)\\
	\qquad eA\qquad \ar@{|->}[rr]\ar@{}[u]|-\vin&&\quad e\Si\quad \ar@{|->}[r]\ar@{}[u]|-\vin&\quad \Pi\quad\ar@{}[u]|-\vin } \]

Our results are summarized as follows. When $a=1$ this is nothing but Amiot's fundamental theorem \cite{Am09}, and in this way our result establishes an $a$-th root analogue of the cluster category. We say that a dg algebra is {\it proper} if it has finite dimensional total cohomology.
\begin{Thm}[=\ref{rc}]
Let $A$ be a smooth and proper dg algebra, and $(U,P)$ a cyclically invariant $a$-th root pair of $\RHom_{A^e}(A,A^e)[d]$. Let $\Si=\rT^\rL_A\!U$ be the derived tensor algebra and $\Pi=\Pi_{d+1}^{(1/a)}(A)$ the $(d+1)$-Calabi-Yau completion. We assume that $\Pi$ is connective and $H^0\Pi$ is finite dimensional.
\begin{enumerate}
	\item The cluster category $\rC(\Si)$ is the triangulated hull of the orbit category $\per A/-\lotimes_AU$.
	\item There is a triangle equivalence $\rC(\Si)\simeq\rC(\Pi)$.
	\suspend{enumerate}
	Therefore, the cluster category $\rC(\Pi)$ is equivalent to the canonical triangulated hull of $\per A/-\lotimes_A U$.
	\resume{enumerate}
	\item The image of $P\in\per A$ under $\per A\to\per A/-\lotimes_AU\hookrightarrow\rC(\Pi)$ is a $d$-cluster tilting object.
\end{enumerate}
\end{Thm}
It is important that the dg algebras $\Pi$ and $\Si$ have the equivalent cluster categories, as stated in (2). To prove this we first discuss the (Adams) graded version of the cluster category, which we call the {\it graded cluster category}. For such categories, we give a Beilinson-type theorem (\ref{B}). This shows that the graded cluster categories of $\Pi$ and $\Si$ are equivalent and have the compatible (Adams) degree shift functors, which implies that the ungraded cluster categories are also equivalent.

Although the above result is valid for dg algebras $A$, it is important that $A$ is an ordinary (finite dimensional) algebra in order to understand the relevant categories as well as the root $U$ of the shifted Serre functor. Lead by this, we introduce and study {\it strictness} (\ref{strict}) for root pairs on finite dimensional algebras. This allows us to understand the root of $\RHom_{A^e}(A,A^e)[d]$ as a certain kind of APR tilting and gives an accessible description of the cluster category, see \ref{twD_4} for an example. 
We show under a mild assumption that strict root pairs on finite dimensional algebras also reproduce under tensor products, up to a derived equivalence (\ref{stricttensor}) which is a counterpart of \ref{iCYSegre}.

Finally, we give examples of root pairs from Dynkin quivers. The classification of $a$-th roots of (the inverse of) the Auslander-Reiten translation $\RHom_{kQ^e}(kQ,kQ^e)[1]$ for connected Dynkin quivers $Q$ (see \ref{Dyn}) asserts that the only possibility is $a=2$ and $Q$ is of even type $A$. We are thus lead to study square roots of type $A_{2n}$ quivers. We give explicit descriptions of the bimodule $U$ which consists a root pair $(U,P)$, the dg algebra $\Si=\rT^\rL_{kQ}\!U$, the Calabi-Yau completion $\Pi_m^{(1/2)}(kQ)$, and the cluster category $\rC_{m-1}^{(1/2)}(kQ)$.
To this end, we obtain the following family of dg path algebras as Calabi-Yau completions of these root pairs. We refer to Section \ref{Dynkin} for the definition of the orientation of the type $A_{2n}$-quiver $Q$, and the root pair $(U,P)$ on $kQ$.
\begin{Thm}[{=\ref{odd}, \ref{dgq}}]
\def\ard{\ar@<-2pt>[d]_-v}
\def\aru{\ar@<-2pt>[u]_-u}
\def\arr{\ar@<2pt>[r]}
\def\arl{\ar@<2pt>[l]}
\def\lo{\ar@(dl,dr)[]}
The following dg path algebra presents the Calabi-Yau completion $\Pi=\Pi^{(1/2)}_{2d+2}(kQ)$, where the vertices numbered as $n, n-1,\ldots$ for odd positions, and as $1,2,\ldots$ for even positions, and $c=a$ if $n$ is odd, and $c=b$ if $n$ is even.
\[ 
\xymatrix{
	n\arr^-a\lo_-t&1\arl^-a\arr^-b\lo_-t&n-1\arr^-a\arl^-b\lo_-t&\cdots\arr\arl^-a&\circ\arr\arl\lo_t&\circ\arl\lo_t\ar@(ul,ur)[]^-c}
\qquad
\xymatrix@R=1mm{
	|a|=|b|=-d, \,|t|=-2d-1\\
	da=db=0, \quad dt=a^2+\e b^2 } \]
Moreover, the dg algebra $\Pi$ is $(2d+2)$-Calabi-Yau if $(-1)^d=-1$.
\end{Thm}
We refer to \ref{typeA} for examples of the ($2$-)folded cluster categories $\rC^{(1/2)}_{2d+1}(kQ)$ associted to these $\Pi$.

\subsection*{Acknowledgements}
The author would like to thank Bernhard Keller, Tasuki Kinjo, Osamu Iyama, Hiroyuki Minamoto, and Steffen Oppermann for stimulating discussions and valuable comments.

\section{Calabi-Yau completions}
We know from \cite{Ke11} that given a smooth dg category $\A$ and $n\in\Z$ one can form the {\it $(n+1)$-Calabi-Yau completion}
\[ \Pi_{n+1}(\A):=\rT^\rL_{\A}(\RHom_{\A^e}(\A,\A^e)[n]), \]
which turns out to be an $(n+1)$-Calabi-Yau dg algebra. In fact, it is known that the Calabi-Yau completion has a canonical exact left Calabi-Yau  structure \cite{Ye16,Ke11+}. The first aim of this paper is to construct a root-of-tau version of this Calabi-Yau completion.
\subsection{Inverse dualizing bimodules and tensor categories}\label{prelim}
Let us first collect some basic terminology and conventions.

For a dg category $\A$, we denote by $\A^e=\A^\op\otimes\A$ the {\it enveloping category}. We will identify left and right modules over $\A^e$ by the automorphism $(\A^e)^\op\simeq\A^e$, $a\otimes b\leftrightarrow(-1)^{|a||b|}b\otimes a$, which are further identified with ($k$-central) bimodules over $\A$. We will often denote by $(-)^\vee:=\RHom_{\A^e}(-,\A^e)\colon\rD(\A^e)\to\rD(\A^e)$ the {\it bimodule dual}. Recall that the {\it inverse dualizing bimodule} of $\A$ is an object
\[ \RHom_{\A^e}(\A,\A^e)\in\rD(\A^e), \]
the derived category of bimodules.

We start with the following relationship between the bimodule dual and a one-sided dual.
\begin{Lem}\label{duals}
	Let $\A$ be a dg category, and $X$ an $(\A,\A)$-bimodule such that $X\in\per\A$ and $\RHom_\A(X,\A)\in\per\A$. Then there is a natural isomorphism in $\rD(\A^e)$:
	\[ \RHom_{\A^e}(X,\A^e)\simeq\RHom_\A(X,\A)\lotimes_\A\RHom_{\A^e}(\A,\A^e). \]
\end{Lem}
\begin{proof}
	We have the following sequence of isomorphisms in $\rD(\A^e)$:
	\[
	\begin{aligned}
		\RHom_{\A^e}(X,\A^e)&\overset{(1)}{=}\RHom_{\A^e}(\A,\RHom_\A(X,\A^e))\\
		&\overset{(2)}{=}\RHom_{\A^e}(\A,\A\otimes\RHom_\A(X,\A))\\
		&\overset{(3)}{=}(\A\otimes\RHom_\A(X,\A))\lotimes_{\A^e}\RHom_{\A^e}(\A,\A^e)\\
		&=\RHom_\A(X,\A)\lotimes_\A\RHom_{\A^e}(\A,\A^e)\lotimes_\A\A\\
		&=\RHom_\A(X,\A)\lotimes_\A\RHom_{\A^e}(\A,\A^e).
	\end{aligned}
	\]
	Indeed, these follow from (1) a general formula $\RHom_{\A^e}(-,?)=\RHom_{\A^e}(\A,\RHom_\A(-,?))$ on $\rD(\A^e)$, (2) $X\in\per\A$, and (3) $\RHom_\A(X,\A)\in\per\A$, hence $\A\otimes\RHom_\A(X,\A)\in\per\A^e$.
\end{proof}

\begin{Rem}
The proof shows that if $\A$ is smooth, then the same conclusion holds when $X\in\per\A$ without assuming $\RHom_\A(X,\A)\in\per\A$. This is because the isomorphism (3) holds by the smoothness of $\A$.
\end{Rem}

For a dg category $\A$ we denote by $\per_\dg\!\A$ the {\it dg perfect derived category}, which is the idempotent completion of the pretriangulated hull of $\A$. Concretely, it is (quasi-equivalent to) the dg category whose objects are perfect and cofibrant dg $\A$-modules, and whose morphism spaces are $\Hom$-complexes over $\A$.
We will identify an object $A\in\A$ and a representable dg module $\A(-,A)\in\rD(\A)$ (or $\A(A,-)\in\rD(\A^\op)$).

For a bimodule $X$ over a dg category $\A$ and $n\geq0$ we write $X^{n}=X^{\lotimes_\A n}=X\lotimes_\A\cdots\lotimes_\A X$ ($n$ factors). Then the {\it derived tensor category} is the dg category $\rT^\rL_{\A}\!X$ with the same objects as $\A$ and with morphism spaces $(\rT^\rL_{\A}\!X)(A,B)=\bigoplus_{i\geq0}\RHom_\A(A,B\lotimes_\A X^i)$.

We need an observation on general tensor categories. Let $\Si=\rT^\rL_{\A}\!X$ be the derived tensor category. For $n\geq0$ we denote by $\Si_{\geq n}$ the $(\Si,\Si)$-bimodule $\bigoplus_{l\geq n}X^l$.
\begin{Lem}\label{ten}
There exists a triangle
\[ \xymatrix{ \Si\lotimes_\A X^{n+1}\lotimes_\A\Si\ar[r]& \Si\lotimes_\A X^n\lotimes_\A\Si\ar[r]& \Si_{\geq n} } \]
in $\rD(\Si^e)$, where the first map is given by $1 \otimes (x_0,\ldots,x_{n})\otimes1\mapsto x_0\otimes(x_1,\ldots, x_{n})\otimes1-1\otimes(x_0,\ldots, x_{n-1})\otimes x_{n}$.
\end{Lem}
\begin{proof}
	For $n=0$, the sequence in question is a well-known one
	\[ \xymatrix{ \Si\lotimes_\A X\lotimes_\A\Si\ar[r]& \Si\lotimes_\A\Si\ar[r]& \Si}, \]
	with the first map given by $1\otimes x\otimes 1\mapsto x\otimes1-1\otimes x$. Applying $-\lotimes_\Si \Si_{\geq n}$ yields a triangle
	\[ \xymatrix{ \Si\lotimes_\A X\lotimes_\A \Si_{\geq n}\ar[r]& \Si\lotimes_\A \Si_{\geq n}\ar[r]& \Si_{\geq n}}, \]
	which gives the assertion since $\Si_{\geq n}=X^n \lotimes_\A\Si$ as $(\A,\Si)$-bimodules.
\end{proof}

We end this preliminary subsection with the following compatibility. Recall that dg categories $\A$ and $\B$ are {\it derived Morita equivalent} if there is an $(\A,\B)$-bimodule $X$ giving an equivalence $-\lotimes_\A X\colon\per\A\xsimeq\per\B$. They are furthermore {\it quasi-equivalent} if the equivalence restricts to $H^0\A\xsimeq H^0\B$ under the Yoneda embeddings $H^0\A\to\per\A$ and $H^0\B\to\per\B$.
\begin{Lem}\label{cten}
Let $X\in\rD(\A^\op\otimes\B)$ be a bimodule giving a derived Morita equivalence (resp. a quasi-equivalence), and let $U\in\rD(\A^e)$ and $V\in\rD(\B^e)$ be bimodules which are perfect on the right such that $U\lotimes_\A X\simeq X\lotimes_\B V$ in $\rD(\A^\op\otimes\B)$, that is, the diagram
\[ \xymatrix@!R=2mm{
	\per\A\ar[r]^-{-\lotimes_\A U}\ar[d]_-{-\lotimes_\A X}&\per\A\ar[d]^-{-\lotimes_\A X}\\
	\per\B\ar[r]^-{-\lotimes_\B V}\ar[r]&\per\B } \]
is commutative. Then the derived tensor categories $\rT^{\rL}_\A\!U$ and $\rT^{\rL}_\B\!V$ are derived Morita equivalent (resp. quasi-equivalent).
\end{Lem}
\begin{proof}
	We may assume that the derived Morita equivalence (resp. the quasi-equivalenc) between $\A$ and $\B$ is given by a {\it Morita functor} (resp. a {\it quasi-functor}) $\A\to\B$, that is, a dg functor such that the induction $-\lotimes_\A\B$ yields an equivalence $\per\A\xsimeq\per\B$ (resp. an equivalence $\per\A\xsimeq\per\B$ restricting to an equivalence $H^0\A\xsimeq H^0\B$ via the Yoneda embeddings).
	Indeed, let $\C:=\{X(-,A)\mid A\in\A\}\subset\per_\dg\!\B$ and consider the dg functor $\A\to\per_\dg\!\B$, $A\mapsto X(-,A)$ and the Yoneda embedding $\B\to\per_\dg\!\B$. The first one is a quasi-equivalence (hence a derived Morita equivalence), and the second one is always a derived Morita equivalence, and restricts to a quasi-equivalence $\B\to\C$ when $X$ gives a quasi-equivalence.
	We have a bimodule $W$ over $\C$ making the following diagram commutative.
	\[ \xymatrix{
		\per\A\ar[r]^-\simeq\ar[d]_-{-\lotimes_\A U}&\per\C\ar[d]^-{-\lotimes_\C W}&\per\B\ar[l]_-\simeq\ar[d]^-{-\lotimes_\B V}\\
		\per\A\ar[r]^-\simeq&\per\C&\per\B\ar[l]_-\simeq } \]
	This shows that we may assume there is a Morita functor (resp. a quasi-functor) $\A\to\B$, and in this case $U\in\rD(\A^e)$ is just the restriction of $V\in\rD(\B^e)$.
	Then there is a canonical dg functor $\rT^{\rL}_\A\!U\to\rT^{\rL}_\B\!V$ on the tensor categories, which is certainly a Morita functor (resp. a quasi-equivalence) since $U^{\lotimes_\A n}\xsimeq V^{\lotimes_\B n}$ in $\rD(\A^e)$ for every $n\geq0$. 
\end{proof} 

\subsection{Root pairs and their Calabi-Yau completions}
Let $\A$ and $\B$ be dg categories. Recall that an $(\A,\B)$-bimodule $X$ is {\it invertible} if there is a $(\B,\A)$-bimodule $Y$ such that $X\lotimes_\B Y\simeq\A$ in $\rD(\A^e)$ and $Y\lotimes_\A X\simeq\B$ in $\rD(\B^e)$.
In this case we have $Y\simeq\RHom_\B(X,\B)\simeq\RHom_{\A^\op}(X,\A)$ as bimodules, which we denote by $X^{-1}$. When $\A=\B$ we also write $X^{-n}=\RHom_\A(X,\A)^n$ for each $n\geq0$.

We are now in a position to state our setup, in which we formulate roots of shifted Serre functors in the setting of dg categories.
\begin{Def}\label{setup}
Let $\A$ be a dg category, $d\in\Z$, and $a\geq1$. An {\it $a$-th root} of the shifted inverse dualizing complex $\RHom_{\A^e}(\A,\A^e)[d]$ is an invertible bimodule $U$ satisfying
\begin{enumerate}
	\renewcommand\labelenumi{(\roman{enumi})}
	\renewcommand\theenumi{\roman{enumi}}
	\item\label{root} $U^{a}\simeq\RHom_{\A^e}(\A,\A^e)[d]$ in $\rD(\A^e)$.
\suspend{enumerate}
An {\it $a$-th root pair}, or simply a {\it root pair}, $(U,\P)$ over $\A$ consists of an $a$-th root $U$ of the shifted inverse dualizing complex and a full subcategory $\P\subset\A$ satisfying the following. 
\resume{enumerate}
	\renewcommand\labelenumi{(\roman{enumi})}
	\renewcommand\theenumi{\roman{enumi}}
	\item\label{P} $\per\A=\thick\{ P\lotimes_\A U^i\mid P\in\P, 0\leq i\leq a-1\}$,
	\item\label{exc} $\RHom_\A(P\lotimes_\A U^i,Q)=0$ for all $P,Q\in\P$ and $0\leq i\leq a-1$.
\end{enumerate}
\end{Def}
If $(U,\P)$ is a root pair, then by (\ref{exc}) there is a semi-orthogonal decomposition $\per\A=\thick(\P\lotimes_\A U^a)\perp\cdots\perp\thick\P$. Also by (\ref{P}), the pair $(U,\P)$ determines $\A$ up to a derived Morita equivalence.

\begin{Ex}\label{first}
\begin{enumerate}
\item The case $a=1$ is the ordinary one with the additional assumption that the inverse dualizing complex is invertible. Indeed, the conditions (\ref{P}) and (\ref{exc}) are empty, while (\ref{root}) requires the invertibility of $\RHom_{\A^e}(\A,\A^e)$.
\item\label{Bei} The following case is the motivating example. Let $d\geq1$ and let $A$ be the $d$-Beilinson algebra, thus it is presented by the following quiver with $d+1$ arrows between the vertices, and with commutativity relations.
\[ \xymatrix{ 0\ar@3[r]&1\ar@3[r]&\cdots\ar@3[r]&d } \]
Let $\A=A$ viewed as a dg category and $\P=P_0$ be the subcategory corresponding to the projective module at $0$. Also, let $U=A/P_0\oplus\tau_d^{-1}P_0$, which we view as a bimodule over $A$ by fixing an isomorphism $A\xsimeq\End_A(U)$. Then we have $U^{d+1}\simeq\RHom_A(DA,A)[d]$ and these data fit into the above setting.
More generally, we have a natural $a$-th root pair for $d$-representation infinite algebras arising from graded Calabi-Yau algebras of Gorenstein parameter $a$ \cite{ha3}.
\item Let $\B$ be a dg category and $\A=\B\times\cdots\times\B$, the product of $a$ copies of $\B$. Define $U\in\rD(\A^e)$ by the matrix
\[ \begin{pmatrix} 0&\cdots& 0&\theta \\ \B&\cdots&0& 0\\ \vdots&\ddots&\vdots&\vdots\\0&\cdots&\B&0 \end{pmatrix} \]
with $\theta:=\RHom_{\B^e}(\B,\B^e)[d]$ so that $-\lotimes_\A U\colon (X_1,\ldots,X_a)\mapsto (X_2,\ldots,X_a,X_1\lotimes_\B\theta)$. Letting $\P\subset\A$ be the full subcategory corresponding to the first entry, it is clear that $(U,\P)$ is an $a$-th root pair.
\end{enumerate}
\end{Ex}

The notion of root pairs is left-right symmetric in the following sense.
\begin{Lem}
If $(U,\P)$ is a root pair on $\A$, then $(U,\P^\op)$ is a root pair on $\A^\op$.
\end{Lem}
\begin{proof}
	Clearly the condition (\ref{root}) is left-right symmetric. We have to verify the following.
	\begin{itemize}
	\item[(\ref{P})$^\op$] $\per\A^\op=\thick\{ U^i\lotimes_\A P\mid P\in\P, 0\leq i\leq a-1\}$.
	\item[(\ref{exc})$^\op$] $\RHom_{\A^\op}(U^i\lotimes_\A P,Q)=0$ for all $P,Q\in\P$ and $0\leq i\leq a-1$.
	\end{itemize}
	We make use of the duality $\RHom_\A(-,\A)\colon\per\A\leftrightarrow\per\A^\op$. By (\ref{P}) we also have that $\per\A=\thick\{P\lotimes_\A U^{-i}\mid P\in\P, 0\leq i\leq a-1\}$ since $U$ is invertible. Applying the duality we deduce (\ref{P})$^\op$. Similarly for (\ref{exc})$^\op$, applying the duality to $\RHom_\A(P\lotimes_\A U^i,Q)=0$ yields $\RHom_{\A^\op}(Q,U^{-i}\lotimes_\A P)=0$, hence the desired condition (\ref{exc})$^\op$.
\end{proof}

The above setting leads to the following definition of the $a$-th root version of the Calabi-Yau completion.
\begin{Def}\label{CYcomp}
Let $(U,\P)$ be an $a$-th root pair on a dg category $\A$. We define $\Pi^{(1/a)}_{d+1}(\A)$ to be the full subcategory of the derived tensor category $\rT^\rL_{\A}\!U$ formed by the objects in $\P$, and call it the {\it $(d+1)$-Calabi-Yau completion} of the $a$-th root pair $(U,\P)$.
\end{Def}
Thus the objects of $\Pi^{(1/a)}_{d+1}(\A)$ are the same as $\P$, and the morphisms are given by
\[ \Pi^{(1/a)}_{d+1}(\A)(P,Q)=\bigoplus_{i\geq0}\RHom_\A(P,Q\lotimes_\A U^i) \]
for each $P,Q\in\P$ with the natural composition.
Notice that $\Pi^{(1/a)}_{d+1}(\A)$ depends on $(U,\P)$, thus the notation is ambiguous.
\begin{Rem}
We shall see that the Calabi-Yau completion $\Pi=\Pi^{(1/a)}_{d+1}(\A)$ is {\it not} in general Calabi-Yau; it is only {\it twisted} Calabi-Yau in the sense that there is an isomorphism $\RHom_{\Pi^e}(\Pi,\Pi^e)[d+1]\simeq\Pi_{\al}$ for some automorphism $\al$ of $\Pi$. For $\Pi$ to be (non-twisted) Calabi-Yau we have to make a careful choice of our bimodule $U$, which will be discussed in the following section as {\it cyclically invariant} bimodules. We refer to \ref{example} for a concrete example where the Calabi-Yau property fails.
\end{Rem}

Let us discuss how much the category $\Pi^{(1/a)}_{d+1}(\A)$ depends on $U$, $\P$ or $\A$. 
For this we introduce some equivalence relations on root pairs based on the standard ones of {derived Morita equivalence} and {quasi-equivalence} for dg categories, see the end of Section \ref{prelim}.
\begin{Def}\label{qe}
Let $(U,\P)$ be a root pair on $\A$ and $(V,\cQ)$ a root pair on $\B$ (for fixed $d$ and $a$).
\begin{enumerate}
\item We say $(U,\P)$ and $(V,\cQ)$ are {\it derived Morita equivalent} if the following conditions are satisfied.
\begin{itemize}
	\item There exists an invertible $(\A,\B)$-bimodule $X$ which restricts to an invertible $(\P,\cQ)$-bimodule.
	\item We have an isomorphism $U\lotimes_\A X\simeq X\lotimes_\B V$ in $\rD(\A^\op\otimes\B)$.
\end{itemize}
In other words, there are commutative diagrams of equivalences
\[
\xymatrix@!R=2mm{
	\per\A\ar[r]^-{-\lotimes_\A X}&\per\B\\
	\per\P\ar@{}[u]|-{\vsubset}\ar[r]&\per\cQ,\ar@{}[u]|-{\vsubset} }
\qquad
\xymatrix@!R=2mm{
	\per\A\ar[r]^-{-\lotimes_\A U}\ar[d]_-{-\lotimes_\A X}&\per\A\ar[d]^-{-\lotimes_\A X}\\
	\per\B\ar[r]^-{-\lotimes_\B V}\ar[r]&\per\B. } \]
\item We say $(U,\P)$ and $(V,\cQ)$ are {\it quasi-equivalent} if they are derived Morita equivalent in such a way that $X$ restricts to a quasi-equivalence from $\P$ to $\cQ$. 
\end{enumerate}
\end{Def}

\begin{Rem}
\begin{enumerate}
\item In (1) the commutativity $U\lotimes_\A X\simeq X\lotimes_\B V$ for $a=1$ is automatic \cite[3.10(e)]{Ke11}. 
\item In (2) we do {\it not} require $\A$ and $\B$ to be quasi-equivalent, we only require them to be derived Morita equivalent. This is because the Calabi-Yau completion $\Pi_{d+1}^{(1/a)}(\A)$ depends only on the {\it derived} Morita equivalence class of $\A$, see also the proof of \ref{cMor}(\ref{QE}) below.
\end{enumerate}
\end{Rem}
We have the following natural analogue of \cite[4.2]{Ke11}.
\begin{Prop}\label{cMor}
Let $(U,\P)$ and $(V,\cQ)$ be root pairs on $\A$ and $\B$ respectively.
\begin{enumerate}
\item\label{MO} If $(U,\P)$ and $(V,\cQ)$ are derived Morita equivalent, then $\Pi^{(1/a)}_{d+1}(\A)$ and $\Pi_{d+1}^{(1/a)}(\B)$ are derived Morita equivalent.
\item\label{QE} If $(U,\P)$ and $(V,\cQ)$ are quasi-equivalent, then $\Pi^{(1/a)}_{d+1}(\A)$ and $\Pi_{d+1}^{(1/a)}(\B)$ are quasi-equivalent.
\end{enumerate}
\end{Prop}
\begin{proof}
	(\ref{MO})  Since $\A$ and $\B$ are derived Morita equivalent and $U$ and $V$ are corresponding autoequivalences, the tensor categories $\rT^{\rL}_\A\!U$ and $\rT^{\rL}_\B\!V$ are derived Morita equivalent by \ref{cten}. Therefore their corresponding subcategories $\Pi^{(1/a)}_{d+1}(\A)$ and $\Pi_{d+1}^{(1/a)}(\B)$ are derived Morita equivalent.
	
	(\ref{QE})  Recall that the morphism complex of the Calabi-Yau completion is given as
	\[ \Pi_{d+1}^{(1/a)}(\A)(P,Q)=\bigoplus_{i\geq0}\RHom_\A(P,Q\lotimes_\A U^i) \]
	for each $P,Q\in\P$. Then the assertion certainly follows from this description.
\end{proof}

We end this section with the following important question.
\begin{Qs}
Let $(U,\P)$ and $(U,\cQ)$ be root pairs on a dg category $\A$ for the same $U$. Then are the Calabi-Yau completions derived Morita equivalent?
\end{Qs}

%


\section{Statements of the results}\label{results}
We state the general main results of this article. We have defined the Calabi-Yau completion \ref{CYcomp} of an $a$-th root pair $(U,\P)$ of $\RHom_{\A^e}(\A,\A^e)[d]$ on a dg category $\A$. We formulate a property of a bimodule $U$ which assures the Calabi-Yau property of $\Pi$.
\subsection{The Calabi-Yau property}\label{CYstr}
Let us first summarize our setup.
\begin{itemize}
	\item $\A$ is a dg category and $(U,\P)$ is an $a$-th root pair of $\RHom_{\A^e}(\A,\A^e)[d]$ on $\A$. We assume that $U$ is cofibrant as an $(\A,\A)$-bimodule.
	\item Let $\Si=\rT_\A\!U$ be the tensor category, and $\Pi^{(i)}$ a full subcategory of $\Si$ consisting of objects in $\P^{(i)}:=\{P\otimes_\A U^i\mid P\in\P\}\subset\per_\dg\!\A$. We let $\Pi:=\Pi^{(0)}=\Pi^{(1/a)}_{d+1}(\A)$.
\end{itemize}
Note that we have quasi-equivalences $\P\xsimeq\P^{(i)}$ given by $P\mapsto P\otimes_\A U^i$, and the corresponding quasi-equivalences $\Pi\xsimeq\Pi^{(i)}$.

Let $M$ be a cofibrant bimodule over a dg category $\X$. Consider the action of the cyclic group of order $n$ on $\X\otimes_{\X^e}M^n$ given by
\[ 1\otimes(z_1\otimes\cdots\otimes z_n)\mapsto (-1)^{|z_1|(|z_2|+\cdots+|z_n|)}1\otimes(z_2\otimes\cdots\otimes z_n\otimes z_1). \]
Clearly the action descends to the cohomology.
\begin{Def}
We say that a closed morphism $\varphi\colon\RHom_{\A^e}(\A,\A^e)\to M^n$ of degree $-d$ is {\it $n$-cyclically invariant} if the preimage of $\varphi$ under the canonical isomorphisms
\[ \xymatrix{ \Hom_{\rD(\A^e)}(\A^\vee[d],M^n)&H^{-d}(M^n\lotimes_{\A^e}\A^{\vee\vee})\ar[l]_-\simeq&\ar[l]_-\simeq H^{-d}(M^n\lotimes_{\A^e}\A) } \]
is stable under the action of $\Z/n\Z$.
\end{Def}

\begin{Def}\label{cyc}
We say that an $a$-th root $U$ of $\RHom_{\A^e}(\A,\A^e)[d]$ is {\it cyclically invariant} if one can choose a quasi-isomorphism
\[ \xymatrix{\RHom_{\A^e}(\A,\A^e)\ar[r]& U^a }\]
of degree $-d$ to be $a$-cyclically invariant. We call a root pair $(U,\P)$ {\it cyclically invariant} if $U$ is.
\end{Def}

Now we are ready to state the first main result of the paper.
\begin{Thm}\label{CY}
Let $\A$ be a smooth dg category and $(U,\P)$ an $a$-th root pair of $\RHom_{\A^e}(\A,\A^e)[d]$ which is cyclically invariant. Then $\Pi:=\Pi_{d+1}^{(1/a)}(\A)$ is smooth and is $(d+1)$-Calabi-Yau, that is, there is an isomorphism in $\rD(\Pi^e)$:
\[ \RHom_{\Pi^e}(\Pi,\Pi^e)[d+1]\simeq\Pi. \] 
\end{Thm}

\begin{Rem}
The Calabi-Yau completion $\Pi^{(1/a)}_{d+1}(\A)$ in fact has a canonical {\it left Calabi-Yau structure} in the sense of \cite{KS,Ke11+,Ye16}, which is a refinement of \ref{CY} in terms of the (negative) cyclic homology, and is the `correct' definition of the Calabi-Yau property. 
We will not give a proof of this fact here and postpone it to another paper.
\end{Rem}

In fact, we shall prove the following intermediate result for $a$-th roots of a shifted inverse dualizing complex (with $\P$ possibly absent), which yields a dg category which is not Calabi-Yau, but is still useful and close to being Calabi-Yau.
\begin{Thm}\label{fakeCY}
Let $\A$ be a smooth dg category and $U$ an $a$-th root of $\RHom_{\A^e}(\A,\A^e)[d]$ which is cyclically invariant. Then the tensor category $\Si=\rT^{\rL}_\A\!U$ satisfies
\[ \RHom_{\Si^e}(\Si,\Si^e)[d+1]\simeq\Si_{\geq a-1} \]
in $\rD(\Si^e)$, where $\Si_{\geq a-1}=\bigoplus_{l\geq a-1}U^l$ the two-sided ideal of $\Si$.
\end{Thm}

We shall prove these results in Section \ref{proof}. Before that we present a sketch of the proof of \ref{CY} below which is convincing enough. 
Let $\Si=\rT^\rL_{\A}\!U$ be the derived tensor category and consider the triangle
\[ \xymatrix{ \Si\lotimes_\A U\lotimes_\A\Si\ar[r]& \Si\lotimes_\A\Si\ar[r]& \Si }. \]
We will denote by $(-)e_i$ the restriction of a $\Si$-module to the objects in $\P^{(i)}$, and similarly for left modules or bimodules.
In view of \ref{setup}(\ref{exc}) we have $e_0\Si e_{a-1}\simeq \Pi$ in $\rD(\Pi^e)$ via the quasi-equivalence $\Pi\to\Pi^{(a-1)}$. Therefore restricting the triangle on the left to $\P=\P^{(0)}$ and on the right to $\P^{(a-1)}$ yields a triangle
\[ \xymatrix{ e_0\Si\lotimes_\A U\lotimes_\A\Si e_{a-1}\ar[r]& e_0\Si\lotimes_\A\Si e_{a-1}\ar[r]& \Pi } \]
in $\rD(\Pi^e)$. Applying $\RHom_{\Pi^e}(-,\Pi^e)$ to the sequence and using adjunctions, we get a morphism
\[ \xymatrix{ \RHom_{\A^e}(\A,\RHom_{\Pi^e}(e_0\Si\otimes\Si e_{a-1},\Pi^e))[d]\ar[r]&\RHom_{\A^e}(U,\RHom_{\Pi^e}(e_0\Si\otimes\Si e_{a-1},\Pi^e))[d] } \]
whose mapping cone is $\RHom_{\Pi^e}(\Pi,\Pi^e)[d+1]$.
We have $\RHom_{\Pi^e}(e_0\Si\otimes\Si e_{a-1},\Pi^e)=\RHom_{\Pi^\op}(e_0\Si,\Pi)\otimes\RHom_{\Pi}(\Si e_{a-1},\Pi)$, which is isomorphic to $\Si e_{a-1}\otimes e_0\Si$ by \ref{Pi-dual} below, so that the above morphism extends to the triangle
\[ \xymatrix{e_0\Si\lotimes_\A\RHom_{\A^e}(\A,\A^e)[d]\lotimes_\A\Si e_{a-1}\ar[r]&e_0\Si\lotimes_\A\RHom_{\A^e}(U,\A^e)[d]\lotimes_\A\Si e_{a-1}\ar[r]&\RHom_{\Pi^e}(\Pi,\Pi^e)[d+1] }. \]

On the other hand, let $\Si_{\geq a-1}$ be the submodule $\bigoplus_{l\geq a-1}U^l$ of $\Si$ and consider the triangle
\[ \xymatrix{ \Si\lotimes_\A U^{a}\lotimes_\A\Si\ar[r]& \Si\lotimes_\A U^{a-1}\lotimes_\A\Si\ar[r]& \Si_{\geq a-1} } \]
from \ref{ten}. Restricting as $e_0(-)e_{a-1}$ and noting that $e_0(\Si_{\geq a-1})e_{a-1}\simeq\Pi$ again by \ref{setup}(\ref{exc}), we have a triangle
\[ \xymatrix{ e_0\Si\lotimes_\A U^{a}\lotimes_\A\Si e_{a-1}\ar[r]& e_0\Si\lotimes_\A U^{a-1}\lotimes_\A\Si e_{a-1}\ar[r]& \Pi } \]
in $\rD(\Pi^e)$.

Now by \ref{setup}(\ref{root}) we have $\RHom_{\A^e}(\A,\A^e)[d]=U^a$, and \ref{duals} yields $\RHom_{\A^e}(U,\A^e)[d]=\RHom_\A(U,\A)\lotimes_\A\RHom_{\A^e}(\A,\A^e)[d]=U^{a-1}$. We have therefore obtained two triangles whose first two terms are isomorphic. In fact, one can show that the morphisms are isomorphic, that is, the square in the diagram of triangles below is commutative in $\rD(\Pi^e)$, under the assumption that the root pair is cyclically invariant.
\[ \xymatrix{
	e_0\Si\lotimes_\A\RHom_{\A^e}(\A,\A^e)[d]\lotimes_\A\Si e_{a-1}\ar[r]\ar[d]^-{\rsimeq}&e_0\Si\lotimes_\A\RHom_{\A^e}(U,\A^e)[d]\lotimes_\A\Si e_{a-1}\ar[r]\ar[d]^-{\rsimeq}&\RHom_{\Pi^e}(\Pi,\Pi^e)[d+1]\\
	e_0\Si\lotimes_\A U^{a}\lotimes_\A\Si e_{a-1}\ar[r]& e_0\Si\lotimes_\A U^{a-1}\lotimes_\A\Si e_{a-1}\ar[r]& \Pi } \]
We shall therefore conclude that $\RHom_{\Pi^e}(\Pi,\Pi^e)[d+1]\simeq\Pi$ in $\rD(\Pi^e)$.

\subsection{Correspondences}\label{corr}
We establish a bijection between Calabi-Yau algebras with Gorenstein parameters and higher representation infinite algebras with roots of the Auslander-Reiten translations. 

Let us start by recalling the relevant notions. We will work on ($\Z$-){\it graded dg categories}, that is, a $\Z^2$-graded category with a differential of degree $(1,0)$ subject to the Leibniz rule. The first grading is the usual cohomological grading on dg categories, and the second one is called the {\it Adams grading}. For a graded dg category $\A$ we denote by $\rD^\Z(\A)$ the {\it graded derived category} (see \cite[Section 2.1]{haI} for details).

This leads to the following definition of Calabi-Yau dg categories in the graded setting.
\begin{Def}\label{GPDef}
Let $\Pi$ be an $n$-Calabi-Yau dg category which is Adams graded. We say that $\Pi$ has {\it Gorenstein parameter $a$} if there is an isomorphism
\[ \RHom_{\Pi^e}(\Pi,\Pi^e)[n]\simeq\Pi(a) \]
in the graded derived category $\rD^\Z(\Pi^e)$.
\end{Def}

We next turn to the finite dimensional algebra side.
Let $A$ be a finite dimensional algebra of finite global dimension and $d\in\Z$. We denote by
\[ \nu_d=-\lotimes_ADA[-d]\colon\Db(\mod A)\to \Db(\mod A) \]
the shifted Serre functor. Its inverse is given by $\nu_d^{-1}=-\lotimes_A\RHom_A(DA,A)[d]=-\lotimes_A\RHom_{A^e}(A,A^e)[d]$.
\begin{Def}[{\cite{HIO}}]\label{dRIDef}
Let $d\geq0$ be an integer. A finite dimensional algebra is {\it $d$-representation infinite} if
\[ \nu_d^{-i}A\in\mod A\subset\Db(\mod A) \]
for all $i\geq0$.
\end{Def}
We understand $0$-representation infinite as semisimple. It is a reformulation of the definition that a finite dimensional algebra $A$ is $d$-representation infinite if and only if its $(d+1)$-Calabi-Yau completion is concentrated in degree $0$.

The following result is the bijection which we aim to generalize.
\begin{Thm}[{\cite{Ke11,MM}, see also \cite{AIR,HIO}}]\label{mm}
Let $k$ be a perfect field and $d\geq0$. There exists a bijection between the following.
\begin{enumerate}
	\renewcommand{\labelenumi}{(\roman{enumi})}
	\renewcommand{\theenumi}{\roman{enumi}}
	\item\label{CYalg} The set of Morita equivalence classes of $(d+1)$-Calabi-Yau algebras of Gorenstein parameter $1$, such that each degree part is finite dimensional.
	\item\label{dRI} The set of Morita equivalence classes of $d$-representation infinite algebras.
\end{enumerate}
	The correspondence from (\ref{CYalg}) to (\ref{dRI}) is given by taking the degree $0$ part, and from (\ref{dRI}) to (\ref{CYalg}) by taking the $(d+1)$-Calabi-Yau completion.
\end{Thm}
We extend the above theorem to higher Gorenstein parameters in terms of root pairs. In fact we give the correspondence in the setting of dg categories, and deduce the aforementioned case as the formal case.
We shall apply the construction of Calabi-Yau algebras from the previous section to get a map from (\ref{dRI}) to (\ref{CYalg}).

\subsubsection{Calabi-Yau dg categories and root pairs}
To state our results we recall some equivalence relations for dg categories in the Adams graded setting.
We say that two $\Z$-graded dg categories $\A$ and $\B$ are {\it graded derived Morita equivalent} if there is a graded $(\A,\B)$-bimodule $X$ such that $-\lotimes_\A\B\colon\rD^\Z(\A)\to\rD^\Z(\B)$ is an equivalence. We also say $\A$ and $\B$ are {\it graded quasi-equivalent} if $X$ restricts to an equivalence $H^0\A_0\xsimeq H^0\B_0$, where we regard $H^0\A_0\hookrightarrow\rD^\Z(\A)$ and $H^0\B_0\hookrightarrow\rD^\Z(\B)$ via the Yoneda embedding.

\begin{Thm}\label{MM}
Let $d$ be an integer and $a$ a positive integer. There exists a bijection between the following.
\begin{enumerate}
	\renewcommand{\labelenumi}{(\roman{enumi})}
	\renewcommand{\theenumi}{\roman{enumi}}
	\item\label{CYdgalg} The set of graded quasi-equivalence classes of positively graded $(d+1)$-Calabi-Yau category $\Pi$ of Gorenstein parameter $a$ such that $\Pi_0\in\per\Pi$ and $\Pi_0\in\per\Pi^\op$.
	\item\label{smdg} The set of quasi-equivalence classes of cyclically invariant root pairs $(U,\P)$ on smooth dg categories $\A$.
\end{enumerate}
The correspondences are given as follows.
\begin{itemize}
\item From {\rm(\ref{CYdgalg})} to {\rm(\ref{smdg})}: Take $\Pi$ to $\A=\begin{pmatrix}\Pi_0&0&\cdots&0\\ \Pi_1&\Pi_0&\cdots&0\\ \vdots&\vdots&\ddots&\vdots\\ \Pi_{a-1}&\Pi_{a-2}&\cdots&\Pi_0\end{pmatrix}$, $U=\begin{pmatrix}\Pi_1 &\Pi_0&\cdots& 0\\\vdots&\vdots&\ddots&\vdots \\ \Pi_{a-1}&\Pi_{a-2}&\cdots&\Pi_0\\ \Pi_a&\Pi_{a-1}&\cdots&\Pi_1\end{pmatrix}$, and $\P$ the full subcategory of $\A$ corresponding to the upper left entry of the matrix.
\item From {\rm(\ref{smdg})} to {\rm(\ref{CYalg})}: Take the $(d+1)$-Calabi-Yau completion.
\end{itemize}
\end{Thm}

We get the following consequence as the case our dg Calabi-Yau category $\Pi$ is concentrated in (cohomological) degree $0$. 
\begin{Cor}\label{fdMM}
Let $d\geq0$ be an integer and $a$ a positive integer. There exists a bijection between the following.
\begin{enumerate}
	\renewcommand{\labelenumi}{(\roman{enumi})}
	\renewcommand{\theenumi}{\roman{enumi}}
	\item\label{CYalg/a} The set of graded Morita equivalence classes of positively graded $(d+1)$-Calabi-Yau algebra $\Pi$ of Gorenstein parameter $a$ such that each $\Pi_i$ is finite dimensional.
	\item\label{dRI/a} The set of Morita equivalence classes of cyclically invariant $a$-th root pairs on $d$-representation infinite algebras.
\end{enumerate}
\end{Cor}


\subsubsection{Non-Calabi-Yau dg categories and roots of inverse dualizing complexes}
In fact, we have the following version which characterizes dg categories as in \ref{fakeCY}. Let $\A$ be a dg category and $U$ an $a$-th root of $\RHom_{\A^e}(\A,\A^e)[d]$. We say that such pairs $(\A,U)$ and $(\B,V)$ are {\it quasi-equivalent} if there is an $(\A,\B)$-bimodule $X$ giving a quasi-equivalence between $\A$ and $\B$ such that $U\lotimes_\A X\simeq X\lotimes_\B V$ in $\rD(\A^\op\otimes\B)$. 
\begin{Thm}\label{fakeMM}
Let $d$ be an integer and let $a$ be a positive integer. There exists a bijection between the following.
\begin{enumerate}
\renewcommand{\labelenumi}{(\roman{enumi})}
\renewcommand{\theenumi}{\roman{enumi}}
\item The set of graded quasi-equivalence classes of Adams positively graded smooth dg categories $\Si$ satisfying $\RHom_{\Si^e}(\Si,\Si^e)[d+1]\simeq\Si_{\geq a-1}(a)$ in $\rD^\Z(\Si^e)$.
\item The set of quasi-equivalence classes of pairs $(\A,U)$ where $\A$ is a smooth dg category and $U$ is a cyclically invariant $a$-th root of $\RHom_{\A^e}(\A,\A^e)[d]$.
\end{enumerate}
The correspondence from {\rm (i)} to {\rm (ii)} is given by $\Si\mapsto(\Si_0,\Si_1)$, and from {\rm (ii)} to {\rm (i)} by taking the derived tensor category.
\end{Thm}

The proofs of above results will be given in Section \ref{proof}. Let us demonstrate at this point that we have to make a careful choice of the $a$-th root $U$ in order for $\Pi_{d+1}^{(1/a)}(\A)$ to be Calabi-Yau.
\begin{Ex}\label{example}
	Let $A$ be the path algebra of the Kronecker quiver $\xymatrix{0\ar@2[r]^-x_-y&1}$. Let $\e\in\{\pm1\}$, $n\in\Z$ and consider the dg $(A,A)$-bimodule $U$ concentrated in degrees $-n$ and $-n-1$.
	\[ \xymatrix@R=1mm{
		Ae_1\otimes e_0A\ar[r]&Ae_0\otimes e_1A\\
		e_1\otimes e_0\ar@{|->}[r]& (-1)^n(x\otimes y-\e y\otimes x)} \]
	It is easy to see that as a one-sided module, we have $U\simeq (P_1\oplus\tau^{-1}P_0)[n]$, the shift of the reflection at the simple projective module, and the canonical morphism $A\to\End_A(U)$ is an isomorphism for both $\e=\pm1$. We observe the following.
	\begin{enumerate}
		\item $U\lotimes_AU\simeq\RHom_{A^e}(A,A^e)[2n+1]$ for both $\e=\pm1$.
		\item $U$ is cyclically invariant if and only if $\e=(-1)^n$.
	\end{enumerate}
	It follows from (1) that $(U,P_0)$ is a root pair on $A$, and it is easy to see that its $(2n+2)$-Calabi-Yau completion is the (skew) polynomial ring $k\langle x,y\rangle/(xy-\e yx)$ with $\deg x=\deg y=-n$ and trivial differentials. We know such a formal dg algebra is Calabi-Yau if and only if $\e=(-1)^n$ \cite[5.2]{ha3}, which is consistent with claim (2) and our main result \ref{CY}.

	Let us now verify these claims. To see (1), we can first compute the left-hand-side as
	\[ \xymatrix@R=1mm{
		Ae_0\otimes e_1Ae_1\otimes e_0A\ar[dr] && e_0\otimes e_1e_1\otimes e_0\ar@{|->}[r]&e_0\otimes x\otimes y-\e e_0\otimes y\otimes x\\
		\oplus&Ae_0\otimes V\otimes e_1A&&\\
		Ae_1\otimes e_0Ae_0\otimes e_1A\ar[ur] && e_1\otimes e_0e_0\otimes e_1\ar@{|->}[r]&(-1)^n(x\otimes y\otimes e_1-\e y\otimes x\otimes e_1), } \]
	the complex concentrated in degree $-2n-1$ and $-2n$, with $V=e_1Ae_0$ the $2$-dimensional vector space with basis $\{x,y\}$.
	We next compute the right-hand-side. Consider the bimodule resolution
	\[ \xymatrix@R=0.3mm{
		&&Ae_0\otimes e_0A && v\otimes e_0\\
		pA:&Ae_1\otimes V\otimes e_0A\ar[ur]\ar[dr]&\oplus&e_1\otimes v\otimes e_0\ar@{|->}[ur]\ar@{|->}[dr]&\\
		&&Ae_1\otimes e_1A && -e_1\otimes v} \]
	of $A$. Dualizing, we see that the $\RHom_{A^e}(A,A^e)$ is isomorphic to the complex 
	\[ \xymatrix@R=0.3mm{
		Ae_0\otimes e_0A\ar[dr] && e_0\otimes e_0\ar@{|->}[r]&-e_0\otimes x^\ast\otimes x-e_0\otimes y^\ast\otimes y\\
		\oplus&Ae_0\otimes DV\otimes e_1A&&\\
		Ae_1\otimes e_1A\ar[ur] && e_1\otimes e_1\ar@{|->}[r]&x\otimes x^\ast\otimes e_1+y\otimes y^\ast\otimes e_1 } \]
	concentrated in degree $0$ and $1$, where $\{x^\ast,y^\ast\}$ is the basis of $DV$ dual to $\{x,y\}$.
	We see that there is an isomorphism $\varphi\colon\RHom_{A^e}(A,A^e)\xsimeq U\lotimes_AU$ of degree $-2n-1$ given by multiplication by $\e$ on $Ae_0\otimes e_0A=Ae_0\otimes e_1Ae_1\otimes e_0A$, by $(-1)^n$ on $Ae_1\otimes e_1A=Ae_1\otimes e_0Ae_0\otimes e_1A$, and via $DV\to e_1Ae_0$; $x^\ast\mapsto-y$, $y^\ast\mapsto \e x$ on the remaining term. (Note that since $\varphi$ is of odd degree $-2n-1$, we should get skew-commutative squares.) This proves (1).
	
	To see (2), we use the computation in Section \ref{proof}. Note first that the Casimir element for $pA$ is $(e_0\otimes e_0)\otimes(e_0^\ast\otimes e_0^\ast)+(e_1\otimes e_1)\otimes(e_1^\ast\otimes e_1^\ast)+(\text{terms for arrows})\in pA\otimes_{A^e}(pA)^\vee$. Then by \ref{LemC}(\ref{C**}), we see that the element of $U^2\otimes_{A^e}A$ corresponding (up to homotopy) to $\varphi$ under the quasi-isomorphism $U^2\otimes_{A^e}A\ysimeq U^2\otimes_{A^e}pA\xsimeq\cHom_{A^e}((pA)^\vee,U^2)$ is $\varphi(e_0\otimes e_0)\otimes e_0+\varphi(e_1\otimes e_1)=\e(e_0\otimes e_1e_1\otimes e_0)\otimes e_0+(-1)^n(e_1\otimes e_0e_0\otimes e_1)\otimes e_1$. Recalling that the $\Z/2\Z$-action on $U^2\otimes_{A^e}A$ is the permutation of elements of $U$, we deduce that this is cyclically invariant if and only if $\e=(-1)^n$.

Note that this example also shows that cyclic invariance is not preserved under the suspension.
\end{Ex}

\section{Proofs}\label{proof}
The point is to interpret the commutativity of the diagram in the sketch in Section \ref{CYstr} in terms of Hochschild homology. Let us prepare some necessary material.
\subsection{Casimir elements}
Let $X\to Y$ be a morphism of complexes and $y\in Y$ a cocycle. We say that a cocycle $x\in X$ is a {\it homotopy preimage} of $y$ if its image is homotopic to $y$. Clearly, homotopy preimages under quasi-isomorphisms uniquely exist up to homotopy.

Let $\X$ be a dg category and $M$ a cofibrant dg $\X$-module which is perfect in the derived category $\rD(\X)$. We have a canonical quasi-isomorphism
\[ \xymatrix{M\otimes_\X\cHom_\X(M,\X)\ar[r]&\cHom_\X(M,M) }. \]
Following \cite{Ke11,Ke11+}, a {\it Casimir element} is an element of $M\otimes\cHom_\X(M,\X)$ which is a homotopy preimage of the identity. By definition, a Casimir element is unique up to homotopy.

Let us collect some easy properties of Casimir elements. We denote by $(-)^\ast=\cHom_\X(-,\X)$, and by $\e\colon M\to M^{\ast\ast}$ the evaluation map $x\mapsto (\e_x\colon f\mapsto (-1)^{|f||x|}f(x))$.
\begin{Lem}\label{LemC}
Let $\X$ be a dg category and $M$ a cofibrant $\X$-module which is perfect. Let $\sum_ix_i\otimes x_i^\ast\in M\otimes_\X M^\ast$ be a Casimir element.
\begin{enumerate}
\item\label{mor} For any morphism $f\colon M\to N$ of dg $\X$-modules, its homotopy preimage under the quasi-isomorphism $N\otimes_\X M^\ast\to\cHom_\X(M,N)$ is $\sum_if(x_i)\otimes x_i^\ast$.
\item\label{C*} A Casimir element of $M^\ast$ is $\sum_i\e_{x_i}\otimes x_i^\ast\in M^{\ast\ast}\otimes_\X M^{\ast}$.
\item\label{C**} The homotopy preimage of $f$ under the quasi-isomorphisms $\cHom_{\X^\op}(M^\ast,N)\ysimeq M^{\ast\ast}\otimes_\X N\ysimeq M\otimes_\X N$ is $\sum_i(-1)^{|f||x_i|}x_i\otimes f(x_i^\ast)$.
\item\label{equal} Let $f, g\in\cHom_{\X^\op}(M^\ast,N)$ be closed morphisms of the same degree. Then $f$ and $g$ are homotopy equivalent if and only if $\sum_i(-1)^{|f||x_i|}x_i\otimes f(x_i^\ast)$ and $\sum_i(-1)^{|g||x_i|}x_i\otimes g(x_i^\ast)$ are homotopic in $M\otimes_\X N$.
\end{enumerate}
\end{Lem}
\begin{proof}
	The assertion (\ref{mor}) follows from the commutative diagram on the left below, and (\ref{C*}) from the one on the right.
	\[ \xymatrix{
		M\otimes_\X M^\ast\ar[r]^-\simeq\ar[d]_-{f\otimes1}&\cHom_\X(M,M)\ar[d]^-{f\cdot}\\
		N\otimes_\X M^\ast\ar[r]^-\simeq&\cHom_\X(M,N) }
	\qquad
	   \xymatrix{
		M\otimes_\X M^\ast\ar[r]^-\simeq\ar[d]_-{\e\otimes1}&\cEnd_\X(M)\ar[d]^-{(-)^\ast}\\
		M^{\ast\ast}\otimes_\X M^\ast\ar[r]^-\simeq&\cEnd_{\X^\op}(M^\ast) } \]
	Then (\ref{C**}) and (\ref{equal}) follow immediately.  
\end{proof}

Now suppose that $\X=\A^e$ for a dg category $\A$ and $M$ is a cofibrant and perfect bimodule over $\A$. The commutative diagram
\[ \xymatrix@R=2mm{
	M\otimes_{\A^e}\cHom_{\A^e}(M,\A^e)\ar[drr]\ar@{<->}[dd]&&\\
	&&\cHom_{\A^e}(M,M)\\
	\cHom_{\A^e}(M,\A^e)\otimes_{\A^e}M\ar[urr]&&} \]
shows that the Casimir elements for bimodules are left-right symmetric.	


Let $\A$ be a smooth dg category and $M$ an $(\A,\A)$-bimodule. The isomorphism
\[ \RHom_{\A^e}(M,-)=\RHom_{\A^e}(\A,\RHom_\A(M,-)) \]
on $\rD(\A^e)$ shows that $M$ is perfect as a bimodule if it is perfect as a right (resp. left) module (\cite[proof of 2.5]{Ke98}). In this case, it therefore makes sense to talk about the Casimir element of the bimodule $M$ and that of a right module, and we discuss their relationship.
\begin{Lem}\label{chase}
	Let $\A$ be a smooth dg category and $M$ an $(\A,\A)$-bimodule which is perfect on each side.
	\begin{enumerate}
		\item
		There exists a commutative diagram in $\rD(k)$ consisting of canonical isomorphisms.
		\[ \xymatrix{
			\RHom_{\A^e}(\A,\A)\ar[r]&\RHom_{\A^e}(M,M)\\
			&M\lotimes_{\A^e}\RHom_{\A^e}(M,\A^e)\ar@{-}[d]^-{\ref{duals}}\ar[u]\\
			\A\lotimes_{\A^e}\RHom_{\A^e}(\A,\A^e)\ar[uu]\ar@{=}[r]&\A\lotimes_{\A^e}(M\lotimes_\A\RHom_\A(M,\A)\lotimes_\A\RHom_{\A^e}(\A,\A^e)), } \]
		where the upper horizontal map is given by the multiplication, and the lower one is induced by $\A\xsimeq M\lotimes_\A\RHom_\A(M,\A)$.
		\item If $M$ is invertible and $\sum_ix_i\otimes x_i^\ast\in\A\lotimes_{\A^e}\RHom_{\A^e}(\A,\A^e)$ is a Casimir element for $\A$, and $\sum_lz_l\otimes z_l^\ast\in M\lotimes_\A\RHom_\A(M,\A)$ is a Casimir element for the right module $M$, then the image of the Casimir element of the bimodule $M$ under the isomorphism
		\[ M\lotimes_{\A^e}\RHom_{\A^e}(M,\A^e)\simeq M\lotimes_{\A^e}(\RHom_\A(M,\A)\lotimes_\A\RHom_{\A^e}(\A,\A^e)) \]
		is $\sum_{i,l}\pi(x_i)z_l\otimes z_l^\ast\otimes x_i^\ast$.
	\end{enumerate}
\end{Lem}
\begin{proof}
	We track the above diagram in the resolutions. We may assume that $M$ is cofibrant as a bimodule, and let $\pi\colon p\A\to\A$ be a cofibrant resolution over $\A^e$. With the canonical quasi-isomorphism $\cHom_{\A^e}(-,?)=\cHom_{\A^e}(\A,\cHom_\A(-,?))\xsimeq\cHom_{\A^e}(p\A,\cHom_\A(-,?))$ 
	and the left multiplication map $p\A\to\A\to\cEnd_\A(M)$ in mind, we get the diagram below which represents and refines the above one.
	\[ \xymatrix@C=0.3mm{
		&\cEnd_{\A^e}(M)\ar[d]_-\rsimeq&&M\otimes_{\A^e}\cHom_{\A^e}(M,\A^e)\ar[d]^-\rsimeq\ar[ll]\\
		\cHom_{\A^e}(p\A,p\A)\ar[r]&\cHom_{\A^e}(p\A,\cEnd_\A(M))&&M\otimes_{\A^e}\cHom_{\A^e}(p\A,\cHom_{\A}(M,\A^e))\ar[ll]\\
		p\A\otimes_{\A^e}\cHom_{\A^e}(p\A,\A^e)\ar[u]\ar[r]&\cEnd_\A(M)\otimes_{\A^e}\cHom_{\A^e}(p\A,\A^e)\ar[u]&\\
		&(M\otimes_\A\cHom_\A(M,\A))\otimes_{\A^e}\cHom_{\A^e}(p\A,\A^e)\ar[u]&&M\otimes_{\A^e}(\cHom_\A(M,\A)\otimes_\A\cHom_{\A^e}(p\A,\A^e))\ar[uu]\ar@{=}[ll]} \]
	Here, the rightmost column is the isomorphism constructed in \ref{duals}, and the leftgoing map in the second row is $m\otimes \psi\mapsto (x\mapsto(n\mapsto m\cdot\psi(x)(n)))$. It is not difficult to verify the commutativity.
	
	Note that both $1\in\cEnd_{\A^e}(p\A)$ (in the leftmost column) and $1\in\cEnd_{\A^e}(M)$ (in the top row) are mapped to the same element in $\cHom_{\A^e}(p\A,\cEnd_\A(M))$ (which takes $x\in p\A$ to the left multiplication map by $\pi(x)$). 
	It follows from the commutativity of the diagram that the image of the Casimir element of the bimodule $M$ in the bottom right corner is $\sum_{i,l}\pi(x_i)z_l\otimes z_l^\ast\otimes x_i^\ast$.
\end{proof}

We now assume that the bimodule $M$ is cofibrant and invertible over $\A$. Suppose further that we have a closed morphism
\[ \xymatrix{\varphi\colon\cHom_{\A^e}(p\A,\A^e)\ar[r]&M^n } \]
of degree $-d$ for some $n>0$. Via \ref{duals} it gives rise to a morphism
\[ \xymatrix{
	\varphi^+\colon\cHom_{\A^e}(M,\A^e)\ar[r]^-\simeq&\cHom_\A(M,\A)\otimes_\A\cHom_{\A^e}(p\A,\A^e)\ar[r]^-{1\otimes\varphi}&\cHom_\A(M,\A)\otimes_\A M^n\ar[r]^-{\ev\otimes1^{(n-1)}}&M^{n-1} } \]
of the same degree $-d$. Here the first map is obtained by lifting the isomorphism in \ref{duals} up to homotopy, that is, it makes the following triangle commutative up to homotopy. We denote this map by $\al$.
\[ \xymatrix@R=3mm{
	\cHom_{\A^e}(M,\A^e)\ar[r]^-\simeq\ar@/_18pt/[rr]_-\al&\cHom_{\A^e}(p\A,\cHom_\A(M,\A^e))&\ar[l]_-\simeq\cHom_\A(M,\A)\otimes_\A\cHom_{\A^e}(p\A,\A^e) \\&&} \]
Recall that we denoted the bimodule cofibrant resolution by $\pi\colon p\A\to\A$.
\begin{Lem}\label{itsumo}
	Let $\sum_ix_i\otimes x_i^\ast\in p\A\otimes_{\A^e}\cHom_{\A^e}(p\A,\A^e)$ be the Casimir element for $p\A$, and let $\sum_jy_j\otimes y_j^\ast\in M\otimes_{\A^e}\cHom_{\A^e}(M,\A^e)$ be the Casimir element for $M$. Then we have
	\[  \sum_i(-1)^{d|x_i|}\pi(x_i)\otimes\varphi(x_i^\ast)=\sum_j(-1)^{d|y_j|}y_j\otimes\varphi^+(y_j^\ast) \]
	up to homotopy, under the canonical quasi-isomorphism $\A\otimes_{\A^e}M^n\simeq M\otimes_{\A^e}M^{n-1}$ given by $1\otimes(z_1\otimes\cdots z_n)\leftrightarrow z_1\otimes(z_2\otimes\cdots\otimes z_n)$.
\end{Lem}
\begin{proof}
	Let $\sum_lz_l\otimes z_l^\ast\in M\lotimes_\A\RHom_\A(M,\A)$ be the Casimir element for the right module $M$. By \ref{chase}(2), the image of $\sum_jy_j\otimes y_j^\ast$ under $1\otimes\al$ is homotopic to $\sum_{i,l}\pi(x_i)z_l\otimes z_l^\ast\otimes x_i^\ast$.
	Now we follow the diagram below.
	\[ \xymatrix@C=3mm{
		M\otimes_{\A^e}(\cHom_\A(M,\A)\otimes_\A\cHom_{\A^e}(p\A,\A^e))\ar[r]^-{1\otimes(1\otimes\varphi)}&M\otimes_{\A^e}(\cHom_\A(M,\A)\otimes_\A M^n)\ar[rr]^-{1\otimes(\ev\otimes1^{(n-1)})}\ar@{=}[d]&&M\otimes_{\A^e}M^{n-1}\ar@{=}[d]\\
		&\A\otimes_{\A^e}(M\otimes_\A\cHom_\A(M,\A)\otimes_\A M^n)\ar[rr]^-{1\otimes(1\otimes\ev\otimes1^{n-1})}\ar@{=}[d]&&\A\otimes_{\A^e}(M\otimes_\A M^{n-1})\ar@{=}[d]\\
		&\A\otimes_{\A^e}(M\otimes_\A\cHom_\A(M,\A)\otimes_\A M^n)\ar[rr]^-{1\otimes(\eta\otimes1^{n})}&&\A\otimes_{\A^e}M^n } \]
	It is clear that the upper square is commutative, and the lower one is also commutative up to homotopy (note that $\eta$ is a lift of the counit).
	As the image of $\sum_jy_j\otimes y_j^\ast$ in the upper right corner, we get the element $\sum_j(-1)^{d|y_j|}y_j\otimes\varphi^+(y_j^\ast)\in M\otimes_{\A^e} M^{n-1}$. On the other hand, the image of $\sum_{i,l}\pi(x_i)z_l\otimes z_l^\ast\otimes x_i^\ast$ into the lower right corner is $\sum_i(-1)^{d|x_i|}\pi(x_i)\otimes\varphi(x_i^\ast)\in\A\otimes_{\A^e}M^n$. We conclude that they are homotopic under the canonical identification $M\otimes_{\A^e}M^{n-1}\simeq\A\otimes_{\A^e}M^n$.
\end{proof}

\subsection{Compatible morphisms}
Let us formulate the following which plays a crucial role in the proof of the results in Section \ref{results}. As before let $M$ be an invertible bimodule over a smooth dg category $\A$, $d$ an integer, and $n$ a positive integer.

Let $\Ga=\rT^{\rL}_\A\!M$ be the derived tensor category. We have the standard triangle
\[ \xymatrix{\Ga\lotimes_\A M\lotimes_\A\Ga\ar[r]&\Ga\lotimes_\A\Ga\ar[r]&\Ga } \]
in $\rD(\Ga^e)$, whose bimodule dual is
\[ \xymatrix{\Ga\lotimes_\A\RHom_{\A^e}(\A,\A^e)\lotimes_\A\Ga\ar[r]&\Ga\lotimes_\A\RHom_{\A^e}(M,\A^e)\lotimes_\A\Ga\ar[r]&\RHom_{\Ga^e}(\Ga,\Ga^e)[1] }. \]
Indeed, for $X\in\per\A^e$ we have $\Ga\lotimes_\A X\lotimes_\A\Ga= X\lotimes_{\A^e}\Ga^e$, so its bimodule dual is $\RHom_{\A^e}(X,\Ga^e)$ by adjunction, which is further isomorphic to $\Ga^e\lotimes_{\A^e}\RHom_{\A^e}(X,\A^e)=\Ga\lotimes_\A\RHom_{\A^e}(X,\A^e)\lotimes_\A\Ga$ since $X$ is perfect as a bimodule.
\begin{Def}\label{compat}
	Consider the derived tensor category $\Ga=\rT^{\rL}_\A\!M$.	We say that a pair $(\varphi,\psi)$ of morphisms $\varphi\colon\RHom_{\A^e}(\A,\A^e)[d]\to M^n$ and $\psi\colon\RHom_{\A^e}(M,\A^e)[d]\to M^{n-1}$ in $\rD(\A^e)$ is {\it compatible} is the following diagram in $\rD(\Ga^e)$ is commutative.
	\begin{equation}\label{eqcompat}
	\xymatrix{
		\Ga\lotimes_\A \RHom_{\A^e}(\A,\A^e)[d]\lotimes_\A\Ga\ar[r]\ar[d]_-{1\otimes\varphi\otimes1}&\Ga\lotimes_\A \RHom_{\A^e}(M,\A^e)[d]\lotimes_\A\Ga\ar[d]^-{1\otimes\psi\otimes1}\\
		\Ga\lotimes_\A M^n\lotimes_\A\Ga\ar[r]&\Ga\lotimes_\A M^{n-1}\lotimes_\A\Ga }
	\end{equation}
	Here, the second row is the triangle in \ref{ten}, and the first row is the bimodule dual of the one for $n=1$.
\end{Def}

Now we give the technical heart of the proof. Our discussion closely follows \cite{Ke11+}, but the formulation of cyclic invariance is new.
\begin{Prop}\label{daiji}
The following are equivalent for $\varphi\colon\RHom_{\A^e}(\A,\A^e)[d]\to M^n$.
\begin{enumerate}
	\renewcommand\labelenumi{(\roman{enumi})}
	\renewcommand\theenumi{\roman{enumi}}
	\item $\varphi$ is $n$-cyclically invariant.
	\item There is $\psi\colon\RHom_{\A^e}(M,\A^e)[d]\to M^{n-1}$ such that $(\varphi,\psi)$ is compatible.
\end{enumerate}
In this case, $\psi$ is equal to $-\varphi^+$ as a morphism in $\rD(\A^e)$.
\end{Prop}
\begin{proof}
	We may assume that $M$ is cofibrant, and we also denote by $\pi\colon p\A\to\A$ a bimodule cofibrant resolution. To simplify notation we view $\varphi$ as a (closed) morphism $\cHom_{\A^e}(p\A,\A^e)\to M^n$ of degree $-d$, and similarly for $\psi\colon\cHom_{\A^e}(M,\A^e)\to M^{n-1}$ or $\varphi^+$.
	
	Let $(\varphi,\psi)$ be a pair of closed morphisms of degree $-d$. The commutativity of the diagram (\ref{eqcompat}) in \ref{compat} is presented by the square
	\begin{equation}\label{eqHom}
	\xymatrix@C=3mm{
		\cHom_{\Ga^e}(\Ga\otimes_\A\cHom_{\A^e}(M,\A^e)\otimes_\A\Ga,\Ga\otimes_\A M^n\otimes_\A\Ga)\ar[r]\ar[d]&\cHom_{\Ga^e}(\Ga\otimes_\A\cHom_{\A^e}(p\A,\A^e)\otimes_\A\Ga,\Ga\otimes_\A M^n\otimes_\A\Ga)\ar[d]\\
		\cHom_{\Ga^e}(\Ga\otimes_\A\cHom_{\A^e}(M,\A^e)\otimes_\A\Ga,\Ga\otimes_\A M^{n-1}\otimes_\A\Ga)\ar[r]&\cHom_{\Ga^e}(\Ga\otimes_\A\cHom_{\A^e}(p\A,\A^e)\otimes_\A\Ga,\Ga\otimes_\A M^{n-1}\otimes_\A\Ga), }
	\end{equation}
	that is, the morphism $1\otimes\varphi\otimes1$ lies in the upper right corner, while $1\otimes\psi\otimes1$ lies in the lower left corner, and the commutativity of (\ref{eqcompat}) amounts to saying that their images in the lower right corner are homotopic. By the canonical natural isomorphisms, the above square is quasi-isomorphic to
	\begin{equation}\label{eqotimes}
	\xymatrix{
		(\Ga\otimes_\A M^n\otimes_\A\Ga)\otimes_{\Ga^e}(\Ga\otimes_\A M\otimes_\A\Ga)\ar[r]\ar[d]&(\Ga\otimes_\A M^n\otimes_\A\Ga)\otimes_{\Ga^e}(\Ga\otimes_\A p\A\otimes_\A\Ga)\ar[d]\\
		(\Ga\otimes_\A M^{n-1}\otimes_\A\Ga)\otimes_{\Ga^e}(\Ga\otimes_\A M\otimes_\A\Ga)\ar[r]&(\Ga\otimes_\A M^{n-1}\otimes_\A\Ga)\otimes_{\Ga^e}(\Ga\otimes_\A p\A\otimes_\A\Ga). }
	\end{equation}
	By \ref{LemC}(\ref{C**}), the element $1\otimes\varphi\otimes1$ in the upper right corner of (\ref{eqHom}) becomes
	\[ c_1:=\sum_i(-1)^{|x_i|}(1\otimes\varphi(x_i^\ast)\otimes1)\otimes(1\otimes x_i\otimes1) \]
	in (\ref{eqotimes}), where $\sum_i x_i\otimes x_i^\ast\in p\A\otimes_{\A^e}\cHom_{\A^e}(p\A,\A^e)$ is the Casimir element of the bimodule $p\A$. Similarly, the element $1\otimes\psi\otimes1$ in the lower left corner of (\ref{eqHom}) becomes
	\[ c_2:=\sum_j(-1)^{|y_j|}(1\otimes\psi(y_j^\ast)\otimes1)\otimes(1\otimes y_j\otimes1) \]
	in (\ref{eqotimes}), where $\sum_j y_j\otimes y_j^\ast\in M\otimes_{\A^e}\cHom_{\A^e}(M,\A^e)$ is the Casimir element of $M$.
	Let us write $\varphi(x_i^\ast)\in M^n$ as $\varphi_1(x_i^\ast)\otimes\cdots\otimes\varphi_n(x_i^\ast)$, and define $\varphi_{1\sim n-1}(x_i^\ast)\in M^{n-1}$ and so on in an obvious way. Then the image of $c_1$ in the lower right corner of (\ref{eqotimes}) is
	\[ \sum_i(-1)^{|x_i|}(\varphi_1(x_i^\ast)\otimes\varphi_{2\sim n}(x_i^\ast)\otimes1)\otimes(1\otimes x_i\otimes1)-\sum_i(-1)^{|x_i|}(1\otimes\varphi_{1\sim n-1}(x_i^\ast)\otimes\varphi_n(x_i^\ast))\otimes(1\otimes x_i\otimes1), \]
	which we denote by $c_{11}-c_{12}$. Also, the image of $c_2$ is
	\[ \sum_j(-1)^{|y_j|}(1\otimes\psi(y_j^\ast)\otimes1)\otimes(y_j\otimes \widetilde{1}\otimes1)-\sum_j(-1)^{|y_j|}(1\otimes\psi(y_j^\ast)\otimes1)\otimes(1\otimes \widetilde{1}\otimes y_j), \]
	which we denote by $c_{21}-c_{22}$, and where $\widetilde{1}$ is a preimage of the identity maps.
	Let us track these elements under the isomorphism
	\[ (\Ga\otimes_\A M^{n-1}\otimes_\A\Ga)\otimes_{\Ga^e}(\Ga\otimes_\A p\A\otimes_\A\Ga)=(\Ga\otimes_\A M^{n-1}\otimes_\A\Ga)\otimes_{\A^e}p\A. \]
	It is readily seen that $c_{11}, c_{22}$ lie in the summand $(M\otimes_\A M^{n-1}\otimes_\A\A)\otimes_{\A^e}p\A=M^n\otimes_{\A^e}\A$, and $c_{12}, c_{21}$ lie in $(\A\otimes_\A M^{n-1}\otimes_\A M)\otimes_{\A^e}p\A=M^{n-1}\otimes_{\A^e}M$.
	As elements in $M^n\otimes_{\A^e}\A$, we have 
	\begin{equation}\label{eq1122}
	c_{11}=\sum_i(-1)^{|x_i|}\varphi(x_i^\ast)\otimes\pi(x_i)=\sum_i(-1)^{d|x_i|}\pi(x_i)\varphi(x_i^\ast)\otimes1, \quad \text{and} \quad c_{22}=\sum_j(-1)^{d|y_j|}y_j\psi(y_j^\ast)\otimes1.
	\end{equation}
	On the other hand, as elements of $M^{n-1}\otimes_{\A^e}M=M^n\otimes_{\A^e}\A$, we have 
	\begin{equation}\label{eq1221}
	c_{12}=\sum_i(-1)^{|x_i|}\varphi(x_i^\ast)\otimes\pi(x_i)=\sum_i(-1)^{d|x_i|}\pi(x_i)\varphi(x_i^\ast)\otimes1, \quad \text{and} \quad c_{21}=\sum_j(-1)^{|y_j|}\psi(y_j^\ast)\otimes y_j.
	\end{equation}
	Notice that $c_{22}$ is nothing but the image of $c_{21}$ under the $\Z/n\Z$ action on $M^n\otimes_{\A^e}\A$.
	
	Now we prove (i) implies (ii). Suppose that $\varphi$ is $n$-cyclically invariant. We claim that $(\varphi,-\varphi^+)$ is compatible. For this we have to show that $c_{11}$ and $c_{22}$ in (\ref{eq1122}) as well as $c_{12}$ and $c_{21}$ in (\ref{eq1221}) are homotopic for $\psi=\varphi^+$. By \ref{itsumo}, we know that $c_{11}$ and $c_{22}$ are homotopic. Also, since $c_{21}$ is the cyclic permutation of $c_{22}$, the cyclic invariance of $\varphi$ implies $c_{12}$ and $c_{21}$ are homotopic.
	
	Let us prove the converse. We assume that $(\varphi,-\psi)$ is compatible and show that $\psi=\varphi^+$ and that $\varphi$ is $n$-cyclically invariant. By assumption, we know that $c_{11}$ and $c_{22}$ in (\ref{eq1122}) are homotopic. On the other hand, we also know by \ref{itsumo} that $c_{11}$ and $c_{22}^\prime:=\sum_j(-1)^{d|y_j|}y_j\varphi^+(y_j^\ast)\otimes1$ are homotopic. Therefore, $c_{22}$ and $c_{22}^\prime$ are homotopic, hence $\psi$ and $\varphi^+$ gives the same morphism in the derived category by \ref{LemC}(\ref{equal}). Then the homotopy equivalence of $c_{12}$ and $c_{21}$ in (\ref{eq1221}) implies $\varphi$ is $n$-cyclically invariant.
\end{proof}

Now we make use of a characterization of derived tensor categories among graded dg categories.
Let $\X$ be an (Adams) positively graded dg category. For a subset $I\subset\Z$ we denote by
\[ \per^I\!\X=\thick\{ \X(-,A)(-i)\mid A\in\X,\, i\in I\}, \]
the thick subcategory of $\per^\Z\!\X$ generated in degree $i\in I$. If $I=\{i\in\Z \mid i\geq a\}$ we use the notation $\per^{\geq a}\!\X$ and so on. Then the graded perfect derived category $\per^\Z\!\X$ has a stable $t$-structure
\[ \per^\Z\!\X=\per^{\leq l}\!\X\perp\per^{\geq l+1}\!\X \]
for each $l\in\Z$.
\begin{Prop}[\cite{HaIO}]\label{chtensor}
Let $\Ga$ be a positively graded dg category with $\Ga_0=:\A$ and $\Ga_1=:M$. Suppose that $\Ga\in\per^{[0,1]}\!\Ga^e$. Then the canonical map $\rT^{\rL}_\A\!M\to\Ga$ is a graded quasi-equivalence.
\end{Prop}

We are now ready to prove \ref{fakeCY} as a direct consequence of \ref{daiji}, and also \ref{fakeMM}.

\begin{proof}[Proof of \ref{fakeCY}]
	Dualizing the canonical sequence $\Si\lotimes_\A U\lotimes_\A\Si\to\Si\lotimes_\A\Si\to\Si$ in $\rD(\Si^e)$ we get a morphism in the first row below, whose mapping cone is $\RHom_{\Si^e}(\Si,\Si^e)[d+1]$. On the other hand, we also have a morphism in \ref{ten} on the second row below whose mapping cone is $\Si_{\geq a-1}$.
	By cyclic invariance of $U$ we have an isomorphism $\varphi\colon \RHom_{\A^e}(\A,\A^e)[d]\to U^a$ which makes the diagram
	\[ \xymatrix{
		\Si\lotimes_\A\RHom_{\A^e}(\A,\A^e)[d]\lotimes_\A\Si\ar[r]\ar[d]_-{1\otimes\varphi\otimes1}&\Si\lotimes_\A\RHom_{\A^e}(U,\A^e)[d]\lotimes_\A\Si\ar[d]^-{1\otimes\varphi^+\otimes1}\\
		\Si\lotimes_\A U^a\lotimes_\A\Si\ar[r]&\Si\lotimes_\A U^{a-1}\lotimes_\A\Si } \]
	commutative by \ref{daiji}. Comparing the mapping cones, we obtain $\RHom_{\Si^e}(\Si,\Si^e)[d+1]\xsimeq\Si_{\geq a-1}$.
\end{proof}

\begin{proof}[Proof of \ref{fakeMM}]
	We have seen above that the correspondence from (ii) to (i) is well-defined, and it is clear that the composite of the correspondences (ii)$\to$(i)$\to$(ii) is the identity.
	
	We prove that given a graded dg category $\Si$ as in (i), then $\A:=\Si_0$ and $U:=\Si_1$ form a cyclically invariant pair as in (ii), and $\Si=\rT^{\rL}_\A\!U$. Since $\Si$ is a positively graded smooth dg category, we have $\Si\in\per^{\geq0}\!\Si^e$, thus $\Si^\vee\in\per^{\leq0}\!\Si^e$. On the other hand, $\Si^\vee$ is a perfect bimodule concentrated in (Adams) degree $\geq-1$ by $\Si_{\geq a-1}(a)\simeq\Si^\vee[d+1]$, so $\Si^\vee\in\per^{\geq1}\!\Si^e$. It follows that $\Si^\vee\in\per^{[-1,0]}\!\Si^e$, hence $\Si\in\per^{[0,1]}\!\Si^e$. Then \ref{chtensor} applies and we get a graded quasi-equivalence $\rT^{\rL}_\A\!U\xsimeq\Si$.
	
	It remains to prove that $U$ is a cyclically invariant root of the shifted inverse dualizing complex. Since $\Si$ is a tensor category, one has the standard triangle $\Si\lotimes_\A U\lotimes_\A \Si(-1)\to\Si\lotimes_\A\Si\to \Si$ in $\rD^\Z(\Si^e)$. Dualizing we get a triangle
	\[ \xymatrix{ \Si\lotimes_\A\RHom_{\A^e}(\A,\A^e)[d]\lotimes_\A \Si(-1)\ar[r]&\Si\lotimes_\A\RHom_{\A^e}(U,\A^e)[d]\lotimes_\A\Si\ar[r]& \RHom_{\Si^e}(\Si,\Si^e)(-1)[d+1] } \]
	in $\rD^\Z(\Si^e)$. On the other hand, the standard triangle as in \ref{ten} gives
	\[ \xymatrix{ \Si\lotimes_\A U^{a}\lotimes_\A\Si(-1)\ar[r]&\Si\lotimes_\A U^{a-1}\lotimes_\A\Si\ar[r]&\Si_{\geq a-1}(a-1) } \]
	in $\rD^\Z(\Si^e)$. Notice that in each of the triangles, the leftmost term is generated in degree $1$, the middle term is generated in degree $0$, and the two rightmost terms are isomorphic by assumption. Therefore these are truncations (of the rightmost terms) with respect to the stable $t$-structure $\per^{\leq0}\!\Si^e\perp\per^{>0}\!\Si^e$ on $\per^\Z\!\Si^e$, hence the above two triangles are isomorphic. Now, since the isomorphism on the leftmost term (resp. on the middle term) has (Adams) degree $0$, it must be of the form $1\otimes\varphi\otimes1$ for some isomorphism $\varphi\colon\RHom_{\A^e}(\A,\A^e)[d]\xsimeq U^a$ (resp. $1\otimes\psi\otimes1$ for some isomorphism $\psi\colon\RHom_{\A^e}(U,\A^e)[d]\xsimeq U^{a-1}$). Indeed, any $\Si^e$-linear morphism $\Si\lotimes_\A X\lotimes_\A\Si\to\Si\lotimes_\A Y\lotimes_\A\Si$ of degree $0$ is induced from an $\A^e$-linear $X\to Y$ by $\RHom_{\Si^e}^\Z(\Si\lotimes_\A X\lotimes_\A\Si,\Si\lotimes_\A Y\lotimes_\A\Si)=\RHom_{\A^e}(X,(\Si\lotimes_\A Y\lotimes_\A\Si)_0)=\RHom_{\A^e}(X,Y)$. We see that $U$ is an $a$-th root of $\RHom_{\A^e}(\A,\A^e)[d]$, and that $\varphi$ and $\psi$ are compatible in the sense of \ref{compat}, hence $U$ is cyclically invariant by \ref{daiji}.
\end{proof}

\subsection{Proof of \ref{CY}}
We are almost ready to prove the first main result \ref{CY}.
Let $(U,\P)$ be an $a$-th root pair on a smooth dg category $\A$, $\Si=\rT^{\rL}_\A\!U$ the derived tensor category, and $\Pi=\Pi^{(1/a)}_{d+1}(\A)$ the Calabi-Yau completion. For a $\Si$-module $X$ and $0\leq i\leq a-1$, we denote by $Xe_i$ the restriction of $X$ to $\Pi^{(i)}$, and similarly for the left modules.
\begin{Lem}
$\Pi$ is smooth.
\end{Lem}
\begin{proof}
	The restriction of the standard triangle $\Si\lotimes_\A U\lotimes_\A\Si\to \Si\lotimes_\A\Si\to \Si$ along $e_0(-)e_{a-1}$ yields
	\[ \xymatrix{ e_0\Si\lotimes_\A U\lotimes_\A\Si e_{a-1}\ar[r]&e_0\Si\lotimes_\A \A\lotimes_\A\Si e_{a-1}\ar[r]&e_0\Si e_{a-1} }. \]
	In view of \ref{setup}(\ref{exc}) we know that $e_0\Si e_{a-1}=\Pi$ in $\rD(\Pi^e)$, and we also have $e_0\Si=\bigoplus_{i=0}^{a-1}e_0\Si e_i=\Pi^{\oplus a}$ as left $\Pi$-modules, and similarly $\Si e_{a-1}=\bigoplus_{i=0}^{a-1}e_i\Si e_{a-1}=\Pi^{\oplus a}$ as right $\Pi$-modules, in particular $e_0\Si\in\per\Pi^\op$ and $\Si e_{a-1}\in\per\Pi$. We conclude from the above triangle that $\Pi\in\per\Pi^e$.
\end{proof}

We prepare another important lemma in the setting of root pairs.
\begin{Lem}\label{Pi-dual}
There are canonical isomorphisms
\[ \RHom_\Pi(\Si e_{a-1},\Pi)\simeq {e}_0\Si,\qquad \RHom_{\Pi^\op}({e}_0\Si,\Pi)\simeq\Si e_{a-1} \]
in $\rD(\Si^\op\otimes\Pi)$ and in $\rD(\Pi^\op\otimes\Si)$, respectively.
\end{Lem}
\begin{proof}
	There is a canonical multiplication map $e_0\Si\to\RHom_{\Pi}(\Si e_{a-1},e_0\Si e_{a-1})$ in $\rD(\Pi^\op\otimes\Si)$, and we show that this is a quasi-isomorphism. For this we restrict each side to $\P^{(i)}$ and verify that $e_0\Si e_i\to\RHom_{\Pi}(e_i\Si e_{a-1},e_0\Si e_{a-1})$ is a quasi-isomorphism of complexes. This follows from \ref{setup}(\ref{exc}) since it yields $e_i\Si e_j\simeq \Pi$ whenever $i\leq j$. The second isomorphism is proved in a similar way.
\end{proof}

Now we are ready to prove \ref{CY} using that $(U,\P)$ is cyclically invariant.
\begin{proof}[Proof of \ref{CY}]
	We restrict the isomorphism
	\[ \RHom_{\Si^e}(\Si,\Si\otimes\Si)[d+1]\simeq\Si_{\geq a-1} \]
	in $\rD(\Si^e)$ from \ref{fakeCY} as $e_0(-)e_{a-1}$ to $\rD(\Pi^e)$. By \ref{setup}(\ref{exc}) the right-hand-side becomes quasi-isomorphic to $\Pi$ since $e_0U^le_{a-1}=0$ for $0\leq l<a-1$. On the other hand, (the $[-d-1]$ of) the left-hand-side restricts to $\RHom_{\Si^e}(\Si,\Si e_{a-1}\otimes e_0\Si)$, which is isomorphic by \ref{Pi-dual} to $\RHom_{\Si^e}(\Si,\RHom_{\Pi^\op}(e_0\Si,\Pi)\otimes\RHom_\Pi(\Si e_{a-1},\Pi))=\RHom_{\Si^e}(\Si,\RHom_{\Pi^e}(e_0\Si\otimes\Si e_{a-1},\Pi^e))$. By adjunction this is isomorphic to $\RHom_{\Pi^e}(e_0\Si e_{a-1},\Pi^e)=\RHom_{\Pi^e}(\Pi,\Pi^e)$, again by \ref{setup}(\ref{exc}).
\end{proof}

\subsection{Proof of \ref{MM}}\label{MMproof}
The correspondence from (\ref{smdg}) to (\ref{CYdgalg}) is a graded version of \ref{CY}, together with the following observation on the degree $0$ part.
\begin{Prop}\label{Pi_0}
Let $(U,\P)$ be a root pair and $\Pi$ its Calabi-Yau completion. Then $\Pi_0\in\per^\Z\!\Pi$ and $\Pi_0\in\per^\Z\!\Pi^\op$.
\end{Prop}
\begin{proof}
	We only consider the right module. Since $\Si=\rT^\rL_\A\!U$ is a tensor category we have a triangle $U\lotimes_\A\Si(-1)\to\Si\to\A$ in $\rD^\Z(\A^\op\otimes\Si)$. Restricting along $e_0(-)e_{0}$ we get
	\[ \xymatrix{ e_0U\lotimes_\A\Si e_{0}(-1)\ar[r]& \Pi\ar[r]&\Pi_0 } \]
	in $\rD^\Z(\Pi)$, which gives the assertion by $e_0U\in\per\A$ and $\Si e_{0}\in\per^\Z\!\Pi$.
\end{proof}

The rest of this section is devoted to establishing the converse correspondence.
We shall make use of two stable $t$-structures on the graded derived categories of graded dg categories. Let $\X$ be a positively graded dg category.

(1)  The first one is what we already noted above \ref{chtensor}. Recall that we write, for example, $\per^I\!\X=\thick\{ \X(-,A)(-i)\mid A\in\X,\, i\in I\}$ for a subset $I\subset\Z$. 
Then the graded perfect derived category $\per^\Z\!\X$ has a stable $t$-structure
\[ \per^\Z\!\X=\per^{\leq l}\!\X\perp\per^{\geq l+1}\!\X \]
for each $l\in\Z$.

(2)  The second one is a standard one with respect to the grading: let
\[ 
\begin{aligned}
	\rD^\Z(\X)_{\geq0}&=\{M\in\rD^\Z(\X)\mid M \text{ is concentrated in degree}\geq0 \},\\
	\rD^\Z(\X)_{<0}&=\{M\in\rD^\Z(\X)\mid M \text{ is concentrated in degree}<0 \}.
\end{aligned}
\]
Then $(\rD^\Z(\X)_{\geq0},\rD^\Z(\X)_{<0})$ is a stable $t$-structure on $\rD^\Z(\X)$. There are obvious variants for truncation at other degrees.

\medskip

Let $\Pi$ be as in \ref{MM}(\ref{CYdgalg}), thus it is an Adams positively graded $(d+1)$-Calabi-Yau dg category of Gorenstein parameter $a$ such that $\Pi_0\in\per^\Z\!\Pi$ and $\Pi_0\in\per^\Z\!\Pi^\op$. We put $T=\Pi\oplus\Pi(1)\oplus\cdots\oplus\Pi(a-1)\in\per^\Z\!\Pi$ and define the dg category $\A$ and a bimodule $U$ by
\[ \A=\REnd_\Pi^\Z(T)=\begin{pmatrix}\Pi_0&0&\cdots&0\\ \Pi_1&\Pi_0&\cdots&0\\ \vdots&\vdots&\ddots&\vdots\\ \Pi_{a-1}&\Pi_{a-2}&\cdots&\Pi_0\end{pmatrix}, \quad U=\RHom_\Pi^\Z(T,T(1))=\begin{pmatrix}	\Pi_{1} & \Pi_0&\cdots& 0\\ \vdots& \vdots&\ddots&\vdots \\ \Pi_{a-1}&\Pi_{a-2}&\cdots&\Pi_0 \\ \Pi_{a}&\Pi_{a-1}&\cdots&\Pi_{1}\end{pmatrix}. \]
Formally, the objects of $\A$ are $a$ copies of those of $\Pi$; we denote by $\P^{(i)}$ the full subcategory of $\A$ corresponding to the $i$-th copy for $0\leq i\leq a-1$. The morphisms of $\A$ are given by $\Pi_{j-i}(A,B)$ if $A\in\P^{(i)}$ and $B\in\P^{(j)}$ (thus is $0$ when $i>j$). Similarly, the $(\A,\A)$-bimodule $U$ takes the value $\Pi_{j-i+1}(A,B)$ at $A\in\P^{(i)}$ and $B\in\P^{(j)}$.


Let us prepare some necessary materials.
\begin{Lem}\label{power}
There is an isomorphism $U^{\lotimes_\A l}\simeq\RHom_\Pi^\Z(T,T(l))$ in $\rD(\A^e)$ for every $l\geq0$.
\end{Lem}
\begin{proof}
	By assumption we have $\Pi_0\in\per^\Z\!\Pi$. Let 
	\[ \rC^\Z(\Pi):=\per^\Z\!\Pi/\thick\{ \Pi_0(i) \mid i\in\Z\} \]
	be the {graded cluster category} as in Appendix \ref{A}. Then by \ref{B}(\ref{U}) there is a commutative diagram
	\[ \xymatrix@R=4mm@!C=14mm{
		\rC^\Z(\Pi)\ar@{-}[r]^-\simeq\ar[d]_-{(1)}&\per\A\ar[d]^-{-\lotimes_\A U}\\
		\rC^\Z(\Pi)\ar@{-}[r]^-\simeq&\per\A } \]
	of equivalences, and we deduce that $U^{\lotimes_\A l}=\RHom^\Z_\C(T,T(l))$ for the canonical (graded) dg enhancement $\C$ of $\rC^\Z(\Pi)$. By \ref{F} the canonical map $\RHom_\Pi^\Z(T,T(l))\to\RHom_\C^\Z(T,T(l))$ is a quasi-isomorphism for each $l\geq0$, which proves the assertion.
\end{proof}

In what follows we take a cofibrant resolution of $U$ over $\A^e$, and still denote it by $U$. Let 
\[ \Si=\rT_{\A}\!U \]
be the (derived) tensor category which carries a natural tensor grading. We show that there is an isomorphism as in \ref{fakeCY}, which will yield  by \ref{fakeMM} that $(U,\P^{(0)})$ forms a cyclically invariant root pair.
We denote by $\Pi^{(i)}$ the full subcategory of $\Si$ on objects of $\P^{(i)}$. By \ref{power} each $\Pi^{(i)}$ is isomorphic as a graded dg category to the Calabi-Yau dg category $\Pi$. We represent
\[ \Pi^{(0)}=\Pi, \]
so that there is an inclusion $\Pi\to\Si$ as a subcategory.

We track the morphism $\HH_\ast(\Pi)\to\HH_\ast(\Si)$ on the Hochschild homologies induced by this inclusion. Recall that for a smooth dg category $\X$ and $M\in\rD(\X^e)$ we have a canonical isomorphism
\[ \xymatrix{M\lotimes_{\X^e}\X\ar[r]&\RHom_{\X^e}(\X^\vee,M) }. \]
In particular, the map on Hochschild homologies yields a morphism
\[ \xymatrix{\RHom_{\Pi^e}(\Pi^\vee,\Pi)\ar[r]&\RHom_{\Si^e}(\Si^\vee,\Si) }. \]
Note that $\Pi$ and $\Si$ are graded, and the above morphism preserves the grading.
\begin{Lem}\label{geqa}
The above morphism restricts to
\[ \xymatrix{\RHom_{\Pi^e}(\Pi^\vee,\Pi)_{\geq a}\ar[r]&\RHom_{\Si^e}(\Si^\vee,\Si_{\geq a-1}) }. \]
In particular, any isomorphism $\Pi^\vee[d+1]\to\Pi(a)$ in $\rD^\Z(\Pi^e)$ gives rise to a morphism $\Si^\vee[d+1]\to\Si_{\geq a-1}(a)$ in $\rD^\Z(\Si^e)$.
\end{Lem}
\begin{proof}
	Consider the diagram
	\[ \xymatrix{
		\RHom_{\Pi^e}(\Pi^\vee,\Pi)_{\geq a}\ar[r]&\RHom_{\Pi^e}(\Pi^\vee,\Pi)\ar[r]\ar[d]&\RHom_{\Pi^e}(\Pi^\vee,\Pi)_{< a}\\
		\RHom_{\Si^e}(\Si^\vee,\Si_{\geq a-1})\ar[r]&\RHom_{\Si^e}(\Si^\vee,\Si)\ar[r]&\RHom_{\Si^e}(\Si^\vee,\Si_{< a-1}). } \]
	We claim that there is no morphism from the upper left corner to the lower right corner, which yields the desired morphism. Since the upper left corner is concentrated in (Adams) degree $\geq a$ it is enough to show that the lower right corner is concentrated in degree $<a$. Note that $\RHom_{\Si^e}(\Si^\vee,\Si_{< a-1})=\Si_{<a-1}\lotimes_{\Si^e}\Si$ can be computed by the bimodule resolution $\Si\lotimes_\A U\lotimes_\A\Si(-1)\to\Si\lotimes_\A\Si\to\Si$, which yields a triangle
	\[ \xymatrix{ \Si_{<a-1}\lotimes_{\A^e}U(-1)\ar[r]&\Si_{<a-1}\lotimes_{\A^e}\A\ar[r]&\Si_{<a-1}\lotimes_{\Si^e}\Si } \]
	in $\rD^\Z(\Mod k)$. Since $\Si_{<a-1}$ is concentrated in degree $[0,a-1)$, this triangle shows that the rightmost term is concentrated in degree $[0,a-1]$, as desired.
\end{proof}
In fact, the above morphism restricts to one between isomorphisms.
\begin{Lem}\label{well}
The map $\HH_\ast(\Pi)\to\HH_\ast(\Si)$ yields a map
\[ \xymatrix{ \Hom^\Z_{\rD(\Pi^e)}(\Pi^\vee[d+1],\Pi(a))^\times\ar[r]&\Hom^\Z_{\rD(\Si^e)}(\Si^\vee[d+1],\Si_{\geq a-1}(a))^\times }. \]
Here $\Hom_\rC(-,-)^\times$ means the set of isomorphisms in a category $\rC$.
\end{Lem}
The rest of this subsection is devoted to the proof of this lemma. 

As in the previous section, we use the notation, for example, $Xe_i$ for the restriction of a $\Si$-module $X$ to $\P^{(i)}$, which we view as a $\Pi$-module.
Also as before we make use of the two bimodules, $e_0\Si\in\rD^\Z(\Pi^\op\otimes\Si)$ and $\Si e_{a-1}\in\rD^\Z(\Si^\op\otimes\Pi)$.
In view of the grading we have $e_0\Si\simeq\Pi\oplus\Pi(-1)\oplus\cdots\oplus\Pi(-a+1)$ as left $\Pi$-modules and $\Si e_{a-1}\simeq\Pi(-a+1)\oplus\cdots\oplus\Pi(-1)\oplus\Pi$ as right $\Pi$-modules. This gives the following graded refinement of \ref{Pi-dual}.
\begin{Lem}\label{gr-Pi-dual}
	There exist isomorphisms
	\[ \xymatrix{ \Si e_{a-1}(a-1)\ar[r]&\RHom_{\Pi^\op}(e_0\Si,\Pi),& e_0\Si (a-1)\ar[r]&\RHom_\Pi(\Si e_{a-1},\Pi) } \]
	of in $\rD^\Z(\Si^\op\otimes\Pi)$ {\rm (}resp. in $\rD^\Z(\Pi^\op\otimes\Si)${\rm)}.
\end{Lem}

Now we are ready to prove \ref{well}.
\begin{proof}[Proof of \ref{well}]
	We need to track the morphism $\HH_\ast(\Pi)\to\HH_\ast(\Si)$ in detail. Recalling that $\Pi=e_0\Si e_0$, the image of $f\colon\Pi^\vee(-a)[d+1]\to\Pi$ under the morphism $\RHom_{\Pi^e}(\Pi^\vee,\Pi)\to\RHom_{\Si^e}(\Si^\vee,\Si)$ is given by the composition along the following diagram.
	\begin{equation}\label{eqfg}
	\xymatrix@R=3mm@C=15mm{
		\Si e_0\lotimes_\Pi\RHom_{\Pi^e}(\Pi,\Pi^e)(-a)[d+1]\lotimes_\Pi e_0\Si\ar[r]^-{1\otimes f\otimes 1}& \Si e_0\lotimes_\Pi e_0\Si\ar[d]\\
		\RHom_{\Pi^e}(\Pi, e_0\Si\otimes \Si e_0)(-a)[d+1]\ar@{=}[u]&\Si\\
		\RHom_{\Si^e}(\Si e_0\lotimes_\Pi e_0\Si,\Si^e)(-a)[d+1]\ar@{=}[u]&\\
		\RHom_{\Si^e}(\Si,\Si^e)(-a)[d+1]\ar[u]&&, }
	\end{equation}
	where the first and last maps are induced by the multiplication map $\Si e_0\lotimes_\Pi e_0\Si\to \Si$. By \ref{geqa}, for any $f$ of Adams degree $0$, the induced map factors through $\Si_{\geq a-1}$.

	
	Now let $f\colon\Pi^\vee(-1)[d+1]\to\Pi(a-1)$ be an isomorphism and let $g\colon\Si^\vee(-1)[d+1]\to\Si_{\geq a-1}(a-1)$ be its image viewed as a map to $\Si_{\geq a-1}$. We have to show that this is an isomorphism in $\rD^\Z(\Si^e)$. 
	
	We first claim that the restriction of $g$ along $e_0(-)e_{a-1}$ is an isomorphism. It is enough to check that this is the case for the first map $\RHom_{\Si^e}(\Si,\Si^e)\to\RHom_{\Si^e}(\Si e_0\lotimes_\Pi e_0\Si,\Si^e)$ in (\ref{eqfg}) since every other map as well as the inclusion $\Si_{\geq a-1}\to\Si$ evidently restrict to isomorphisms under $e_0(-)e_{a-1}$. Let $\underline{\Si}$ be the mapping cone of $\Si e_0\lotimes_\Pi e_0\Si\to\Si$ in $\rD^\Z(\Si^e)$. Then $e_0\RHom_{\Si^e}(\underline{\Si},\Si^e)e_{a-1}=\RHom_{\Si^e}(\underline{\Si},\Si e_{a-1}\otimes e_0\Si)$ which is isomorphic by \ref{Pi-dual} to $\RHom_{\Si^e}(\underline{\Si},\RHom_{\Pi^e}(e_0\Si\otimes \Si e_{a-1},\Pi^e))$, thus to $\RHom_{\Pi^e}(e_0\underline{\Si}e_{a-1},\Pi^e)$ by adjunction, hence this is $0$ since $\underline{\Si}$ is annihilated by $e_0$. We conclude that $\RHom_{\Si^e}(\Si,\Si e_{a-1}\otimes e_0\Si)\to\RHom_{\Si^e}(\Si e_0\lotimes_\Pi e_0\Si,\Si e_{a-1}\otimes e_0\Si)$ is an isomorphism.
	
	We next prove that for each $0\leq i,j\leq a-1$ the restriction along $e_i(-)e_{a-1-j}$ is an isomorphism, which will prove that $g$ is indeed an isomorphism. We prepare a piece of notation: for $l\geq0$ we denote by $\Si_{\geq-l}$ the $(\Si,\Si)$-bimodule $\RHom_{\Si}(\Si_{\geq l},\Si)$. Being the inverse of $\Si_{\geq l}$, it is also isomorphic to $\RHom_{\Si^\op}(\Si_{\geq l},\Si)$.
	
	Note first that since $e_i\Si=e_0\Si_{\geq i}(i)$ as graded $(\Pi,\Si)$-bimodules, restricting on the left to $e_i$ amounts to applying $e_0\Si_{\geq i}(i)\lotimes_\Si-$. Similarly, restriction on the right to $e_{a-1-j}$ is isomorphic to $-\lotimes_\Si\Si_{\geq j}e_{a-1}(j)$. It follows that the restriction of the codomain of $g$ is $e_i(\Si_{\geq a-1}(a-1))e_{a-1-j}=e_0(\Si_{\geq a-1+i+j}(a-1+i+j))e_{a-1}$, which is $\Pi_{\geq i+j}(i+j)$.
	On the other hand, recalling that the bimodule action of $\Si$ on $\Si^\vee=\RHom_{\Si^e}(\Si,\Si\otimes\Si)$ comes from the inner bimodule structure on $\Si\otimes\Si$, one can compute the restriction of the domain of $g$ as $e_i\RHom_{\Si^e}(\Si,\Si\otimes\Si(-1)[d+1])e_{a-1-j}=\RHom_{\Si^e}(\Si,\Si_{\geq j}e_{a-1}(j)\otimes e_0\Si_{\geq i}(i))(-1)[d+1]=\RHom_{\Si^e}(\Si_{\geq-i-j},\Si e_{a-1}\otimes e_0\Si)(i+j-1)[d+1]$, which is isomorphic to $\RHom_{\Pi^e}(e_0\Si_{\geq-i-j}e_{a-1},\Pi^e)(-2a+i+j+1)[d+1]$ by \ref{gr-Pi-dual} and adjunction.
	
	Now, we want to show that $\RHom_{\Pi^e}(e_0\Si_{\geq-i-j}e_{a-1},\Pi^e)(-2a+1)[d+1]\to \Pi_{\geq i+j}$ is an isomorphism, where we know that this is already the case for $i=j=0$. Note that these morphisms fit into the commutative diagram below for $l:=i+j$.
	\[ \xymatrix@R=3mm{
		\RHom_{\Pi^e}(e_0\Si_{\geq-l}e_{a-1},\Pi^e)(-2a+1)[d+1]\ar[r]\ar[d]& \Pi_{\geq l}\ar[d]\\
		\RHom_{\Pi^e}(e_0\Si e_{a-1},\Pi^e)(-2a+1)[d+1]\ar[r]^-\simeq\ar[d]& \Pi\ar[d]\\
		\RHom_{\Pi^e}(e_0\Si_{[-l,-1]}e_{a-1},\Pi^e)(-2a+1)[d+2]& \Pi_{< l}. } \]
	Therefore it is enough to show that the left triangle is also the truncation with respect to the grading.
	
	We first show that $\RHom_{\Pi^e}(e_0\Si_{\geq-l}e_{a-1},\Pi^e)(-2a+1)$ is concentrated in degree $\geq l$. Note first that the triangle in \ref{ten} generalizes to negative powers, that is, we have a triangle
	\[ \xymatrix{ \Si\lotimes_\A U^{-l+1}\lotimes_\A\Si(-1)\ar[r]&\Si\lotimes_\A U^{-l}\lotimes_\A\Si\ar[r]&\Si_{\geq-l}(-l) } \]
	simply by the same proof as in \ref{ten}. It follows that $\Si_{\geq-l}\in\per^{[-l,-l+1]}\!\Si^e$. Now we apply $e_0(-)e_{a-1}$. Since we have $e_0\Si\simeq\Pi\oplus\Pi(-1)\oplus\cdots\oplus\Pi(-a+1)$ as left $\Pi$-modules and $\Si e_{a-1}\simeq\Pi(-a+1)\oplus\cdots\oplus\Pi(-1)\oplus\Pi$ as right $\Pi$-modules, the restriction $e_0(-)e_{a-1}\colon\per^\Z\!\Si^e\to\per^\Z\!\Pi^e$ takes $\per^{[s,t]}\!\Si^e$ to $\per^{[s,t+2a-2]}\!\Pi^e$. In particular, we see that $e_0\Si_{\geq-l}e_{a-1}\in\per^{[-l,-l+2a-1]}\!\Pi^e$, hence its bimodule dual is in $\per^{[l-2a+1,l]}\!\Pi^e$. We conclude that $\RHom_{\Pi^e}(e_0\Si_{\geq-l}e_{a-1},\Pi^e)(-2a+1)\in\per^{[l,l+2a-1]}\!\Pi^e$ is concentrated in degree $\geq l$.
	
	We next show that $\RHom_{\Pi^e}(e_0\Si_{[-l,-1]}e_{a-1},\Pi^e)(-2a+1)$ is concentrated in degree $<l$. We claim that for each perfect bimodule over $\A=\Si_0$, the complex $\RHom_{\Pi^e}(e_0Me_{a-1},\Pi^e)$ is concentrated in Adams degree $-2a$, which will certainly prove the assertion.
	We have $\RHom_{\Pi^e}(e_0Me_{a-1},\Pi^e)=\RHom_{\Pi^e}(M\lotimes_{\Si^e}(e_0\Si\otimes\Si e_{a-1}),\Pi^e)$, which is isomorphic by adjunction to $\RHom_{\Si^e}(M,\RHom_{\Pi^e}(e_0\Si\otimes\Si e_{a-1},\Pi^e))$, thus by \ref{gr-Pi-dual} to $\RHom_{\Si^e}(M,\Si e_{a-1}\otimes e_0\Si)(2a-2)$. It remains to show that $\RHom_{\Si^e}(M,\Si^e)$ is concentrated in degree $2$. Note that we may assume $M=\Si_0^e$, thus the complex in question becomes $\RHom_{\Si^\op}(\Si_0,\Si)\otimes\RHom_{\Si}(\Si_0,\Si)$. Now by the triangle
	\[ \xymatrix{ U\lotimes_\A\Si(-1)\ar[r]&\Si\ar[r]&\Si_0}, \]
	in $\rD(\A^\op\otimes\Si)$ and its opposite, we easily get $\RHom_{\Si}(\Si_0,\Si)=\RHom_\A(U,\A)(1)[-1]$ and $\RHom_{\Si^\op}(\Si_0,\Si)=\RHom_{\A^\op}(U,\A)(1)[-1]$. Therefore the claim follows, which completes the proof.
\end{proof}

Now that we have an isomorphism $\Si^\vee[d+1]\simeq\Si_{\geq a-1}(a)$, we are ready to prove \ref{MM}.
\begin{proof}[Proof of \ref{MM}]
	We know from \ref{CY} that cyclically invariant root pairs yield $(d+1)$-Calabi-Yau algebras. Let us discuss the grading. We know by \ref{Pi_0} that the degree $0$ part is perfect on each side. By tracking the degree of the isomorphism in \ref{fakeCY}, we see that its graded version is
	\[ \RHom_{\Si^e}(\Si,\Si^e)[d+1]\simeq\Si_{\geq a-1}(a). \]
	In view of the isomorphisms \ref{gr-Pi-dual} and the proof of \ref{CY}, we deduce the desired isomorphism
	\[ \RHom_{\Pi^e}(\Pi,\Pi^e)[d+1]\simeq\Pi(a), \]
	that is, $\Pi$ has Gorenstein parameter $a$.
	
	Suppose conversely we are given a $(d+1)$-Calabi-Yau dg category of Gorenstein parameter $a$. We define $\A$, $U$, and $\P$ as in \ref{MM}.
	Then \ref{well} yields an isomorphism $\Si^\vee[d+1]\simeq\Si_{\geq a-1}(a)$ in $\rD^\Z(\Si^e)$, so \ref{fakeMM} shows that $U$ is a cyclically invariant $a$-th root of $\RHom_{\A^e}(\A,\A^e)[d]$. Finally, it is clear from the descriptions of $\A$ and $U$ as triangular matrices that there is a semi-orthogonal decomposition $\per\A=\U_{a-1}\perp\cdots\perp\U_1\perp\U_0$ with $\U_i=\thick\{ P\lotimes_\A U^i\mid P\in\P\}\subset\per\A$, so that $(U,\P)$ forms a root pair. 
\end{proof}

\section{Folded cluster categories}
Having constructed a Calabi-Yau dg algebra, we can apply the construction of the cluster category of Amiot \cite{Am09} which yields a Calabi-Yau triangulated category with a cluster tilting object. The aim of this section is to demonstrate that our construction of Calabi-Yau dg algebras from root pairs has a nice description of the cluster categories.

Recall that a dg algebra $A$ is called {\it connective} if $H^iA=0$ for all $i>0$. We denote by $\Db(A)$ the full subcategory of $\rD(A)$ consisting of dg modules of finite dimensional total cohomology.
\begin{Thm}[{\cite{Am09}}]\label{Am}
	Let $\Pi$ be a connective $(d+1)$-Calabi-Yau dg algebra such that $H^0\Pi$ is finite dimensional. Then
	\[ \rC(\Pi):=\per\Pi/\Db(\Pi) \]
	is $d$-Calabi-Yau, and the image of $\Pi\in\per\Pi$ in $\rC(\Pi)$ is a $d$-cluster tilting object. We call $\rC(\Pi)$ the {\rm cluster category} of $\Pi$.
\end{Thm}

It is an important problem to understand the structure of such cluster categories. When $\Pi$ is the (ordinary) Calabi-Yau completion of a finite dimensional algebra $A$, then $\rC(\Pi)$ is the triangulated hull of the orbit category $\Db(\mod A)/-\lotimes_ADA[-d]$. We give the following version for root pairs, which thereby establishes the root version of the cluster category.
Recall that a smooth, proper dg algebra $A$ is {\it $\nu_d$-finite} if for each $X,Y\in\per A$ one has $\Hom_{\rD(A)}(X,\nu_d^iY)=0$ for almost all $i\in\Z$, where $\nu_d=-\lotimes_ADA[-d]$. Also, for a triangulated category $\T$ and an autoequivalence $F$ on $\T$ (with fixed enhancements), we denote by $(\T/F)_\triangle$ the triangulated hull \cite{Ke05} of the orbit category $\T/F$.
\begin{Thm}\label{rc}
Let $A$ be a smooth, proper, and connective dg algebra which is $\nu_d$-finite, let $(U,P)$ be a cyclically invariant $a$-th root pair of $\RHom_{A^e}(A,A^e)[d]$ such that $U$ is concentrated in cohomological degree $\leq0$, and let $\Pi=\Pi_{d+1}^{(1/a)}(A)$ be the $(d+1)$-Calabi-Yau completion of $(U,P)$.
\begin{enumerate}
	\item\label{O} If $e\in A$ is the idempotent of $A$ corresponding to $P$, then the functor $-\lotimes_A(\rT^{\rL}_A\!U) e$ induces an equivalence $(\per A/-\lotimes_AU)_\triangle\xsimeq\rC(\Pi)$.
	\item\label{CTs} $P\in(\per A/-\lotimes_AU)_\triangle$ is a $d$-cluster tilting object, and $\add\{ P\lotimes_AU^i\mid i\in\Z\}\subset\per A$ is a $d$-cluster tilting subcategory.
\end{enumerate}
\end{Thm}

The statement (\ref{O}) is an analogue of \cite{Am09,Ke05}, and (\ref{CTs}) is an analogue of \cite{BMRRT,Am09,Iy11}.
\begin{Def}
In the setting of \ref{rc}, we call $\rC(\Pi_{d+1}^{(1/a)}(A))$ the {\it $a$-folded $d$-cluster category} of $A$, and denote it by $\rC_d^{(1/a)}(A)$.
\end{Def}
Note that $\rC_d^{(1/a)}(A)$ depends on $U$, so the notation contains an ambiguity. Since $U^a\simeq\RHom_{A^e}(A,A^e)[d]$, the category $\rC_d^{(1/a)}(A)=(\per A/-\lotimes_AU)_\triangle$ is the triangulated hull of $\rC_d(A)/-\lotimes_AU$, the $\Z/a\Z$-action $-\lotimes_AU$. Moreover, thanks to the cyclic invariance of $U$, the $a$-folded $d$-cluster category defined here is $d$-Calabi-Yau, compare \cite{ha3}, and contains a $d$-cluster tilting object, compare \cite{haI}.

Our proof is based on the discussion of graded cluster categories in Appendix \ref{A}. Let $\Si=\rT^\rL_A\!U$ be the derived tensor algebra and let $e\in A$ be the idempotent of $A$ corresponding to $P$ so that $\Pi=e\Si e$. 
\begin{Prop}\label{ExB}
Let $\rC(\Si):=\per\Si/\Db(\Si)$ and $\rC(\Pi)=\per\Pi/\Db(\Pi)$ be cluster categories.
\begin{enumerate}
\item We have an equivalence $-\lotimes_A\Si\colon(\per A/-\lotimes_AU)_\triangle\to\rC(\Si)$.
\item There is an $(A,\Pi)$-bimodule $T$ inducing an equivalence $-\lotimes_AT\colon(\per A/-\lotimes_AU)_\triangle\to\rC(\Pi)$.
\item The diagram
\[ \xymatrix@C=16mm@!R=3mm{
	(\per A/-\lotimes_AU)_\triangle\ar[r]^-{-\lotimes_A\Si}\ar[dr]_-{-\lotimes_AT}&\rC(\Si)\ar[d]^-{(-)e}\\
	&\rC(\Pi) } \]
is commutative. Therefore, the vertical functor is an equivalence.
\end{enumerate}
\end{Prop}
\begin{proof}
	(1)  We apply \ref{B} to the graded dg algebra $\Si=\rT^\rL_A\!U$ (with the tensor grading). We claim that $\Si$ has Gorenstein parameter $1$ in the sense of \ref{GP}. By \ref{fakeMM} there is an isomorphism $\RHom_{\Si^e}(\Si,\Si^e)[d+1]\simeq\Si_{\geq a-1}(a)$ in $\rD^\Z(\Si^e)$. Then we have $D\RHom_\Si(A,\Si)=\RHom_\Si(\RHom_{\Si^e}(\Si,\Si^e),A)$ by Serre duality, and using the above isomorphism, the right-hand-side is isomorphic to $\RHom_\Si(\Si_{\geq a-1}(a),A)[d+1]=\RHom_\Si(U^{a-1}\lotimes_A\Si(1),A)[d+1]=\RHom_A(U^{a-1},A)(-1)[d+1]$, which concentrates in Adams degree $1$. This proves the claim, and the assertion follows from \ref{B}(3).
	
	(2)  By \ref{MM}, the algebra $\Pi$ has Gorenstein parameter $a$. Let $T:=\bigoplus_{i=0}^{a-1}\Pi(i)\in\rD^\Z(\Pi)$. Then we have $A=\REnd_\Pi^\Z(T)$ and $U=\RHom_\Pi^\Z(T,T(1))$, and we can view $T$ as an $(A,\Pi)$-bimodule. We obtain the assertion again by \ref{B}(3).
	
	(3)  We construct a morphism $\Si e\to T$ of $(A,\Pi)$-bimodule whose mapping cone has finite dimensional total cohomology, which will prove that the two functors are naturally isomorphic.
	
	By \ref{power} we have an isomorphism $U^{\lotimes_Al}\simeq\RHom_\Pi^\Z(\Pi,\Pi(l))$ for each $l\geq0$, which gives 
	\[ A=\begin{pmatrix}\Pi_0&0&\cdots&0\\ \Pi_1&\Pi_0&\cdots&0\\ \vdots&\vdots&\ddots&\vdots\\ \Pi_{a-1}&\Pi_{a-2}&\cdots&\Pi_0\end{pmatrix}
	\quad\text{ and }\quad
	\Si=\bigoplus_{l\geq0}\begin{pmatrix}\Pi_l&\Pi_{l-1}&\cdots&\Pi_{l-a+1}\\ \Pi_{l+1}&\Pi_l&\cdots&\Pi_{l-a+2}\\ \vdots&\vdots&\ddots&\vdots\\ \Pi_{l+a-1}&\Pi_{l+a-2}&\cdots&\Pi_l\end{pmatrix}. \]
	It follows that we have $\Si e\simeq{}^t\!\begin{pmatrix}\Pi & \Pi_{\geq1} & \cdots &\Pi_{\geq a-1}\end{pmatrix}$ as graded $(A,\Pi)$-bimodules. We then see that $\Si e$ naturally injects into $T(-a+1)={}^t\!\begin{pmatrix}\Pi & \Pi(-1) & \cdots &\Pi(-a+1)\end{pmatrix}$ whose cokernel is certainly finite dimensional.
\end{proof}

Now we are ready to prove \ref{rc}.
\begin{proof}[Proof of \ref{rc}]
	(\ref{O})  This is immediate from \ref{ExB}.
	
	(\ref{CTs})  Noting that it is $P=eA$ that is mapped to $\Pi=e\Si e$ under the equivalence $-\lotimes_A\Si e\colon(\per A/-\lotimes_AU)_\triangle\to\rC(\Pi)$, the first assertion follows from \ref{Am}. Then the second assertion is straightforward.
\end{proof}

We refer to \ref{C^half} and to \ref{typeA} for examples.

\section{Tensor products and $a$-Segre products}
\subsection{Tensor products of root pairs}
We discuss how to reproduce roots of inverse dualizing complexes from given ones by means of tensor products.
Let $\A$ be a dg category and $U$ an $a$-th root of $\RHom_{\A^e}(\A,\A^e)[d]$. Consider another dg category $\B$ and an $a$-th root $V$ of $\RHom_{\B^e}(\B,\B^e)[e]$. We allow $d$ and $e$ to be different, but require $a$ to be the same.
We start our discussion with the following easy observation.
\begin{Prop}\label{UotimesV}
Suppose $\A$ and $\B$ are smooth and let $\C=\A\otimes\B$.
\begin{enumerate}
\item The $(\C,\C)$-bimodule $U\otimes V$ is an $a$-th root of $\RHom_{\C^e}(\C,\C^e)[d+e]$.
\item If $U$ and $V$ are cyclically invariant, then so is $U\otimes V$.
\end{enumerate}
\end{Prop}
\begin{proof}
	Since $\A$ and $\B$ are smooth, the canonical map $\RHom_{\A^e}(\A,\A^e)\otimes\RHom_{\B^e}(\B,\B^e)\to\RHom_{\C^e}(\C,\C^e)$ is an isomorphism in $\rD(\C^e)$. This shows (1). Also, if elements in $U^a\lotimes_{\A^e}\A$ and $V^a\lotimes_{\B^e}\B$ are stable (up to homotopy) under the $\Z/a\Z$-action, then so is the natural corresponding element in $(U\otimes V)^a\lotimes_{\C^e}\C$. This proves (2).
\end{proof}
%

The basic result of this section is the following reproduction of root pairs. Note that we need a root {\it pair} only in one of the factors.
\begin{Thm}\label{roottensor}
Let $\A$ and $\B$ be smooth dg categories, $(U,\P)$ an $a$-th root pair on $\A$, and $V$ an $a$-th root on $\B$. Then $(U\otimes V,\P\otimes\B)$ is an $a$-th root pair on $\A\otimes\B$.
\end{Thm}
\begin{proof}
	We write $\C=\A\otimes\B$ and verify the axioms in \ref{setup}.
	
	(\ref{root})  This follows from \ref{UotimesV}(1). 
	
	(\ref{P})  We have to show that $\T:=\{ (P\otimes B)\lotimes_\C (U\otimes V)^i\mid P\in\P, \, B\in\B, \, 0\leq i\leq a-1\}$ generates $\per\C$. Note first that $(P\otimes B)\lotimes_\C (U\otimes V)^i=(P\lotimes_\A U^i)\otimes(B\lotimes_\B V^i)$. Fix $0\leq i\leq a-1$ and consider the subcategory $\T_i:=\{(P\lotimes_\A U^i)\otimes(B\lotimes_\B V^i)\mid P\in\P, \, B\in\B\}$. Since $\per \B$ is generated by $\{B\lotimes_\B V^i\mid B\in\B\}$, we see by applying $(P\lotimes_\A U^i)\otimes-\colon\per \B\to\per\C$ that $(P\lotimes_\A U^i)\otimes B\in\thick \T_i\subset\per\C$ for each $P\in\P$ and $B\in\B$. Now, since $\per\A$ is generated by $\{ P\lotimes_\A U^i\mid P\in\P, \, 0\leq i\leq a-1\}$, we deduce that $A\otimes B\in\thick\{\T_i\mid 0\leq i\leq a-1\}$ for all $A\in\A$ and $B\in\B$.
	
	(\ref{exc})  This is immediate by $\RHom_\C((P\otimes B)\lotimes_\C(U\otimes V)^i,P^\prime\otimes B^\prime)=\RHom_\C((P\otimes_\A U^i)\otimes(B\lotimes_\B V^i),P^\prime\otimes B^\prime)=\RHom_\A(P\lotimes_\A U^i,P^\prime)\otimes\RHom_\B(B,B^\prime)=0$ for $P,P^\prime\in\P$ and $B,B^\prime\in\B$.
\end{proof}
In what follows we study applications of this result, which allows us to reproduce Calabi-Yau algebras in view of the correspondences in Section \ref{corr}.

Let us first note a consequence of \ref{roottensor}. Let $\X$ and $\Y$ be $\Z$-graded dg categories. Recall that the {\it Segre product} of $\X$ and $\Y$ is the dg category
\[ \X\seg\Y:=\bigoplus_{i\in\Z}\X_i\otimes\Y_i, \]
that is, the (non-full) subcategory of the tensor product $\X\otimes\Y$ with the same objects and morphisms spanned by those of the form $f\otimes g$ with (Adams) $\deg f=\deg g$.
\begin{Cor}\label{Pi and Si}
Let $\Pi$ be a graded $(d+1)$-Calabi-Yau dg category of Gorenstein parameter $a$, and let $\Si$ be a graded dg category which is smooth and satisfies $\RHom_{\Si^e}(\Si,\Si^e)[e+1]\simeq\Si_{\geq a-1}(a)$ in $\rD^\Z(\Si^e)$. Then the Segre product $\Pi\seg\Si$ is $(d+e+1)$-Calabi-Yau of Gorenstein parameter $a$.
\end{Cor}
\begin{proof}
	By \ref{MM} we can write $\Pi$ as the Calabi-Yau completion of a cyclically invariant $a$-th root pair $(U,\P)$ on a dg category $\A$, and by \ref{fakeMM} we can write $\Si$ as the derived tensor category of a cyclically invariant $a$-th root $V$ on a dg category $\B$. By \ref{roottensor} the pair $(U\otimes V,\P\otimes\B)$ is a cyclically invariant $a$-th root pair on $\A\otimes\B$. Now we consider the Calabi-Yau completion of $(U\otimes V,\P\otimes\B)$. It is the restriction of the tensor category $\rT^{\rL}_{\A\otimes\B}(U\otimes V)=\rT^{\rL}_\A\!U\seg\rT^{\rL}_\B\!V$ to $\P\otimes\B$, hence is $\Pi\seg\Si$, therefore we obtain the assertion.	
\end{proof}

Even for a very special choice of $\Pi$, namely the polynomial ring in one variable, we have the following consequence.
Recall also that for an Adams graded dg category $\X=\bigoplus_{i\in\Z}\X_i$ and $n\in\Z$, its {\it $n$-th Veronese subcategory $\X^{(n)}$} is the dg category $\bigoplus_{i\in\Z}\X_{ni}$, that is, it has the same objects as $\X$, and its morphism complexes are the Adams degree $n\Z$ part of $\X$.
\begin{Cor}\label{v}
	Let $\Si$ be an Adams positively graded dg category satisfying $\RHom_{\Si^e}(\Si,\Si^e)[d+1]\simeq\Si_{\geq a-1}(a)$ in $\rD^\Z(\Si^e)$. Then its $a$-th Veronese subcategory $\Pi:=\Si^{(a)}$ is $(d+1)$-Calabi-Yau of Gorenstein parameter $1$.
\end{Cor}
\begin{proof}
	This is a consequence of \ref{Pi and Si} since the $a$-th Veronese subcategory of a dg category is nothing but the Segre product with $\Pi_1(k)=k[x]$, with (Adams) $\deg x=a$.
\end{proof}

\begin{Ex}
Let $A$ be the path algebra of the ($2$-)Kronecker quiver $\xymatrix{\circ\ar@2[r]^-u_-v&\circ}$ and consider the square root $U$ of $\RHom_{A^e}(A,A^e)[1]$ as in \ref{example}, with $n=0$ and $\e=1$ so that it is cyclically invariant. Then let $\Si=\rT^\rL_A\!U$ be the tensor algebra, which is presented by the following graded quiver with commutativity relations.
\[ \xymatrix{ \circ\ar@2[r]^-u_-v&\circ\ar@/_15pt/[l]_-t, & \deg t=1, \,utv=vtu} \]
By \ref{fakeCY} we know that $\Si$ satisfies $\RHom_{\Si^e}(\Si,\Si^e)[2]\simeq\Si_{\geq1}(2)$.

(1)  First consider the second Veronese subring of $\Si$ which is just the Segre product with $k[x]$ with $\deg x=2$. Then it is easy to see that it is presented by the quiver below with commutativity relations.
\[ \xymatrix{ \circ\ar@2[r]^-u_-v&\circ\ar@2@/_18pt/[l]_-{u}^-{v} } \]
We see that $\Si^{(2)}$ is nothing but the ($2$-)preprojective algebra of $A$, which is certainly $2$-Calabi-Yau.

(2)  Next let 
\[ \Pi=k[x,y] \]
with $\deg x=\deg y=1$, which is a $2$-Calabi-Yau algebra of Gorenstein parameter $2$.

Then by \ref{Pi and Si} we see that the Segre product $\Pi\seg\Si$ is $3$-Calabi-Yau. In fact, it is not difficult to see that it is given by the quiver below with commutativity relations.
\[ \xymatrix{ \circ\ar@2[r]^-u_-v&\circ\ar@2@/_18pt/[l]_-{tx}^-{ty} }, \]
which is indeed $3$-Calabi-Yau. (This is a non-commutative crepant resolution of the $3$-dimensional $A_1$-singularity $k[a,b,c,d]/(ad-bc)$.)
\end{Ex}

\subsection{$a$-Segre products of Calabi-Yau categories}\label{a-Segre}
We study the effect of tensor products of root pairs on the Calabi-Yau completions.
Recall from \cite{ha6} that, given a positive integer $a$, the {\it $a$-Segre product} of $\Z$-graded dg categories $\X$ and $\Y$ is the $\Z$-graded dg category
\[ \bigoplus_{i\in\Z}\X_i\otimes\begin{pmatrix}	\Y_{i} &\Y_{i-1}&\cdots&\Y_{i-a+1}\\ \Y_{i+1}& \Y_i&\cdots&\Y_{i-a+2} \\ \vdots&\vdots&\ddots&\vdots \\ \Y_{i+a-1}&\Y_{i+a}&\cdots&\Y_{i}\end{pmatrix}. \]
Another description is as follows. We consider the $\Z$-graded dg category $\X$ as a dg category whose objects are $\Z$, and the morphism complex from $i$ to $j$ is given by $\X_{j-i}$, and the same for $\Y$. Now one can naturally view $\X\otimes\Y$ as a dg category whose objects are $\Z^2$. Then the $a$-Segre product is the full subcategory formed by the objects on the `thick diagonal' $\{(i,j)\mid 0\leq j-i\leq a-1\}$.

We are now able to prove the following, which is a translation of \ref{roottensor} to Calabi-Yau algebras.
\begin{Thm}\label{CYSegre}
Let $\Pi$ {\rm(}resp. $\Pi^\prime${\rm)} be a positively graded $(d+1)$-Calabi-Yau dg category {\rm(}resp. $(d^\prime+1)$-Calabi-Yau dg category{\rm)} of Gorenstein parameter $a>0$ such that $\Pi_0$ {\rm(}resp. $\Pi^\prime_0${\rm)} is perfect over $\Pi$ {\rm(}resp. over $\Pi^\prime${\rm)} on each side. Then their $a$-Segre product is positively graded $(d+d^\prime+1)$-Calabi-Yau dg category of Gorenstein parameter $a$.
\end{Thm}
\begin{Rem}
\begin{enumerate}
	\item The statement for $a=1$ becomes as follows: the (usual) Segre product of CY algebras of Gorenstein parameter $1$ is again CY, which is known, at least for ordinary algebras e.g.\,\cite[4.5]{Thi}.
	\item This is shown implicitly in \cite[Section 3]{ha6} for module-finite algebras (concentrated in cohomological degree $0$), using an explicit construction of projective resolutions of simple modules. Our setting and method here are completely different.
\end{enumerate}
\end{Rem}
The proof of \ref{CYSegre} will follow and are summarized as in the following diagram.
\[ \xymatrix@C=20mm{
	\text{root pairs } ((U,\P), (V,\cQ))\ar@{|->}[r]^-{\text{tensor product}}\ar@{|->}[d]_-{\text{CY completions}}&\text{root pair } (U\otimes V,\P\otimes\B)\ar@{|->}[d]^-{\text{CY completion}}\\
	\text{CY categories } (\Pi_{d+1}^{(1/a)}(\A), \Pi_{e+1}^{(1/a)}(\B))\ar@{|->}[r]^-{\text{$a$-Segre product}}&\Pi_{d+e+1}^{(1/a)}(\C) } \]
\begin{proof}
	Let $\Pi$ and $\Pi^\prime$ be Calabi-Yau dg categories. By \ref{MM} we can write $\Pi=\Pi_{d+1}^{(1/a)}(\A)$ and $\Pi^\prime=\Pi_{d^\prime+1}^{(1/a)}(\A^\prime)$ as Calabi-Yau completions of cyclically invariant root pairs $(U,\P)$ on $\A$ and $(U^\prime,\P^\prime)$ on $A^\prime$. By \ref{roottensor}, the pair $(U\otimes U^\prime,\P\otimes\A^\prime)$ is a root pair on $\A\otimes\A^\prime$ which is cyclically invariant by \ref{UotimesV}, and therefore its Calabi-Yau completion $\Pi_{d+d^\prime+1}^{(1/a)}(\A\otimes\A^\prime)$ is $(d+d^\prime+1)$-Calabi-Yau by \ref{CY}.
	
	It remains to verify that $\Pi_{d+d^\prime+1}^{(1/a)}(\A\otimes\A^\prime)$ is the $a$-Segre product of $\Pi$ and $\Pi^\prime$. The Calabi-Yau completion is the restriction of the tensor category $\rT^{\rL}_{\A\otimes\A^\prime}\!(U\otimes U^\prime)=\rT^{\rL}_\A\!U\otimes\rT^{\rL}_{\A^\prime}\!U^\prime$ to the subcategory of objects of $\P\otimes\A^\prime$, hence it is the Segre product $\Pi_{d+1}^{(1/a)}(\A)\seg\rT^{\rL}_{\A^\prime}\!U^\prime$. In view of its definition, this is nothing but the $a$-Segre product of $\Pi_{d+1}^{(1/a)}(\A)$ and $\Pi_{d^\prime+1}^{(1/a)}(\A^\prime)$.
\end{proof}

Recall that for an Adams graded dg category $\X=\bigoplus_{i\in\Z}\X_i$ and $n\in\Z$, its {\it $n$-th quasi-Veronese category $\X^{[n]}$} is the dg category
\[ \X^{[n]}=\bigoplus_{i\in\Z}\begin{pmatrix}\X_{ni}&\X_{ni-1}&\cdots&\X_{n(i-1)+1}\\ \X_{ni+1}&\X_{ni}&\cdots&\X_{n(i-1)+2}\\ \vdots&\vdots&\ddots&\vdots\\ \X_{n(i+1)-1}&\X_{n(i+1)-2}&\cdots&\X_{ni}\end{pmatrix}. \]
Notice that this is nothing but the $n$-Segre product with $k[x]$ with $\deg x=n$, and with the ``divided (Adams) grading''. Indeed, the $n$-Segre product itself with $k[x]$ concentrates in degree multiples of $n$, and dividing the grading by $n$ gives the above $n$-th quasi-Veronese category with the correct Adams grading.
This discussion leads to the following result.
\begin{Cor}\label{qv}
Let $\Pi$ be an Adams positively graded $(d+1)$-Calabi-Yau dg category of Gorenstein parameter $a$, that is, $\RHom_{\Pi^e}(\Pi,\Pi^e)[d+1]\simeq\Pi(a)$ in $\rD^\Z(\Pi^e)$. Suppose that $\Pi_0$ is perfect on each side. Then its $a$-th quasi-Veronese category $\Pi^{[a]}$ is a graded $(d+1)$-Calabi-Yau dg category of Gorenstein parameter $1$.
\end{Cor}
\begin{proof}
	Consider the $1$-Calabi-Yau completion $\Pi_1(k)$ of the base field with the multiplied Adams grading, that is, $k[x]$ with $x$ having cohomological degree $-1$ and Adams degree $a$. Then the assertion follows from \ref{CYSegre} as the $a$-Segre product of $k[x]$ and $\Pi$ with the divided Adams grading.
\end{proof}

For Calabi-Yau dg categories of Gorenstein parameter $1$, one can freely modify the Calabi-Yau dimension \cite{IQ}\cite[2.16]{haI}.
This can be explained using the ($1$-)Segre product as follows.
\begin{Ex}
Let $\Pi$ be an Adams positively graded $(d+1)$-Calabi-Yau dg category of Gorenstein parameter $1$, so that $\Pi$ is the $(d+1)$-Calabi-Yau completion of a dg category $\A$ by \ref{MM}: $\Pi=\Pi_{d+1}(\A)$.
Let $n$ be an arbitrary integer and consider the $(n+1)$-Calabi-Yau completion $\Pi_{n+1}(k)$ of $k$, thus $\Pi_{n+1}(k)=k[x]$ with $\deg x=(-n,1)$.
By \ref{qv}, the ($1$-)Segre product of $\Pi_{d+1}(\A)$ and $\Pi_{n+1}(k)$ is $(d+n+1)$-Calabi-Yau, and is nothing but the $(d+n)$-Calabi-Yau completion of $\A$;
\[ \Pi_{n+1}(k)\seg\Pi_{d+1}(\A)=\Pi_{d+n+1}(\A). \]
Note that this amounts to shifting the cohomological degree of the inverse dualizing bimodule $\RHom_{\A^e}(\A,\A^e)$ by $n$ by means of the Segre product.
\end{Ex}


\section{Strict root pairs on finite dimensional algebras}
We study root pairs on finite dimensional algebras, and introduce {\it strictness} of root pairs on finite dimensional algebras. These are root pairs which behave nicely with the algebra structure. 
The advantage of strictness is that one can stay inside the world of finite dimensional algebras, not dg algebras, which helps us understand the structures of the categories we are interested in.

\medskip

According to \ref{setup}, an $a$-th root pair on a finite dimensional algebra $A$ is a pair $(U,P)$ consisting of a complex of $(A,A)$-bimodules $U$ and a finitely generated projective module $P$ satisfying the following.
\begin{enumerate}
	\renewcommand\labelenumi{(\roman{enumi})}
	\renewcommand\theenumi{\roman{enumi}}
	\item $U^a\simeq\RHom_A(DA,A)[d]$ in $\rD(A^e)$.
	\item $\per A=\thick(\bigoplus_{i=0}^{a-1}P\lotimes_AU^i)$.
	\item $\Hom_{\rD(A)}(P\lotimes_AU^i,P)=0$ for all $0< i\leq a-1$.
\end{enumerate}

Our prototypical example \ref{first}(\ref{Bei}) suggests the following stronger version for finite dimensional algebras.
\begin{Def}\label{strict}
An $a$-th root pair $(U,P)$ on a finite dimensional algebra $A$ is {\it strict} if it satisfies the following instead of (\ref{P}) above.
\begin{itemize}
\item[(\ref{P})'] $\add A=\add(\bigoplus_{i=0}^{a-1}P\lotimes_AU^i$) in $\rD(A)$.
\end{itemize}
\end{Def}
Note that the condition (\ref{P}) shows that $(U,P)$ determines $A$ up to derived equivalence, while (\ref{P})' shows that $(U,P)$ determines up to Morita equivalence.
We show that strict root pairs also reproduce among finite dimensional algebras under a mild assumption. Note however, that naive tensor product does not work (as it did for general root pairs, \ref{UotimesV}), as the following easiest example suggests. 
\begin{Ex}\label{twD_4}
Let $A$ be the path algebra $kQ$ for the quiver $Q\colon\xymatrix{\circ\ar[r]&\circ}$ of type $A_2$. Then it has a strict (square) root pair $(DA,P)$ for the simple projective module $P$. Indeed, the functor $-\lotimes_ADA$ moves ``one place to the right'' on the Auslander-Reiten quiver, which is $\tau^{-1/2}$. More precisely, we have $DA\lotimes_ADA=\RHom_A(DA,A)[1]$ in $\rD(A^e)$.

Now let $A=B=kQ$ and consider the tensor product $A\otimes B$. Its derived category has a square root of $\nu_2$ given by the tensor product of $\nu_1^{1/2}$ on each factor (\ref{UotimesV}).
We know that $A\otimes B$ is derived equivalent to the path algebra of type $D_4$ via the tilting object $1\oplus2\oplus3\oplus4$ below.
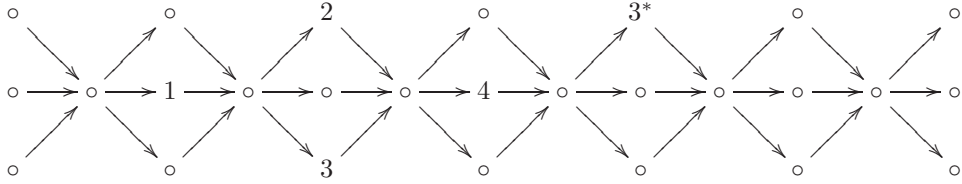
\[ \xymatrix@!R=2mm@!C=2mm{
	\circ\ar[dr]&&\circ\ar[dr]&&2\ar[dr]&&\circ\ar[dr]&&3^\ast\ar[dr]&&\circ\ar[dr]&&\circ\\
	\circ\ar[r]&\circ\ar[r]\ar[ur]\ar[dr]&1\ar[r]&\circ\ar[r]\ar[ur]\ar[dr]&\circ\ar[r]&\circ\ar[r]\ar[ur]\ar[dr]&4\ar[r]&\circ\ar[r]\ar[ur]\ar[dr]&\circ\ar[r]&\circ\ar[r]\ar[ur]\ar[dr]&\circ\ar[r]&\circ\ar[r]\ar[ur]\ar[dr]&\circ\\
	\circ\ar[ur]&&\circ\ar[ur]&&3\ar[ur]&&\circ\ar[ur]&&\circ\ar[ur]&&\circ\ar[ur]&&\circ} \]
An easy computation reveals that we have $\nu_2^{-1/2}=\nu_1^{-1/2}\otimes\nu_1^{-1/2}$ equals $\tau^{-2}$, which is consistent with $\nu=\tau^{-2}$, $[1]=\tau^{-3}$, hence $\nu_2^{-1}=\tau^{-4}$ on $\Db(kD_4)$.
We see that there is {\it no} projective $(A\otimes B)$-module $P$ satisfying (\ref{P})' above for $U=DA\otimes DB$.

On the other hand, if we take the left mutation of the above tilting object at $3$, we obtain a new tilting object $1\oplus2\oplus3^\ast\oplus4$ which is strict for this $\nu_2^{1/2}$. Its endomorphism ring is the path algebra of type $A_4$ modulo the longest path, which is to be described as follows.
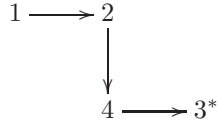
\[ \xymatrix{
	1\ar[r]&2\ar[d]&\\
	&4\ar[r]&3^\ast } \]
\end{Ex}

The above example generalizes to the existence result of {strict} root pairs for tensor products, {up to a derived equivalence}. 
\begin{Thm}\label{stricttensor}
Let $(U,P)$ be a strict $a$-th root pair on $A$, and $V$ an $a$-root on $B$. Suppose that we have $V^{i}\in\mod B$ for all $0\leq i\leq a-1$. Then $(U\otimes V,P\otimes B)$ is a strict $a$-th root pair on a finite dimensional algebra $C$ which is derived equivalent to $A\otimes B$. 
\end{Thm}
\begin{Rem}\label{C}
\begin{enumerate}
\item By the remark that strict root pairs determine algebras, we see that $C$ is (Morita equivalent to) $\End_{\rD(A\otimes B)}(\bigoplus_{i=0}^{a-1}(P\lotimes_A U^i)\otimes (B\lotimes_BV^i))$.
\item The root pair $(U\otimes V,P\otimes B)$ is not strict on $A\otimes B$, but strict on a derived equivalent algebra $C$, which is the point of the above result.
\end{enumerate}
\end{Rem}
\begin{proof}
	We know by \ref{roottensor} that $(U\otimes V,P\otimes B)$ is a root pair. To prove that this is a strict root pair for some finite dimensional algebra, it is enough to show that $T:=\bigoplus_{i=0}^{a-1}(P\otimes B)\lotimes_{A\otimes B}(U\otimes V)^i$ is a tilting object in $\per(A\otimes B)$.
	Since we already know by \ref{roottensor} that $T$ generates $\per(A\otimes B)$, we only have to show the vanishing of extensions. We have $\RHom_{A\otimes B}(T,T)=\bigoplus_{i,j=0}^{a-1}(\RHom_A(P\lotimes_AU^i,P\lotimes_AU^j)\otimes\RHom_B(B\lotimes_BV^i,B\lotimes_BV^j))$. Consider the first tensor factor $\RHom_A(P\lotimes_AU^i,P\lotimes_AU^j)$; it is concentrated in degree $0$ since $\bigoplus_{i=0}^{a-1}P\lotimes_AU^i=A$ (up to additive equivalence), and its $0$-th cohomology is $0$ unless $i\leq j$ by the condition (\ref{exc}) of the root pairs. Therefore it remains to show that the second tensor factor $\RHom_B(B\lotimes_BV^i,B\lotimes_BV^j)=V^{j-i}$ is concentrated in degree $0$ for $0\leq i\leq j\leq a-1$, but this is nothing but our assumption requiring $V^{i}\in\mod B$ for each $0\leq i\leq a-1$.
\end{proof}

Let us demonstrate how strictness can be used through the following example of cluster categories.
\begin{Ex}\label{C^half}
Let $A=B$ be the path algebra of type $A_2$ and $C$ the algebra presented by the following quiver modulo the longest path.
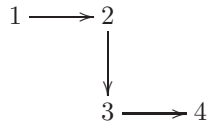
\[ \xymatrix{
	1\ar[r]&2\ar[d]&\\
	&3\ar[r]&4 } \]
We have seen in \ref{twD_4} that $C$ has a strict square root pair $(U,P)$ with $P=(e_1+e_2)C$, where the bimodule $U\in\rD(C^e)$ corresponds to $DA\otimes DB$ under a derived equivalence $\Db(\mod C)\simeq\Db(\mod A\otimes B)$.

Let $\rC_2(C)=\Db(\mod C)/-\lotimes_CDC[-2]$ be the (usual) $2$-cluster category of $C$, and let $\rC_2^{(1/2)}(C)$ be the $2$-folded $2$-cluster category given by the square root $U$ of $\RHom_C(DC,C)[2]$.
Let us give a description of the projection functor $\rC_2(C)\to\rC_2^{(1/2)}(C)$. The Auslander-Reiten quiver of $\rC_2(C)$ is as follows, where the circles depict the fundamental domain, and the black dots the summands of $C\in\rC_2(C)$.
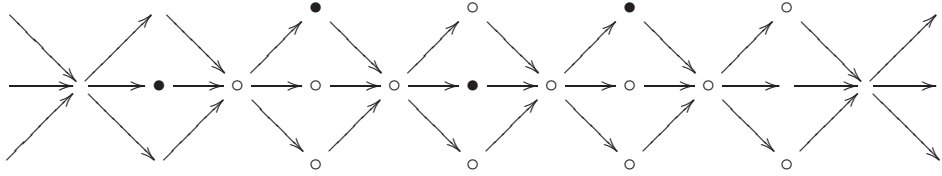
\[ \xymatrix@!R=2mm@!C=2mm{
	\ar[dr]&&\ar[dr]&&\bullet\ar[dr]&&\circ\ar[dr]&&\bullet\ar[dr]&&\circ\ar[dr]&&\\
	\ar[r]&\ar[r]\ar[ur]\ar[dr]&\bullet\ar[r]&\circ\ar[r]\ar[ur]\ar[dr]&\circ\ar[r]&\circ\ar[r]\ar[ur]\ar[dr]&\bullet\ar[r]&\circ\ar[r]\ar[ur]\ar[dr]&\circ\ar[r]&\circ\ar[r]\ar[ur]\ar[dr]&\ar[r]&\ar[r]\ar[ur]\ar[dr]&\\
	\ar[ur]&&\ar[ur]&&\circ\ar[ur]&&\circ\ar[ur]&&\circ\ar[ur]&&\circ\ar[ur]&& } \]
The action $-\lotimes_CU$ is given by $\tau^{-2}$, and $\rC_2^{(1/2)}(C)$ has the following structure, where the circles depict the fundamental domain, and the black dots the summands of $P\in\rC_2^{(1/2)}(C)$ .
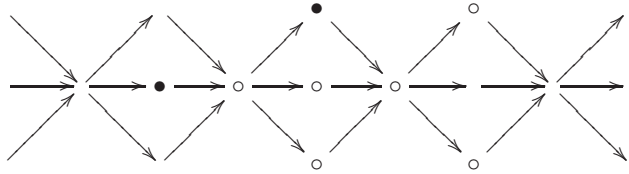
\[ \xymatrix@!R=2mm@!C=2mm{
	\ar[dr]&&\ar[dr]&&\bullet\ar[dr]&&\circ\ar[dr]&&\\
	\ar[r]&\ar[r]\ar[ur]\ar[dr]&\bullet\ar[r]&\circ\ar[r]\ar[ur]\ar[dr]&\circ\ar[r]&\circ\ar[r]\ar[ur]\ar[dr]&\ar[r]&\ar[r]\ar[ur]\ar[dr]&\\
	\ar[ur]&&\ar[ur]&&\circ\ar[ur]&&\circ\ar[ur]&& } \]

Notice that the initial cluster tilting object $C\in\rC_2(C)$ is stable under the $\Z/2\Z$-action by the strictness of $(U,P)$ with respect to $C$. This explains the existence of the cluster tilting object $P\in\rC_2^{(1/2)}(C)$ as the image of $C$ under $\rC_2(C)\to\rC_2^{(1/2)}(C)$. 
\end{Ex}


It is also important to know when the Calabi-Yau dg category is an ordinary algebra. In view of \ref{fdMM}, this amounts to discuss higher representation infinite algebras. The following result is a counterpart of \ref{CYSegre} for finite dimensional algebras.
This is also an $a$-th root generalization of \cite[2.10]{HIO} for tensor product of higher representation infinite algebras. 
\begin{Thm}\label{RI}
	Let $A$ be a $d$-representation infinite algebra with a strict $a$-th root pair $(U,P)$ for $\RHom_A(DA,A)[d]$, and let $B$ be an $e$-representation infinite algebra with a strict $a$-th root pair $(V,Q)$ of $\RHom_B(DB,B)[e]$. Suppose that $A/J_A$ or $B/J_B$ is separable over $k$. Then the algebra $C$ in \ref{C} is $(d+e)$-representation infinite with a strict $a$-th root pair $(U\otimes V,P\otimes B)$ for $\RHom_C(DC,C)[d+e]$.
\end{Thm}
\begin{proof}
	We write $\nu_d^{-1/a}$ for $-\lotimes_AU$ on $\Db(\mod A)$ and $\nu_e^{-1/a}$ for $-\lotimes_BV$ on $\Db(\mod B)$. Since $B$ is $e$-representation infinite we have $\nu_e^{-i/a}B\in\add(\bigoplus_{i\geq0}Q\lotimes_BV^i)\subset\mod B$ for all $i\geq0$. Then by \ref{stricttensor} we get a strict root pair $(U\otimes V,P\otimes B)$ on $\Db(\mod A\otimes B)=\Db(\mod C)$.
	Now by the separability assumption the tensor product $A\otimes B$ has finite global dimension, thus so is derived equivalent $C$. Also, since $\RHom_C(P\otimes B,\nu_{d+e}^{-i/a}(P\otimes B))=\RHom_A(P,\nu_d^{-1/a}P)\otimes\RHom_B(B,\nu_e^{-i/a}B)$ is concentrated in degree $0$ for each $i\geq0$, we deduce that $C$ is $(d+e)$-representation infinite.
\end{proof}

\begin{Ex}
	Let $A=B$ be the $d$-Beilinson algebra which has a strict $(d+1)$-st root pair of $\nu_d^{-1}$ corresponding to the $(d+1)$-Calabi-Yau algebra $k[x_0,\ldots,x_d]$. Then the algebra $C$ constructed in \ref{RI} is presented by the following quiver with commutativity relations (for $d=3$).
	\[ \xymatrix{
		\circ\ar@3[r]&\circ\ar@3[r]\ar@3[d]&\circ\ar@3[r]\ar@3[d]&\circ\ar@3[d]&&&\\
		&\circ\ar@3[r]&\circ\ar@3[r]\ar@3[d]&\circ\ar@3[r]\ar@3[d]&\circ\ar@3[d]&&\\
		&&\circ\ar@3[r]&\circ\ar@3[r]\ar@3[d]&\circ\ar@3[r]\ar@3[d]&\circ\ar@3[d]&\\
		&&&\circ\ar@3[r]&\circ\ar@3[r]&\circ\ar@3[r]&\circ,} \]
	where the horizontal arrows are labelled by $\{x_0,\ldots,x_d\}$ and vertical ones by $\{y_0,\ldots,y_d\}$. This algebra is derived equivalent to $\PP^d\times\PP^d$, and is $2d$-representation infinite by \ref{RI}.
	
	The Calabi-Yau completion $\Pi^{(1/(d+1))}_{2d+1}(C)$ of the root pair on $C$ is the $(d+1)$-Segre product of $k[x_0,\ldots,x_d]$ and $k[y_0,\ldots,y_d]$, which is presented by the following quiver with commutativity relations.
	\[ \xymatrix{\circ\ar@3@/^5pt/[r]&\circ\ar@3@/^5pt/[r]\ar@3@/^5pt/[l]&\circ\ar@3@/^5pt/[r]\ar@3@/^5pt/[l]&\circ\ar@3@/^5pt/[l] } \]
	This algebra has the (1-)Segre product $R=k[x_0,\ldots,x_d]\seg k[y_0,\ldots,y_d]$ as an idempotent subalgebra, and the above algebra is a non-commutative crepant resolution of $R$, see \cite[Section 3]{ha6}.
\end{Ex}

\section{Dynkin quivers}\label{Dynkin}
We apply our Calabi-Yau completion of root pairs to path algebras of quivers.
Here we focus on $a$-th roots of $\tau^{-1}=\RHom_{(kQ)^e}(kQ,(kQ)^e)[1]$, the inverse of the Auslander-Reiten translation of the derived category $\Db(\mod kQ)$ of a quiver $Q$. Note that such a root of $\tau$ yields an $a$-th root of $\tau$ on the infinite translation quiver $\Z Q$ as a component of the Auslander-Reiten quiver \cite[6.1]{ha4}. To distinguish the difference of these roots of $\tau$, we give the following definitions.
\begin{Def}
Let $Q$ be a finite quiver and let $a$ be a positive integer.
A {\it combinatorial $a$-th root of $\tau$} is an automorphism $F$ on the infinite translation quiver $\Z Q$ such that $F^a=\tau^{-1}$.
\end{Def}
When we want to emphasize, we will call $U$ a {\it homological} root of $\tau$.

We have the following characterization of quivers with combinatorial roots of $\tau$. We say that quivers $Q$ and $Q^\prime$ are {\it derived equivalent} if the path algebras $kQ$ and $kQ^\prime$ are derived equivalent, which is the case if and only if $\Z Q$ and $\Z Q^\prime$ are isomorphic.
\begin{Prop}[{\cite[6.3]{ha4}}]
Let $Q$ be a finite acyclic quiver and $a$ be a positive integer. Then $\Z Q$ has a combinatorial $a$-th root of $\tau$ if and only if $Q$ is derived equivalent to $Q^\prime$ satisfying the following.
\begin{enumerate}
\renewcommand{\labelenumi}{(\roman{enumi})}
\renewcommand{\theenumi}{\roman{enumi}}
\item The quiver $Q^\prime$ contains $a$ copies $T:=Q^{(0)}, Q^{(1)},\ldots, Q^{(a-1)}$ of a quiver $T$ as full subquivers.
\item\label{sod} There are additional arrows $x\to y$ in $Q$ with $x\in Q^{(l)}$, $y\in Q^{(m)}$ only if $l<m$.
\suspend{enumerate}
We denote by $F$ the permutation of the vertices of $Q$ taking $Q^{(l)}$ to $Q^{(l+1)}$, where $Q^{(a)}:=Q^{(0)}$.
\resume{enumerate}
\renewcommand{\labelenumi}{(\roman{enumi})}
\renewcommand{\theenumi}{\roman{enumi}}
\item $F$ extends to an automorphism of the underlying graph of $Q$.
\end{enumerate}
\end{Prop}

It is a subtle question whether this {\it combinatorial} root of $\tau$ lifts to a {\it homological} one, that is, if there is a $(kQ,kQ)$-bimodule complex $U$ which is an $a$-th root of $\RHom_{kQ}(D(kQ),kQ)[1]$, and furthermore if such a $U$ can be taken to be cyclically invariant.

In this section we settle these problems for Dynkin quivers.
For this we recall the complete description of connected Dynkin quivers with an $a$-th root of $\tau$.
\begin{Prop}[{\cite[6.12]{ha4}}]\label{Dyn}
Let $Q$ be a connected Dynkin quiver and $a\geq2$. Then $\Z Q$ has a combinatorial $a$-th root of $\tau$ for some $a$ if and only if $a=2$ and $Q$ is of type $A_n$ for some even $n$.
\end{Prop}

Let $n\geq1$ and consider the quiver $Q$ of type $A_{2n}$ with the following orientation.
\[ \xymatrix@!R=2mm@!C=2mm{
	&1\ar[dl]_-{x_1}\ar[dr]^-{y_1}&&2\ar[dl]_-{x_2}\ar[dr]^-{y_2}&&\cdots&& n-1\ar[dl]_-{x_{n-1}}\ar[dr]^-{y_{n-1}}&& n\ar[dl]^-{x_n}\\
	n+1&&n+2&&\cdots&&\cdots&& 2n& } \]
The combinatorial $\tau^{-1/2}$ can be described as follows. Consider the automorphism of the infinite translation quiver $\Z Q$ (for type $A_4$ in the picture below) which turns up-side-down. For example, it takes the vertex $1$ to $4$, and $2$ to $3$. We see that this is indeed a square root of $\tau^{-1}$.
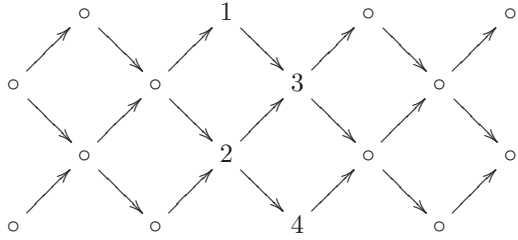
\[ \xymatrix@!R=1mm@!C=1mm{
	&\circ\ar[dr]&&1\ar[dr]&&\circ\ar[dr]&&\circ\\
	\circ\ar[dr]\ar[ur]&&\circ\ar[dr]\ar[ur]&&3\ar[dr]\ar[ur]&&\circ\ar[dr]\ar[ur]&\\
	&\circ\ar[dr]\ar[ur]&&2\ar[dr]\ar[ur]&&\circ\ar[dr]\ar[ur]&&\circ\\
	\circ\ar[ur]&&\circ\ar[ur]&&4\ar[ur]&&\circ\ar[ur]& } \]

Now we construct its lift to the homological level.
In what follows we denote $A:=kQ$. Let $d\in\Z$, $\e\in\{\pm1\}$, and define the (dg) bimodule $U$ over $A$ by
\[ \xymatrix@R=2mm{
	\disoplus_{i=1}^nAe_{2n+1-i}\otimes e_iA\ar[r]&\disoplus_{i=1}^nAe_i\otimes e_{2n+1-i}A \\
	\quad e_{2n+1-i}\otimes e_i \quad \ar@{|->}[r]& \quad (-1)^d(x_{n+1-i}\otimes x_i+ \e y_{n-i}\otimes y_i)& (1\leq i\leq n-1)\\
	\qquad e_{n+1}\otimes e_n \qquad \ar@{|->}[r]& \qquad (-1)^dx_1\otimes x_n } \]
with terms in degrees $-d-1$ and $-d$. We will understand $y_0=y_n=0$ so that the above formula becomes uniform.
The main computation of this section is the following.
\begin{Prop}
\begin{enumerate}
\item $U$ is a square root of $\RHom_{A^e}(A,A^e)[2d+1]$ for any $d\in\Z$.
\item $U$ is cyclically invariant if and only if $(-1)^d=-1$ in $k$.
\end{enumerate}
\end{Prop}
\begin{proof}
	Recall that for any path algebra $\La=kQ$ we have a standard bimodule resolution
	\[ \xymatrix@R=1mm{
		p\La\colon&\disoplus_{x\in Q_1}\La e_{t(x)}\otimes e_{s(x)}\La \ar[r]& \disoplus_{i\in Q_0}\La e_i\otimes e_i\La \\
		&\quad e_{t(x)}\otimes e_{s(x)}\quad \ar@{|->}[r]& \quad x\otimes e_{s(x)}-e_{t(x)}\otimes x\quad } \]
	with terms in degrees $-1$ and $0$, and where $s$ and $t$ mean the source and the target of an arrow. Dualizing over $\La^e$, we get a bimodule cofibrant resolution $\cHom_{\La^e}(p\La,\La^e)\to\RHom_{\La^e}(\La,\La^e)$:
	\[ \xymatrix@R=1mm{
		\disoplus_{i\in Q_0}\La e_{i}\otimes e_{i}\La \ar[r]& \disoplus_{x\in Q_1}\La e_{s(x)}\otimes e_{t(x)}\La \\
		\quad e_i\otimes e_i\quad \ar@{|->}[r]& \quad \sum_{s(x)=i}e_{i}\otimes x-\sum_{t(x)=i}x\otimes e_{i}.\quad } \]
	In particular for our path algebra $A$ of type $A_{2n}$ we get a bimodule resolution $\cHom_{A^e}(pA,A^e)$
	\begin{equation}\label{A^vee}
	\xymatrix@R=3mm@C=2mm{
	\disoplus_{i=1}^nAe_i\otimes e_iA &\oplus\ar@{}[d]^(.25){}="a"^(.75){}="b" \ar "a";"b" &\disoplus_{i=1}^nAe_{n+i}\otimes e_{n+i}A\\
	\disoplus_{i=1}^nAe_i\otimes e_{n+i}A&\oplus&\disoplus_{i=1}^{n-1}Ae_i\otimes e_{n+i+1}A, }
	\quad
	\xymatrix@C=3mm{
	e_i\otimes e_i\ar@{|->}[d]\ar@{|->}[dr]&\\
	e_i\otimes x_i&e_i\otimes y_i, }
	\quad
	\xymatrix@C=3mm{
	& e_{n+i}\otimes e_{n+i}\ar@{|->}[dl]\ar@{|->}[d]\\
	-x_i\otimes e_{n+i}&-y_{i-1}\otimes e_{n+i}. }
	\end{equation}
	of $\RHom_{A^e}(A,A^e)$ with terms in degrees $0$ and $1$.
	
	Let us next compute $U\otimes_AU$. Noting that $e_jAe_i=0$ unless $j-i\in\{0,n,n+1\}$, we easily see that it is the complex
	\begin{equation}\label{U^2}
	\xymatrix@R=4mm@C=2mm{
		\disoplus_{i=1}^nAe_i\otimes e_{2n+1-i}Ae_{2n+1-i}\otimes e_iA &\oplus\ar@{}[d]^(.25){}="a"^(.75){}="b" \ar "a";"b"& \disoplus_{i=1}^nAe_{n+i}\otimes e_{n+1-i}Ae_{n+1-i}\otimes e_{n+i}A \\
		\disoplus_{i=1}^n Ae_{i}\otimes e_{2n+1-i}Ae_{n+1-i}\otimes e_{n+i}A &\oplus& \disoplus_{i=1}^{n-1}Ae_{i}\otimes e_{2n+1-i}Ae_{n-i}\otimes e_{n+1+i}A }
	\end{equation}
	with maps
	\[
	\xymatrix@R=4mm@C=3mm{
		e_i\otimes e_{2n+1-i}\otimes e_i\ar@{|->}[d]\ar@{|->}[dr]&\\
		e_i\otimes x_{n+1-i}\otimes x_i&\e e_i\otimes y_{n-i}\otimes y_i, }
	\quad
	\xymatrix@R=4mm@C=3mm{
		&e_{n+i}\otimes e_{n+1-i}\otimes e_{n+i}\ar@{|->}[dl]\ar@{|->}[d]\\
		(-1)^dx_{i}\otimes x_{n+1-i}\otimes e_{n+i}&(-1)^d\e y_{i-1}\otimes y_{n+1-i} \otimes e_{n+i} }
	\]
	and with terms in degrees $-2d-1$ and $-2d$.
	
	Now we construct a quasi-isomorphism $\RHom_{A^e}(A,A^e)\xsimeq U\lotimes_AU$ of degree $-2d-1$. Let $\varphi$ be the map from \eqref{A^vee} to \eqref{U^2} defined via the isomorphisms
	\[ \xymatrix@R=2mm{
	Ae_i\otimes e_iA\ar[r]^-\simeq& Ae_i\otimes e_{2n+1-i}Ae_{2n+1-i}\otimes e_iA,& e_i\otimes e_i \ar@{|->}[r]&-e_i\otimes e_{2n+1-i}\otimes e_i\\
	Ae_{n+i}\otimes e_{n+i}A\ar[r]^-\simeq &Ae_{n+i}\otimes e_{n+1-i}Ae_{n+1-i}\otimes e_{n+i}A, & e_{n+i}\otimes e_{n+i}\ar@{|->}[r]&(-1)^de_{n+i}\otimes e_{n+1-i}\otimes e_{n+i}} \]
	in degree $0$, and
	\[ \xymatrix@R=2mm{
		Ae_i\otimes e_{n+i}A\ar[r]^-\simeq&Ae_{i}\otimes e_{2n+1-i}Ae_{n+1-i}\otimes e_{n+i}A,& e_i\otimes e_{n+i} \ar@{|->}[r]&e_i\otimes x_{n+1-i}\otimes e_{n+i}\\
		Ae_i\otimes e_{n+i+1}A\ar[r]^-\simeq&Ae_{i}\otimes e_{2n+1-i}Ae_{n-i}\otimes e_{n+1+i}A,& e_{n+i}\otimes e_{n+i}\ar@{|->}[r]&\e e_i\otimes y_{n-i}\otimes e_{n+i+1} } \]
	in degree $1$. 
	Clearly the map $\varphi$ defined in this way is a (quasi-)isomorphism of complexes of $(A,A)$-bimodules of degree $-2d-1$.
	
	Finally we verify the cyclic invariance of $U$. We denote by $\sum_{i}(e_i\otimes e_i)\otimes (e_i\otimes e_i)^\ast+\text{(terms for the arrows)}\in pA\otimes\cHom_{A^e}(pA,A^e)$ the Casimir element of $pA$. For example, $(e_i\otimes e_i)^\ast\in\cHom_{A^e}(pA,A^e)$ corresponds to $e_i\otimes e_i$ in the degree $0$ part of \eqref{A^vee}.
	By \ref{LemC}(\ref{C**}) the element in $U^2\otimes_{A^e}A$ corresponding to $\varphi\in\cHom_{A^e}(pA,U^2)$ is $\sum_{i=1}^{2n}\varphi(e_i\otimes e_i)\otimes e_i$, which is equal to $-\sum_{i=1}^ne_i\otimes e_{2n+1-i}+\sum_{i=1}^{n}(-1)^de_{n+i}\otimes e_{n+1-i}\in\bigoplus_{i=1}^{n}e_iAe_i\otimes e_{2n+1-i}Ae_{2n+1-i}\oplus\bigoplus_{i=1}^{n}e_{n+i}Ae_{n+i}\otimes e_{n+1-i}Ae_{n+1-i}$. Now the $\Z/2\Z$-action on $U^2\otimes_{A^e} A$ is given by the permutation of these two summands. Therefore we deduce that $U$ is cyclically invariant if and only if $(-1)^d=-1$.
\end{proof}

We get the following consequence by \ref{fakeCY} and \ref{CY}.
\begin{Thm}\label{odd}
Let $A$ and $U$ be as above, and suppose that $d$ is odd.
\begin{enumerate}
\item The dg algebra $\Si=\rT^{\rL}_A\!U$ satisfies $\RHom_{\Si^e}(\Si,\Si^e)[2d+2]\simeq\Si_{\geq1}$ in $\rD(\Si^e)$.
\item Let $P$ be the projective $A$-module corresponding to the vertices from $1$ to $n$. Then $(U,P)$ is a strict root pair and its Calabi-Yau completion $\Pi=\Pi^{(1/2)}_{2d+2}(A)$ is $(2d+2)$-Calabi-Yau. 
\end{enumerate}
\end{Thm}

Let us further describe the dg algebras $\rT^{\rL}_A\!U$ and the Calabi-Yau completion $\Pi$ by dg path algebras.
\begin{Prop}\label{dgq}
\def\ard{\ar@<-2pt>[d]_-v}
\def\aru{\ar@<-2pt>[u]_-u}
\def\arr{\ar@<2pt>[r]}
\def\arl{\ar@<2pt>[l]}
\def\lo{\ar@(dl,dr)[]}
Let $A$ and $U$ be as above.
\begin{enumerate}
\item The dg algebra $\Si=\rT^{\rL}_A\!U$ is presented by the following dg path algebra, where the vertices in the first row are numbered as $n, n-1,\ldots$ for odd positions, and as $1,2,\ldots$ for even positions. Also the last vertical arrow $z$ is $x$ if $n$ is odd, and it is $y$ if $n$ is even.
\[ 
\xymatrix{
	n\ard\ar[dr]^-x&1\ard\ar[dr]^-y\ar[dl]&n-1\ard\ar[dr]^-x\ar[dl]&\cdots\ar[dr]\ar[dl]&\circ\ard\ar[dr]\ar[dl]&\circ\ard\ar[dl]\ar@/^15pt/[d]^-z\\
	n+1\aru& 2n\aru& n+2\aru&\cdots&\circ\aru&\circ\aru }
\quad
\xymatrix@R=3mm{
		|x|=|y|=0,\quad |u|=-d, |v|=-d-1\\
		du=0, \quad dv=(-1)^d(xux+ \e yuy) } 
\]
\item The Calabi-Yau completion $\Pi=\Pi^{(1/2)}_{2d+2}(A)$ is presented by the following dg path algebra, with the vertices numbered as in {\rm(1)}, and $c=a$ if $n$ is odd, and $c=b$ if $n$ is even.
\[ 
\xymatrix{
	n\arr^-a\lo_-t&1\arl^-a\arr^-b\lo_-t&n-1\arr^-a\arl^-b\lo_-t&\cdots\arr\arl^-a&\circ\arr\arl\lo_t&\circ\arl\lo_t\ar@(ul,ur)[]^-c}
\quad
\xymatrix@R=1mm{
	|a|=|b|=-d, \,|t|=-2d-1\\
	da=db=0, \quad dt=a^2+\e b^2 } \]
\end{enumerate}
\end{Prop}
\begin{proof}
	(1)  This follows from the definition of $U$ as a cofibrant dg $(A,A)$-bimodule.
	
	(2)  Note that for a (dg) path algebra $kQ$ and an idempotent $e$ corresponding to a subset $I$ of vertices, the algebra $e(kQ)e$ is presented by the quiver whose vertices are $I$ and the arrows from $i$ to $j$ are composites $i\to k\to l\to j$, where $k,l\in Q\setminus I$, $i\to k$ and $l\to j$ are arrows in $Q$, and $k\to l$ is a path in $Q\setminus I$.
	Then we see that the quiver of $\Pi$ has arrows $a:=ux$, $b:=uy$, and loops $t:=uv$. The degrees of these arrows are determined by those of $x$, $y$, $u$, and $v$, and the differential by the Leibniz rule for $\Si$.
\end{proof}

We give some descriptions of the associated cluster categories.
\begin{Ex}\label{typeA}
As in \ref{odd}, let $d$ be an odd integer and $A=A_{2n}$ the path algebra of a quiver of type $A_{2n}$. By \ref{rc}, the cluster category $\rC(\Pi)$ for $\Pi=\Pi^{(1/2)}_{2d+2}(A)$ is equivalent to the orbit category of $\Db(\mod A)/\tau^{-1/2}[d]$ with $\tau^{-1/2}=-\lotimes_AU$. Its Auslander-Reiten quiver is as below for some $n$ and $d$. Here, the numbered vertices describe the fundamental domain.
\[
\xymatrix@!R=1mm@!C=1mm{
	\\
	&\circ\ar[dr]&&2\ar[dr]&&4\ar[dr]&&\cdots\\
	\cdots\ar[ur]&&1\ar[ur]&&3\ar[ur]&&\circ\ar[ur]&\\
	&& n=1,\, d=1;&& &\rC(\Pi)=\rC^{(1/2)}_3(A_2) }
\quad
\xymatrix@!R=1mm@!C=1mm{
	\cdots\ar[dr]&&1\ar[dr]&&5\ar[dr]&&9\ar[dr]&&\circ\ar[dr]&\\
	&\circ\ar[dr]\ar[ur]&&3\ar[dr]\ar[ur]&&7\ar[dr]\ar[ur]&&11\ar[dr]\ar[ur]&&\cdots\\
	\cdots\ar[dr]\ar[ur]&&2\ar[dr]\ar[ur]&&6\ar[dr]\ar[ur]&&10\ar[dr]\ar[ur]&&\circ\ar[dr]\ar[ur]&\\
	&\circ\ar[ur]&&4\ar[ur]&&8\ar[ur]&&12\ar[ur]&&\cdots \\
	&&& n=2,\, d=1; && & \rC(\Pi)=\rC^{(1/2)}_3(A_4)}
\]
We also give an example for larger $d$.
\[
\xymatrix@!R=1mm@!C=1mm{
	&\circ\ar[dr]&&2\ar[dr]&&4\ar[dr]&&6\ar[dr]&&8\ar[dr]&&10\ar[dr]&&\cdots\\
	\cdots\ar[ur]&&1\ar[ur]&&3\ar[ur]&&5\ar[ur]&&7\ar[ur]&&9\ar[ur]&&\circ\ar[ur]& \\
	&& && n=1, \, d=3; && & \rC(\Pi)=\rC^{(1/2)}_5(A_2) } 
\]
\end{Ex}

\begin{appendix}
\section{Graded cluster categories and Beilinson-type theorem}\label{A}
A celebrated result of Beilinson states a triangle equivalence
\[ \Db(\coh\PP^n)\simeq\Db(\mod A) \]
between the derived category of the projective space $\PP^n$ and the Beilinson algebra $A$ via the tilting object $T=\bigoplus_{i=0}^n\O(i)$.
The aim of this section is to give its generalization to the setting of graded dg categories via an interpretation of $\Db(\coh\PP^n)$ as a graded cluster category.

\medskip

Let us start with the following notion which extends the established one for ordinary rings or Calabi-Yau dg algebras. Following the notation from Section \ref{MMproof}, we denote by $\rD^\Z(\Ga)_i$ the full subcategory of $\rD^\Z(\Ga)$ formed by graded dg modules concentrated in (Adams) degree $i$.
\begin{Def}\label{GP}
Let $\Ga$ be an (Adams) graded dg category. We say that $\Ga$ has {\it Gorenstein parameter $a$} if $\RHom_\Ga(-,\Ga)$ takes $\rD^\Z(\Ga)_0$ to $\rD^\Z(\Ga^\op)_{-a}$.
\end{Def}

Let us first note that this definition agrees with the one we gave in \ref{MM} for suitable Calabi-Yau dg categories.
\begin{Prop}\label{CYGP}
Let $\Pi$ be an Adams positively graded $n$-Calabi-Yau dg category of Gorenstein parameter $a$ in the sense of \ref{MM}, that is, $\Pi\in\per^\Z\!\Pi^e$ and $\RHom_{\Pi^e}(\Pi,\Pi^e)[n]\simeq\Pi(a)$ in $\rD^\Z(\Pi^e)$. Suppose that $\Pi_0$ is perfect on each side. Then $\Pi$ has Gorenstein parameter $a$ in the sense of \ref{GP} above.
\end{Prop}
\begin{proof}
	By \ref{MM} we can write $\Pi$ as the Calabi-Yau completion of a cyclically invariant root pair $(U,\P)$ on a smooth dg category $\A$, thus $\Pi=e_0\Si e_0$, where $\Si:=\rT^\rL_\A\!U$ and we denote as before by $Xe_i$ the restriction of a $\Si$-module $X$ to the subcategory $\{P\lotimes_\A U^i\mid P\in\P\}$ for each $0\leq i\leq a-1$, and similarly for left modules and bimodules.
	As before, we also denote by $\Si_{\geq l}=\bigoplus_{i\geq l}U^l$ the two-sided ideal of $\Si$ for each $l\geq0$, which we view as a bimodule over $\Si$ concentrated in degree $\geq l$. Moreover, for $-l\leq 0$ we write $\Si_{\geq -l}=\RHom_\A(U^l,\Si)(l)=\Si\lotimes_\A\RHom_\A(U^l,\A)(l)$ the $(\Si,\A)$-bimodule (concentrated in degree $\geq-l$).
	Consider the triangle
	\[ \xymatrix{ U^a\lotimes_\A\Si(-1)\ar[r]&\Si_{\geq a-1}(a-1)\ar[r]&\Si_{a-1}(a-1) } \]
	in $\rD^\Z(\Si)$. Restricting as $e_0(-)e_{a-1}$ gives 
	\[ \xymatrix{ e_0U^a\lotimes_\A\Si e_{a-1}(-1)\ar[r]&\Pi\ar[r]&\Pi_0 } \]
	in $\rD^\Z(\Pi)$. Then, applying $\RHom_\Pi(-,\Pi)$ we get
	\[ \xymatrix{ \RHom_\Pi(\Pi_0,\Pi)\ar[r]&\Pi\ar[r]&\RHom_\Pi(e_0U^a\lotimes_\A\Si e_{a-1},\Pi)(1)}, \]
	where the last term is isomorphic to $\RHom_\A(e_0U^a,\RHom_\Pi(\Si e_{a-1},\Pi))(1)=\RHom_\A(e_0U^a,e_0\Si)(a)=e_0\Si\lotimes_\A\RHom_\A(e_0U^a,\A)(a)=e_0\Si_{\geq-a}e_0$ by adjunction and \ref{gr-Pi-dual}. Now, the map $\Pi\to e_0\Si_{\geq-a}e_0$ identifies with the restriction $e_0(-)e_0$ of the natural map $\Si\to\Si_{\geq-a}$, so its mapping cone is $\bigoplus_{l=1}^{a}\RHom_\A(e_0U^l,e_0\A)(l)$. By \ref{setup}(\ref{exc}) this is $0$ except for the term $l=a$, which implies $\RHom_\Pi(\Pi_0,\Pi)$ is concentrated in degree $-a$, as desired.	
\end{proof}

\begin{Rem}
When $\Pi_0$ is proper, \ref{CYGP} is immediate from Serre duality \cite[4.1]{Ke08}. Our assertion does not require $\Pi_0$ to be proper, but depends on \ref{MM}, the description of $\Pi$ as a Calabi-Yau completion. We do not know if there is a direct proof.
\end{Rem}

We will be interested in the case where $\Ga$ is positively graded (with respect to the Adams grading), in which case the (Adams) degree $0$ part $\Ga_0(-,A)$ is a graded dg $\G$-module for each object $A\in\Ga$. The following observation gives another characterization of the Gorenstein parameter.
To ease notation we will simply write the category $\{\Ga_0(-,A)\mid A\in\Ga\}$ as $\Ga_0$, and similarly for any bimodule over $\Ga$.
Also, in a triangulated category with coproducts, we denote by $\Loc\S$ the smallest triangulated subcategory containing $\S$ and closed under coproducts.
\begin{Lem}\label{+}
Let $\Ga$ be a positively graded dg category.
\begin{enumerate}
\item\label{gen} We have $\rD^\Z(\Ga)_0=\Loc\Ga_0$.
\item\label{GP2} $\Ga$ has Gorenstein parameter $a$ if and only if $\RHom_\Ga(\Ga_0,\Ga)\in\rD^\Z(\Ga^\op)_{-a}$.
\end{enumerate}
\end{Lem}
\begin{proof}
	(1)  We have a natural functor $\rD(\Ga_0)\to\rD^\Z(\Ga)_0$ which preserves coproducts, and is essentially surjective. Since $\rD(\Ga_0)=\Loc\Ga_0$, we deduce $\rD^\Z(\Ga)_0=\Loc\Ga_0$.
	
	(2)  The ``only if'' part is clear. The ``if'' part is a consequence of (1).
\end{proof}

We will further assume that $\Ga_0$ is perfect on left and right, thus we are in the setting below.
\begin{Setup}\label{setG}
\begin{itemize}
\item $\Ga$ is an Adams positively graded dg category of Gorenstein parameter $a$.
\item $\Ga_0\in\per^\Z\!\Ga$ and $\Ga_0\in\per^\Z\!\Ga^\op$.
\end{itemize}
\end{Setup}

We can define (graded) cluster categories in this setting.
\begin{Def}\label{CZ}
In the setting of \ref{setG}, we define the {\it graded cluster category} as the Verdier quotient
\[ \rC^\Z(\Ga):=\per^\Z\!\Ga/\thick\{\Ga_0(i)\mid i\in\Z\}. \]
Also, the (ungraded) {\it cluster category} in this setting is defined as
\[ \rC(\Ga):=\per\Ga/\thick\Ga_0. \]
\end{Def}
There is a natural forgetful functor $\rC^\Z(\Ga)\to\rC(\Ga)$ and the ungraded cluster category is the triangulated hull of $\rC^\Z(\Ga)/(1)$.

Let us now turn to stating the main result of this appendix.
Recall that a stable $t$-structure $\T=\X\perp\Y$ in a triangulated category $\T$ is also called a {\it semi-orthogonal decomposition}. We will use the notation $\T=\X_1\perp\cdots\perp\X_n$ if $\T$ is decomposed into more components.
\begin{Thm}\label{B}
Let $\Ga$ be a graded dg category satisfying \ref{setG}, and let $T=\bigoplus_{i=0}^{a-1}\Ga(i)$. Define a dg category $\A$ and an $(\A,\A)$-bimodule $U$ by
\[ \A=\REnd_\Ga^\Z(T) =\begin{pmatrix}\Ga_0&0&\cdots&0\\ \Ga_1&\Ga_0&\cdots&0\\ \vdots&\vdots&\ddots&\vdots\\ \Ga_{a-1}&\Ga_{a-2}&\cdots&\Ga_0\end{pmatrix}, \quad
U=\RHom_\Ga^\Z(T,T(1))=\begin{pmatrix}	\Ga_{1} & \Ga_0&\cdots& 0\\ \vdots& \vdots&\ddots&\vdots \\ \Ga_{a-1}&\Ga_{a-2}&\cdots&\Ga_0 \\ \Ga_{a}&\Ga_{a-1}&\cdots&\Ga_{1}\end{pmatrix}. \]
\begin{enumerate}
\item\label{equiv} The functor $-\lotimes_\A T$ gives an equivalence $\per\A\xsimeq\rC^\Z(\Ga)$.
\item\label{SOD} There is a semi-orthogonal decomposition $\rC^\Z(\Ga)=\thick\Ga(a-1)\perp\cdots\perp\thick\Ga(1)\perp\thick\Ga$.
\item\label{U} There exists a commutative diagram below, where $(-)_\triangle$ is the triangulated hull.
\[ \xymatrix@R=5mm@!C=20mm{
	\rC^\Z(\Ga)\ar@{-}[d]^-\rsimeq\ar[r]^-{(1)}&\rC^\Z(\Ga)\ar@{-}[d]^-\rsimeq\ar[r]&\rC(\Ga)\ar@{-}[d]^-\rsimeq\\
	\per\A\ar[r]^-{-\lotimes_\A U}&\per\A\ar[r]&(\per\A/-\lotimes_\A U)_\triangle } \]
\end{enumerate}
\end{Thm}

We start the proof of \ref{B} with an observation on a resolution.
\begin{Lem}\label{cof}
\begin{enumerate}
\item If $X\in\per^\Z\!\Ga$ is concentrated in Adams degree $\geq0$, then $X\in\per^{\geq0}\!\Ga$.
\item There exist triangles
\[ \xymatrix@R=1mm{
	X_1\ar[r]&P_0\ar[r]&\Ga_0\\
	X_2\ar[r]&P_1\ar[r]\ar@{}[dd]|-{\vdots}&X_1\\
	\\
	P_n\ar[r]&P_{n-1}\ar[r]&X_{n-1} } \]
in $\per^\Z\!\Ga$ for some $n\geq0$ with $P_0=\Ga$ and $P_1,\ldots,P_n\in\add\{\Ga(-i)\mid 1\leq i\leq a\}$.
\end{enumerate}
\end{Lem}
\begin{proof}
	(1)  If $X$ is concentrated in Adams degree $\geq0$, then $\RHom_\Ga^\Z(P,X)=0$ for all $P\in\per^{<0}\!\Ga$. Then we have the assertion by the stable $t$-structure $\per^\Z\!\Ga=\per^{<0}\!\Ga\perp\per^{\geq0}\!\Ga$.
	
	(2)  One may take the first triangle as just $\Ga_{\geq1}\to\Ga\to\Ga_0$. We have to show that $X_1\in\per^{[1,a]}\!\Ga$. Since $X_1$ is concentrated in Adams degree $\geq1$, we have $X_1\in\per^{\geq1}\!\Ga$ by (1). It remains to show $X_1\in\per^{\leq a}\!\Ga$, in other words, $X_1^\ast\in\per^{\geq-a}\!\Ga^\op$. Dualizing the first triangle, we get
	\[ \xymatrix{ \Ga_0^\ast\ar[r]&\Ga\ar[r]&X_1^\ast }. \]
	Since $\Ga$ has Gorenstein parameter $a$, the first term is concentrated in Adams degree $-a$, in particular lies in $\per^{\geq-a}\!\Ga^\op$ by (1). We see by the triangle above that so is $X_1^\ast$, which finishes the proof.
\end{proof}

\begin{Lem}\label{star}
In the setting of \ref{setG}, the duality $(-)^\ast=\RHom_\Ga(-,\Ga)\colon\per^\Z\!\Ga\leftrightarrow\per^\Z\!\Ga^\op$ restricts to a duality $\thick\Ga_0\leftrightarrow\thick\Ga_0(a)$.
\end{Lem}
\begin{proof}
	By the left-right symmetry of our setting, we only have to show that $(-)^\ast$ takes $\thick\Ga_0$ to $\thick\Ga_0(a)$. Since $\Ga$ has Gorenstein parameter $a$, we have $\Ga_0^\ast$ is concentrated in (Adams) degree $-a$, thus lies in the localizing subcategory generated by $\Ga_0(a)\in\rD^\Z(\Ga^\op)$ by \ref{+}(\ref{GP2}). Now by our assumption we have that $\Ga_0^\ast$ is compact in $\rD^\Z(\Ga^\op)$, so it must be in the thick subcategory generated by $\Ga_0$.
\end{proof}	

This leads to the following important step.
\begin{Prop}\label{GEN}
The category $T=\bigoplus_{i=0}^{a-1}\Ga(i)$ generates $\rC^\Z(\Ga)$.
\end{Prop}
\begin{proof}
	The graded cluster category $\rC^\Z(\Ga)$ is certainly generated by $\add\{\Ga(i)\mid i\in\Z\}$. By \ref{star} the duality $(-)^\ast\colon\per^\Z\!\Ga\leftrightarrow\per^\Z\!\Ga^\op$ induces a duality $\rC^\Z(\Ga)\leftrightarrow\rC^\Z(\Ga^\op)$, so it is enough to prove that $\Ga(a)\in\thick T$. Noting that $\Ga_0=0$ in $\rC^\Z(\Ga)$, the triangles in \ref{cof}(2) show $\Ga\in\thick\{\Ga(-i)\mid 1\leq i\leq a\}$.
\end{proof}

The next lemma allows us to compute the morphism space in $\rC^\Z(\Ga)$. For brevity, we denote $\thick\Ga(\Z):=\thick\{\Ga(i)\mid i\in\Z\}$, $\thick\Ga(\geq\!0):=\thick\{\Ga(i)\mid i\geq0\}$, and so on.
\begin{Lem}\label{F}
The natural map $\Hom_{\rD^\Z(\Ga)}(X,Y)\to\Hom_{\rC^\Z(\Ga)}(X,Y)$ is an isomorphism for each $X\in\per^{\geq-a+1}\!\Ga$ and $Y\in\per^{\leq0}\!\Ga$.
\end{Lem}
\begin{proof}
	We first prove injectivity. Suppose that a morphism $X\to Y$ in $\per^\Z\!\Ga$ factors through an object $Z\in\thick\{\Ga_0(i)\mid i\in\Z\}$. We decompose $Z$ as $Z_{\geq-a+1}\to Z\to Z_{\leq-a}$ along the stable $t$-structure $\thick\Ga_0(\Z)=\thick\Ga_0(\leq a-1)\perp\thick\Ga_0(\geq a)$.
	\[ \xymatrix{
		X\ar[dr]\ar[rr]\ar@{-->}[d]&&Y\\
		Z_{\geq-a+1}\ar[r]&Z\ar[r]\ar[ur]&Z_{\leq-a} } \]
	Since $X\in\per^{\geq-a+1}\!\Ga$, we have $\Hom_{\rD^\Z(\Ga)}(X,Z_{\leq-a})=0$, so the morphism $X\to Z$ factors through $Z_{\geq-a+1}$. Now by $Z_{\geq-a+1}\in\thick\Ga_0(\leq a-1)$, $Y\in\per^{\leq0}\!\Ga=\thick\Ga(\geq\!0)$, and the assumption that $\RHom_\Ga(\Ga_0,\Ga)$ is concentrated in Adams degree $-a$, we deduce $\RHom_{\Ga}(Z_{\geq-a+1},Y)$ is concentrated in Adams degree $>0$, hence $\Hom_{\rD^\Z(\Ga)}(Z_{\geq-a+1},Y)=0$.
	
	We next prove surjectivity. Let the diagram
	\[ \xymatrix@R=3mm{
		&W&\\
		X\ar[ur]&&Y\ar[ul]_-s } \]
	with $Z:=\cocone s\in\thick\Ga_0(\Z)$ present a morphism in the Verdier quotient. Truncating $Z$ as above along the stable $t$-structure $\thick\Ga_0(\Z)=\thick\Ga_0(\leq a-1)\perp\thick\Ga_0(\geq a)$ as in the second row below, and noting that $\Hom_{\rD^\Z(\Ga)}(Z_{\geq-a+1},Y)=0$ as in the first step, we get an octahedron
	\[ \xymatrix{
		Z_{\geq-a+1}\ar[r]\ar@{=}[d]&W[-1]\ar[r]\ar[d]&W^\prime[-1]\ar[r]\ar[d]&Z_{\geq-a+1}[1]\ar@{=}[d]\\
		Z_{\geq-a+1}\ar[r]&Z\ar[r]\ar[d]& Z_{\leq-a}\ar[r]\ar[d]&Z_{\geq-a+1}[1]\\
		&Y\ar@{=}[r]&Y&\quad. } \]
	Then the original diagram is equivalent to the following, in which $\cocone s^\prime=Z_{\leq-a}$.
	\[ \xymatrix@R=3mm{
		&W^\prime&\\
		X\ar[ur]&&Y\ar[ul]_-{s^\prime} } \]
	By $\Hom_{\rD^\Z(\Ga)}(X,Z_{\leq-a})=0$ we deduce that the above morphism in $\rC^\Z(\Ga)$ comes from a one in $\rD^\Z(\Ga)$, proving surjectivity.
\end{proof}

Now we are ready to prove \ref{B}.
\begin{proof}[Proof of \ref{B}]
	Let $\C$ be the canonical (graded) dg enhancement of $\rC^\Z(\Ga)$, so that there is a commutative diagram with vertical equivalences
	\[ \xymatrix@R=5mm{
		\per^\Z\!\C\ar[r]\ar@{-}[d]^-\rsimeq&\per\C\ar@{-}[d]^-\rsimeq\\
		\rC^\Z(\Ga)\ar[r]&\rC(\Ga). } \]
	We know by \ref{GEN} that $\rC^\Z(\Ga)$ is generated by $T$, so by \cite[2.8]{haI}, the dg category $\C$ is $\Z$-graded derived Morita equivalent to the dg orbit category of $\REnd_\C^\Z(T)$ by $\RHom_\C^\Z(T,T(1))$.
	
	We have $\REnd^\Z_\C(T)\ysimeq\REnd^\Z_\Ga(T)$ by \ref{F}, which is isomorphic to $\A$, proving (\ref{equiv}).
	To show semi-orthogonality (\ref{SOD}), it is enough to prove $\Hom_{\rC^\Z(\Ga)}(\Ga(i),\Ga[j])=0$ for $1\leq i\leq a-1$ and $j\in\Z$, but this follows from \ref{F} and the fact that $\Ga$ is positively graded. Finally we verify (\ref{U}). The autoequivalence on $\per\A$ corresponding to the degree shift on $\rC^\Z(\Ga)$ is given by the $(\A,\A)$-bimodule $\RHom_\C^\Z(T,T(1))$, which is isomorphic again by \ref{F} to $\RHom_\Ga^\Z(T,T(1))=U$. This gives the commutativity of the left square. Also, the $\Z$-graded quasi-equivalence $\C\simeq\A/U$ (where the right-hand-side is the dg orbit category of $\A$ by $U$) shows that we have a commutative diagram
	\[ \xymatrix@R=5mm{
		\per^\Z\!\C\ar[r]\ar@{-}[d]^-\rsimeq&\per\C\ar@{-}[d]^-\rsimeq\\
		\per\A\ar[r]&(\per\A/-\lotimes_\A U)_\triangle. } \]
	We obtain the commutativity of the right square from the above two diagrams.
\end{proof}

We end this section with some most fundamental examples.
\begin{Ex}
(1)  Let $\Ga$ is a polynomial ring $S=k[x_0,x_1,\ldots,x_n]$ regarded as a graded dg algebra concentrated in cohomological degree $0$ and $\deg x_i=1$ for the Adams grading. 
Then we have $\per^\Z\!\Ga=\Db(\mod^\Z\!S)$, and $\thick\{\Ga_0(i)\mid i\in\Z\}=\Db(\fl^\Z\!S)$, the derived category of finite length graded $S$-modules. Therefore we get $\rC^\Z(\Ga)=\Db(\mod^\Z\!S)/\Db(\fl^\Z\!S)=\Db(\mod^\Z\!S/\fl^\Z\!S)$, which is equivalent to $\Db(\coh\PP^n)$ by Serre's theorem. We see that \ref{B} in this case is Beilinson's theorem. 

(2)  Let $A$ be a smooth, proper, and connective dg algebra, and let $\Pi=\rT^\rL_{A}(\RHom_{A^e}(A,A^e)[d])$ be the $(d+1)$-Calabi-Yau completion. We regard $\Pi$ as a graded dg algebra with the tensor grading. Then $\Pi$ is graded $(d+1)$-Calabi-Yau of Gorenstein parameter $1$ by \cite{Ke11}, see also \ref{MM}(\ref{CYdgalg}).
In this case the graded cluster category is $\rC^\Z(\Pi)=\per^\Z\!\Pi/\thick A(\Z)$, thus it is just a covering of the usual cluster category. The diagram in \ref{B} becomes
\[ \xymatrix{
	\rC^\Z(\Pi)\ar[r]\ar[d]^-\rsimeq&\rC(\Pi)\ar[d]^-\rsimeq\\
	\per A\ar[r]&(\per A/-\lotimes_ADA[-d])_\triangle. } \]
\end{Ex}

We refer to \ref{ExB} for more examples where we compare cluster categories of several different dg algebras.

\end{appendix}
\end{document}